\documentclass[12pt,a4paper]{amsart}
\usepackage{a4wide,mathtools,amssymb,amsthm,url,color}
\allowdisplaybreaks

\newcommand{\N}{\mathbb{N}}
\newcommand{\Z}{\mathbb{Z}}
\newcommand{\R}{\mathbb{R}}
\newcommand{\C}{\mathbb{C}}
\newcommand{\RT}{\mathsf{RT}}
\newcommand{\RM}{\mathsf{RM}}
\newcommand{\Pol}{\mathsf{P}}
\newcommand{\To}{\longrightarrow}
\newcommand{\eps}{\varepsilon}
\newcommand{\om}{\Omega}
\newcommand{\ga}{\Gamma}

\newcommand{\calD}{\mathcal{D}}
\newcommand{\calM}{\mathcal{M}}
\newcommand{\calV}{\mathcal{V}}
\newcommand{\calL}{\mathcal{L}}
\newcommand{\calQ}{\mathcal{Q}}
\newcommand{\calB}{\mathcal{B}}
\newcommand{\Sb}{\mathbb{S}}
\newcommand{\Tb}{\mathbb{T}}
\newcommand{\pw}{\mathsf{pw}}
\newcommand{\red}{\mathsf{red}}
\newcommand{\ela}{\mathsf{ela}}
\newcommand{\bih}{\mathsf{bih}}
\newcommand{\rhm}{\mathsf{Rhm}}
\newcommand{\rhmp}{\rhm}

\newcommand{\bihone}{\bih,1}
\newcommand{\bihtwo}{\bih,2}
\newcommand{\bihonep}{\bihone}

\newcommand{\bihtwop}{\bihtwo}

\newcommand{\elap}{\ela}

\DeclareMathOperator{\id}{Id}

\DeclareMathOperator{\dev}{dev}
\DeclareMathOperator{\tr}{tr}

\DeclareMathOperator{\sym}{sym}
\DeclareMathOperator{\skw}{skw}
\renewcommand{\ker}{\operatorname{ker}}
\DeclareMathOperator{\ran}{ran}
\DeclareMathOperator{\dom}{dom}
\DeclareMathOperator{\ind}{ind}
\DeclareMathOperator{\spn}{spn}
\DeclareMathOperator{\lin}{lin}
\DeclareMathOperator{\supp}{supp}
\DeclareMathOperator{\intd}{d}

\DeclareMathOperator{\p}{\partial}
\DeclareMathOperator{\dive}{div}
\DeclareMathOperator{\rdive}{\mathring{\dive}}
\DeclareMathOperator{\curl}{curl}
\DeclareMathOperator{\rcurl}{\mathring{\curl}}
\DeclareMathOperator{\grad}{grad}
\DeclareMathOperator{\rgrad}{\mathring{\grad}}
\DeclareMathOperator{\Grad}{Grad}

\DeclareMathOperator{\devGrad}{devGrad}
\DeclareMathOperator{\rdevGrad}{\mathring{\devGrad}}
\DeclareMathOperator{\symGrad}{symGrad}
\DeclareMathOperator{\rsymGrad}{\mathring{\symGrad}}
\DeclareMathOperator{\Dive}{Div}
\DeclareMathOperator{\rDive}{\mathring{\Dive}}
\DeclareMathOperator{\Curl}{Curl}
\DeclareMathOperator{\rCurl}{\mathring{\Curl}}
\DeclareMathOperator{\symCurl}{symCurl}
\DeclareMathOperator{\rsymCurl}{\mathring{\symCurl}}
\newcommand{\DS}{\Dive_{\Sb}}
\newcommand{\rDS}{\rDive_{\Sb}}
\newcommand{\DT}{\Dive_{\Tb}}
\newcommand{\rDT}{\rDive_{\Tb}}
\newcommand{\CS}{\Curl_{\Sb}}
\newcommand{\rCS}{\rCurl_{\Sb}}
\newcommand{\symCT}{\symCurl_{\Tb}}
\newcommand{\rsymCT}{\rsymCurl_{\Tb}}
\newcommand{\Gg}{\operatorname{Gradgrad}}
\newcommand{\rGg}{\mathring{\Gg}}
\newcommand{\dD}{\operatorname{divDiv}}
\newcommand{\rdD}{\mathring{\dD}}
\newcommand{\dDS}{\dD_{\Sb}}
\newcommand{\rdDS}{\rdD_{\Sb}}
\newcommand{\CC}{\operatorname{CurlCurl}}
\newcommand{\CCt}{\CC^{\top}}
\newcommand{\CCtS}{\CCt_{\Sb}}
\newcommand{\rCCt}{\mathring{\CC}{}^{\top}}
\newcommand{\rCCtS}{\rCCt_{\Sb}}

\newcommand{\cc}{\mathsf{cc}}

\newcommand{\Lt}{L^{2}}
\newcommand{\Ltt}{L^{2,3}}
\newcommand{\Lttt}{L^{2,3\times3}}
\newcommand{\LtttS}{\Lttt_{\Sb}}
\newcommand{\LtttT}{\Lttt_{\Tb}}
\newcommand{\Ltom}{\Lt(\om)}
\newcommand{\Lttom}{\Ltt(\om)}
\newcommand{\Ltttom}{\Lttt(\om)}
\newcommand{\LtttSom}{\LtttS(\om)}
\newcommand{\LtttTom}{\LtttT(\om)}
\newcommand{\Ci}{C^{\infty}}
\newcommand{\Cit}{C^{\infty,3}}
\newcommand{\Citt}{C^{\infty,3\times3}}
\newcommand{\Cic}{\Ci_{c}}
\newcommand{\Cict}{\Cit_{c}}
\newcommand{\Cictt}{\Citt_{c}}
\newcommand{\CicttS}{\Citt_{c,\Sb}}
\newcommand{\CicttT}{\Citt_{c,\Tb}}
\newcommand{\Ciom}{\Ci(\om)}
\newcommand{\Citom}{\Cit(\om)}
\newcommand{\Cittom}{\Citt(\om)}
\newcommand{\Cicom}{\Cic(\om)}
\newcommand{\Cictom}{\Cict(\om)}

\newcommand{\CicttSom}{\CicttS(\om)}
\newcommand{\CicttTom}{\CicttT(\om)}
\newcommand{\harm}{\mathcal{H}}
\newcommand{\harmD}{\harm_{D}}
\newcommand{\harmN}{\harm_{N}}
\newcommand{\harmDS}{\harm_{D,\Sb}}
\newcommand{\harmDT}{\harm_{D,\Tb}}
\newcommand{\harmNS}{\harm_{N,\Sb}}
\newcommand{\harmNT}{\harm_{N,\Tb}}
\newcommand{\harmDom}{\harmD^{\rhm}(\om)}
\newcommand{\harmNom}{\harmN^{\rhm}(\om)}
\newcommand{\harmbihoneDom}{\harmDS^{\bihone}(\om)}

\newcommand{\harmbihoneNom}{\harmNT^{\bihone}(\om)}
\newcommand{\harmelaDom}{\harmDS^{\ela}(\om)}

\newcommand{\harmelaNom}{\harmNS^{\ela}(\om)}
\newcommand{\harmbihtwoDom}{\harmDT^{\bihtwo}(\om)}

\newcommand{\harmbihtwoNom}{\harmNS^{\bihtwo}(\om)}

\newtheorem{theorem}{Theorem}[section]
\newtheorem{corollary}[theorem]{Corollary}
\newtheorem{lemma}[theorem]{Lemma}
\newtheorem{proposition}[theorem]{Proposition}
\newtheorem{definition}[theorem]{Definition}
\newtheorem{remark}[theorem]{Remark}
\newtheorem{assumption}[theorem]{Assumption}
\newtheorem{problem}[theorem]{Problem}

\newcommand{\norm}[1]{|#1|}
\newcommand{\bnorm}[1]{\big|#1\big|}

\newcommand{\scp}[2]{\langle#1,#2\rangle}
\newcommand{\bscp}[2]{\big\langle#1,#2\big\rangle}
\newcommand{\Bscp}[2]{\Big\langle#1,#2\Big\rangle}

\newcommand{\pz}{biharmonic }
\newcommand{\PZ}{Biharmonic }

\renewcommand{\subset}{\subseteq}

\numberwithin{equation}{section}

\makeatletter
\def\@tocline#1#2#3#4#5#6#7{\relax
 \ifnum #1>\c@tocdepth % then omit
 \else
  \par \addpenalty\@secpenalty\addvspace{#2}%
  \begingroup \hyphenpenalty\@M
  \@ifempty{#4}{%
   \@tempdima\csname r@tocindent\number#1\endcsname\relax
  }{%
   \@tempdima#4\relax
  }%
  \parindent\z@ \leftskip#3\relax \advance\leftskip\@tempdima\relax
  \rightskip\@pnumwidth plus4em \parfillskip-\@pnumwidth
  #5\leavevmode\hskip-\@tempdima
   \ifcase #1
    \or\or \hskip 1em \or \hskip 2em \else \hskip 3em \fi%
   #6\nobreak\relax
  \hfill\hbox to\@pnumwidth{\@tocpagenum{#7}}\par% <---- \dotfill -> \hfill
  \nobreak
  \endgroup
 \fi}
\makeatother

\begin{document}

\title[Index of Mixed Order Dirac-Type Operators]
{The Index of Some Mixed Order Dirac-Type Operators\\
and Generalised Dirichlet-Neumann Tensor Fields}
\author{Dirk Pauly}
\address{Fakult\"at f\"ur Mathematik,
Universit\"at Duisburg-Essen, Campus Essen, Germany}
\email[Dirk Pauly]{dirk.pauly@uni-due.de}
\author{Marcus Waurick}
\address{Department of Mathematics and Statistics, University of Strathclyde,
Livingstone Tower, 26 Richmond Street, Glasgow G1 1XH, Scotland}
\email[Marcus Waurick]{marcus.waurick@strath.ac.uk}

\date{\today; {\it Corresponding Author}: Marcus Waurick}
%\dedicatory{Dedicated to Captain Rainer Picard}
%\dedicatory{Dedicated to the Captain}
\thanks{{\it Achnowledgements}: We cordially thank Walter Zulehner, JKU Linz, 
for many inspiring and fruitful discussions and kind help related 
to the Poincar\'e maps necessary for the construction of Neumann fields.}

\begin{abstract} 
We revisit a construction principle of Fredholm operators using Hilbert complexes 
of densely defined, closed linear operators and apply this to particular choices of differential operators. 
The resulting index is then computed with the help of explicitly describing 
the dimension of the cohomology groups of generalised (`harmonic') Dirichlet and Neumann tensor fields. 
The main results of this contribution are the computation of the indices of Dirac-type operators associated 
to the elasticity complex and the newly found \pz complex,
relevant for the \pz equation, elasticity, and for the theory of general relativity. 
The differential operators are of mixed order and cannot be seen 
as leading order type with relatively compact perturbation. 
As a by-product we present a comprehensive description 
of the underlying generalised Dirichlet-Neumann vector and tensor fields
defining the respective cohomology groups, including an explicit construction of bases in terms of topological invariants, 
which are of both analytical and numerical interest. 
Though being defined by certain projection mechanisms, 
we shall present a way of computing these basis functions by solving certain PDEs given in variational form. 
For all of this we rephrase core arguments in the work of Rainer Picard \cite{P1982} 
applied to the de Rham complex and use them as a blueprint for the more involved cases presented here.
In passing, we also provide new vector-analytical estimates of generalised Poincar\'e--Friedrichs type 
useful for elasticity or the theory of general relativity.
\end{abstract}

\maketitle
\newpage
\tableofcontents
\newpage

\section{Introduction}
\label{sec:int}

It is one of the greatest mathematical achievements of the twentieth century 
to relate the analytic notion of the Fredholm index for operators defined on Hilbert spaces 
to particular elliptic operators and their corresponding geometric properties 
of the underlying compact manifold the operators are defined on. 
Here, a closed, densely defined, linear operator 
$\calD\colon\dom \calD\subset \mathcal{K}_{0} \to \mathcal{K}_{1}$ 
between Hilbert spaces $\mathcal{K}_{0}$ and $\mathcal{K}_{1}$ is called a \emph{Fredholm operator}, 
if the range, $\ran \calD\subset\mathcal{K}_{1} $, is closed and both the kernel, 
$\ker \calD$, and the co-kernel, $(\ran \calD)^\bot$, are finite-dimensional. 
In this case, the \emph{index of $\calD$}, $\ind\calD$, is given by
$$\ind\calD= \dim \ker\calD - \dim (\ran \calD)^\bot.$$
We refer to the concluding parts in \cite[Chapter 3]{GW2016} for a brief round up 
and some standard references to the theory of unbounded Fredholm operators in Hilbert spaces. 
Generally spoken it is often very difficult -- if not impossible -- 
to compute either $\dim \ker\calD $ or $\dim (\ran \calD)^\bot$ directly. 
However, due to invariance under homotopies and compact perturbations, 
it is sometimes possible to have a better understanding of $\ind\calD$ instead.

Indeed, one of the corner stones of results hinted at above is the celebrated 
Atijah-Singer index theorem, see e.g.~\cite{M2013}, 
where the (Fredholm) index for some elliptic operators defined on a manifold 
can be represented solely in terms of the topological properties of this manifold.
The methods of proof led to the invention of $K$-theory, which has evolved ever since and is an active field of research. 
Albeit being a breakthrough in mathematics, $K$-theory is a rather difficult tool to work with when it comes 
to explicitly compute the index for particular examples. 
In any case there is a need to provide many examples, where it is possible to obtain an index formula, 
which is as explicit as possible. In fact, the Fredholm index for a perturbed Dirac operator represents physical quantities, 
see the concluding example in \cite{C1978} and the references therein.
The Witten index, a generalised version of the Fredholm index, is interesting for both physics and mathematics. 
Indeed, it has been shown that in particular situations the Witten index corresponds 
to the electromagnetic spin of a particle as well as to the spectral flow of an operator family, 
see the seminal paper \cite{GL2011}.

The results in \cite{GL2011} are based on -- among other things -- a deeper understanding 
of the one-dimensional situation of \cite{C1978}, which addresses the Fredholm case. 
The higher-dimensional cases treated in \cite{C1978} (with an index formula properly 
justified in \cite{BS1978}) were generalised in \cite{GW2016}. 
The transition from the Fredholm situation to the Witten index 
has been performed in \cite[Chapter 14]{GW2016}. 
Again, a thorough understanding of the Fredholm case has led to further examples for the Witten index, 
which in turn might prove useful for both mathematics and physics.

The main contribution of the present study is to enlarge the variety of examples, 
where it is possible to explicitly compute the Fredholm index in terms of the topological properties 
of the underlying (bounded) domain $\om\subset \R^3$ the differential (Fredholm) operators are defined on. 
The list of examples treated here is even more particular as it is possible 
to compute not only the index but also the dimension of the kernel 
and the co-dimension of the range in terms of topological invariants.

Moreover, this article is concerned with the explicit computation of the Fredholm index 
if a differential operator is `apparently' of mixed order. 
We shall establish a collection of theorems like the following:

\begin{theorem}
\label{thm:main} 
Let $\om\subset\R^3$ be open, bounded with strong\footnote{The boundary
of a strong Lipschitz domain is locally a graph of some Lipschitz function.} 
Lipschitz boundary. Then there exists a subspace $\calV\subset\LtttTom\times\Ltom$ such that
$$\calD
:=\begin{pmatrix}\Dive & 0\\ \symCurl &\Gg\end{pmatrix}:
\calV\subset\LtttTom\times\Ltom\to\Lttom\times\LtttSom$$
and $\calD^{*}$ are densely defined and closed Fredholm operators, where $\LtttTom$ and $\LtttSom$ 
denote the sets of trace free and symmetric $3\times 3$ matrices with entries in $\Ltom$, respectively. Moreover, 
$$\ind\calD=4(p-m-n+1),\qquad
\ind\calD^{*}=-\ind\calD$$
where $n$ is the number of connected components of $\om$,
$m$ is the number of connected components of its complement $\R^3\setminus\overline{\om}$, 
and $p$ is the number of handles (see Section \ref{app:sec:DirNeuF}).
\end{theorem}

A closer inspection of the operator $\mathcal{D}$ also shows the following estimate; 
see also Corollary \ref{cor:FP-tb-bih}. 
Note that the subspace $\mathcal{W}$ asserted to exist 
in the following result -- and this is the catch of the corollary -- is constructed explicitly by providing a basis, 
see Section \ref{app:sec:NeuTFbih}.

\begin{corollary} 
There exists a finite-dimensional subspace $\mathcal{W}\subset \mathcal{V}$ and $c>0$ 
such that for all $(T,u)\in\mathcal{V}\cap \mathcal{W}^{\bot_{L^{2}}}$ we have
$$c\,\bnorm{(T,u)}_{\LtttTom\times\Ltom}
\leq\norm{\Gg u}_{\LtttSom}
+\norm{\Dive T}_{\Lttom}
+\norm{\symCurl T}_{\LtttSom}.$$
%\begin{align*}
%\bnorm{(T,u)}_{\LtttTom\times\Ltom}
%&\leq c\big(\norm{\Gg u}_{\LtttSom}^{2}\\
%&\qquad\qquad\qquad+\norm{\Dive T}_{\Lttom}^{2}+\norm{\symCurl T}_{\LtttSom}^{2}\big)^{1/2}.
%\end{align*}
\end{corollary}

In the course of the manuscript, we shall, too, describe the subspace 
$\calV=\dom\calD$ explicitly, see Theorem \ref{thm:idxPZC} and Remark \ref{thm:idxPZC:rem}. 
A refined notation will indicate (full) natural boundary conditions by $\mathring{\,\,\,\,}$ 
and algebraic properties of the tensor fields belonging to the domain of definition of the respective operators
by $\Sb$ and $\Tb$ (symmetric and trace free), e.g., the aforementioned operators read
\begin{align*}
\calD=\calD^{\bihonep}&:=\begin{pmatrix}\rDT&0\\\symCT&\rGg\end{pmatrix},
&
(\calD^{\bihonep})^{*}&=\begin{pmatrix}-\devGrad&\rCS\\0&\dDS\end{pmatrix}.
\end{align*}
These operators are related to the (primal and dual) first \pz complex, 
also called $\Gg$ or $\dD$ complex, i.e., 
\begin{align*}
\{0\}\xrightarrow{\iota_{\{0\}}}
\Ltom\xrightarrow{\rGg}
\LtttSom&\xrightarrow{\rCS}
\LtttTom\xrightarrow{\rDT}
\Lttom\xrightarrow{\pi_{\RT_{\pw}}}\RT_{\pw},\\
\{0\}\xleftarrow{\pi_{\{0\}}}
\Ltom\xleftarrow{\dDS}
\LtttSom&\xleftarrow{\symCT}
\LtttTom\xleftarrow{-\devGrad}
\Lttom\xleftarrow{\iota_{\RT_{\pw}}}\RT_{\pw},
\end{align*}
relevant for the \pz equation, elasticity, and in the theory of general relativity. 
Here and in the following $\iota_V$ and $\pi_V$ denote the canonical embedding 
from a closed subspace $V$ of a Hilbert space $H$ into $H$ and 
the orthogonal projection from $H$ onto $V$; the space of piecewise Raviart--Thomas vector-fields, 
$\RT_{\pw}$, is defined in \eqref{eq:def_RT_pw}.

We discuss another example, which is based on the second \pz complex 
where the boundary conditions are interchanged, i.e., 
\begin{align*}
\{0\}\xrightarrow{\iota_{\{0\}}}
\Lttom\xrightarrow{\rdevGrad}
\LtttTom&\xrightarrow{\rsymCT}
\LtttSom\xrightarrow{\rdDS}
\Ltom\xrightarrow{\pi_{\Pol^{1}_{\pw}}}\Pol^{1}_{\pw},\\
\{0\}\xleftarrow{\pi_{\{0\}}}
\Lttom\xleftarrow{-\DT}
\LtttTom&\xleftarrow{\CS}
\LtttSom\xleftarrow{\Gg}
\Ltom\xleftarrow{\iota_{\Pol^{1}_{\pw}}}\Pol^{1}_{\pw},
\end{align*}
leading to the operators (for the space of piecewise first order polynomials, 
$\Pol^{1}_{\pw}$, we refer to \eqref{eq:P1_pw})
\begin{align*}
\calD^{\bihtwop}&:=\begin{pmatrix}\rdDS&0\\\CS&\rdevGrad\end{pmatrix},
&
(\calD^{\bihtwop})^{*}&=\begin{pmatrix}\Gg&\rsymCT\\0&-\DT\end{pmatrix}.
\end{align*}
The corresponding index results can be found in Theorem \ref{thm:idxPZC2} and Remark \ref{thm:idxPZC2:rem}. 

Finally, we address the elasticity complex, also called $\CC$ complex, i.e., 
(the space of piecewise rigid motions, $\RM_{\pw}$, is defined in \eqref{eq:RM_pw})
\begin{align*}
\{0\}\xrightarrow{\iota_{\{0\}}}
\Lttom\xrightarrow{\rsymGrad}
\LtttSom&\xrightarrow{\rCCtS}
\LtttSom\xrightarrow{\rDS}
\Lttom\xrightarrow{\pi_{\RM_{\pw}}}\RM_{\pw},\\
\{0\}\xleftarrow{\pi_{\{0\}}}
\Lttom\xleftarrow{-\DS}
\LtttSom&\xleftarrow{\CCtS}
\LtttSom\xleftarrow{-\symGrad}
\Lttom\xleftarrow{\iota_{\RM_{\pw}}}\RM_{\pw}.
\end{align*}
Here, we shall discuss the Fredholm index of the operators 
\begin{align*}
\calD^{\elap}&:=\begin{pmatrix}\rDS&0\\\CCtS&\rsymGrad\end{pmatrix},
&
(\calD^{\elap})^{*}&=\begin{pmatrix}-\symGrad&\rCCtS\\0&-\DS\end{pmatrix}.
\end{align*}
The solution to the corresponding index problem is provided in Theorem \ref{thm:idxelaC} and Remark \ref{thm:idxelaC:rem}. 
Note that in a distributional setting results concerning the computation 
of the dimension of the generalised Neumann fields have been obtained in \cite{CCGK2007}, 
where a variational setting is preferred.

Here and throughout this paper, we denote by $\grad$, $\curl$, and $\dive$ 
the classical operators from vector analysis.
Moreover, $\Grad$ acts componentwise as $\grad^{\top}$ mapping vector fields to tensor fields.
$\Curl$ and $\Dive$ act row-wise as $\curl^{\top}$ and $\dive$ mapping tensor fields to tensor and vector fields, respectively. 
$L^{2}$-spaces with $k$ components (or $k\times k$-many components) 
are denoted by $L^{2,k}$ (or $L^{2,k\times k}$). 
A similar notation is used for $C^\infty$ and similar sets.

The approach to compute the index in situations as in Theorem \ref{thm:main} is rooted in 
a construction principle for Fredholm operators provided in \cite{BL1992}. 
The fundamental observation given in \cite{BL1992} is that it is possible 
to construct a Fredholm operator with the help of Hilbert complexes 
of densely defined and closed (possibly unbounded) linear operators, i.e,
\begin{align*}
\cdots\xrightarrow{\cdots}
H_{0}\xrightarrow{A_{0}}
H_{1}&\xrightarrow{A_{1}}
H_{2}\xrightarrow{A_{2}}
H_{3}
\xrightarrow{\cdots}\cdots,\\
\cdots\xleftarrow{\cdots}
H_{0}\xleftarrow{A_{0}^{*}}
H_{1}&\xleftarrow{A_{1}^{*}}
H_{2}\xleftarrow{A_{2}^{*}}
H_{3}
\xleftarrow{\cdots}\cdots.
\end{align*}
More precisely, if $A_{0}$, $A_{1}$, and $A_{2}$ are densely defined, closed linear operators defined 
on suitable Hilbert spaces $H_{l}$ such that 
$$\ran A_{0}\subset\ker A_{1},\qquad
\ran A_{1}\subset\ker A_{2},$$ 
then the block matrix operator
\begin{equation}
\label{eq:calD}
\calD:=
\begin{pmatrix}A_{2} & 0\\ A_{1}^{*} & A_{0}\end{pmatrix}
\end{equation}
with its natural domain of definition is closed and densely defined. 
It is Fredholm, if the ranges $\ran A_{0}$, $\ran A_{1}$, and $\ran A_{2}$ are closed
and if both kernels
$$N_{0}:=\ker A_{0},\qquad
N_{2,*}:=\ker A_{2}^{*}$$ 
and both cohomology groups 
$$K_{1}:=\ker A_{1}\cap\ker A_{0}^{*},\qquad
K_{2}:=\ker A_{2}\cap\ker A_{1}^{*}$$ 
are finite-dimensional. In this case, its index is then given by
\begin{align}
\label{index1}
\ind\calD
=\dim N_{0}-\dim K_{1}+\dim K_{2}-\dim N_{2,*},
\end{align}
cf.~Theorem \ref{thm:index}. For its adjoint, which is then Fredholm as well, we simply have
$$\calD^{*}:=
\begin{pmatrix}A_{2}^{*} & A_{1}\\ 0& A_{0}^{*}\end{pmatrix},\qquad
\ind\calD^{*}=-\ind\calD.$$

In a first application of this observation presented in this article, 
we look at the classical de Rham complex
\begin{align*}
%\label{derhamcompl}
%\begin{aligned}
\{0\}\xrightarrow{A_{-1}=\iota_{\{0\}}}
\Ltom\xrightarrow{A_{0}=\rgrad}
\Lttom&\xrightarrow{A_{1}=\rcurl}
\Lttom\xrightarrow{A_{2}=\rdive}
\Ltom\xrightarrow{A_{3}=\pi_{\R_{\pw}}}\R_{\pw},\\
\{0\}\xleftarrow{A_{-1}^{*}=\pi_{\{0\}}}
\Ltom\xleftarrow{A_{0}^{*}=-\dive}
\Lttom&\xleftarrow{A_{1}^{*}=\curl}
\Lttom\xleftarrow{A_{2}^{*}=-\grad}
\Ltom\xleftarrow{A_{3}^{*}=\iota_{\R_{\pw}}}\R_{\pw},
%\end{aligned}
\end{align*} 
where again the super index $\mathring{\;\;}$ signifies homogeneous Dirichlet boundary conditions, 
see Theorem \ref{thm:idx} and $\R_{\pw}$ denotes the space of piecewise constants on $\om$, see \eqref{eq:R_pw}.
By \eqref{index1} in order to compute the index it is necessary to calculate the dimension 
of the cohomology groups, i.e., the dimension of the harmonic Dirichlet and Neumann fields
\begin{align*}
\harmDom:=K_{1}&=
\ker(\rcurl)\cap\ker(\dive),\\
\harmNom:=K_{2}&=
\ker(\rdive)\cap\ker(\curl),
\end{align*}
respectively.
In \cite{P1982}, this has been done by Picard. 
As it turns out these dimensions are related to topological properties of the underlying domain 
the differential operators are defined on, that is, 
$$\dim\harmDom=m-1,\qquad
\dim\harmNom=p,$$
see Theorem \ref{thm:DNF}.
In consequence, it is possible to compute the indices for the block de Rham operators 
$$\calD^{\rhmp}:=
\begin{pmatrix}\rdive & 0\\\curl &\rgrad\end{pmatrix},\qquad
(\calD^{\rhmp})^{*}:=
\begin{pmatrix}-\grad & \rcurl\\0 &-\dive\end{pmatrix}$$ 
by \eqref{index1} in terms of $m$, $p$, and $n$, i.e.,
$$\ind\calD^{\rhmp}=p-m-n+1,\qquad
\ind(\calD^{\rhmp})^{*}=-\ind\calD^{\rhmp},$$
see Theorem \ref{thm:idx}. 
It is noteworthy that this index theorem provides an index theorem for the Dirac operator 
on open manifolds with boundary endowed with a particular boundary condition,
see \cite{PTW2017} and Section \ref{sec:dirac} below.
  
The proofs of the index theorems discussed here combine the structural viewpoint outlined by \cite{BL1992} 
and ideas taken from the explicit computation of the dimension of the cohomology groups in \cite{P1982}. 
More precisely, we shall rephrase the proofs in \cite{P1982} and use these reformulations 
as a blueprint for the proofs for other complexes. We emphasise that the construction 
of the generalised Neumann fields is based on subtle interactions of matrix algebra 
and differential operators (see Lemma \ref{PZformulalem}) 
and a suitable application of so-called Poincar\'e maps yielding (for instance) 
a representation of vector fields by curve integrals over tensor fields, see e.g.~Lemma \ref{devGradHSlemma}.
The foundation for all of this to be applicable, however, is the newly found \pz complex, see \cite{PZ2016,PZ2020a},
and the more familiar elasticity complex, see \cite{PZ2020b,PZ2020c}.
In \cite{PZ2016,PZ2020a} the crucial properties and compact embedding results have been found for the \pz Hilbert complex 
underlying the computation of the index in Theorem \ref{thm:main}. 
In \cite{PZ2020b,PZ2020c} the corresponding results are presented for the elasticity complex. 
These results also stress that the mixed order differential operators discussed here
\emph{cannot} be viewed as a leading order term subject to a relatively compact perturbation. 

Before we come to a more in depth description of the structure of the paper, 
we emphasise the importance of a deeper understanding of Hilbert complexes 
for index theory and other areas of partial differential equations.

Next to \cite{BL1992} (and others), the notion of Hilbert complexes in connection with (Fredholm) index theory 
has also been addressed in \cite{Schulze19} and references therein. 
The work in \cite{Schulze19} is particularly interesting as the authors address manifolds with boundary. 
The focus is on characterising the Fredholm property (i.e., the finiteness of the cohomology groups) 
for certain complexes with boundary in terms of the principal symbol. 
Here, the Fredholm property of the Hilbert complex (i.e., with the terminology of \cite{Schulze19}, 
that a Hilbert complex is a Fredholm complex) follows 
from suitable compactness criteria (see e.g.~\cite{W1974,PZ2020a}) 
and all the dimensions of the cohomology groups and not just the index of the complex are addressed here explicitly.

An understanding of Hilbert complexes in connection with partial differential equations involving 
the classical vector analytic operators $\dive$, $\curl$, and $\grad$ led to \cite{P1984a}, 
where the kernel of the classical Maxwell operator is written by means of other differential operators. 
The resulting Picard's extended Maxwell system is useful for numerical studies \cite{TV2007} 
as well as for the study of the low frequency asymptotics of Maxwell's equations \cite{P1984a}. 
More involved low frequency asymptotics for Maxwell's equations can be found in the series
\cite{P2006a,P2007a,P2008a,P2008b,P2012a}
based on the series \cite{WW1992a,WW1994a,WW1997a,WW1997b}
for the reduced scalar and elasticity wave equations.
The connections of the extended Maxwell system and the Dirac operator are drawn in \cite{PTW2017} 
and shortly commented in this manuscript below.

Rather recently, the notion of Hilbert complexes 
(reusing the idea of writing the kernel of differential operators by means of other operators) 
has found applications in the context of homogenisation theory of partial differential equations. 
More precisely, it was possible to derive a certain operator-theoretic version 
of the so-called $\dive$-$\curl$ lemma (see~\cite{Tartar2009,Murat1978}), 
which implied a whole family of $\dive$-$\curl$ lemma-type results, see \cite{W2018,P2019b}.

Furthermore, the abstract $\dive$-$\curl$ result together with theory 
from Hilbert complexes are used to define and study the notion of nonlocal $H$-convergence, \cite{W2018b}. 
The applications presented in \cite{W2018b,W2019} as well as in \cite{N2019} 
use the assumption of exactness of the Hilbert complex, that is, 
triviality of certain cohomology groups. It is one corollary of the present study 
to describe the topological properties of the domains the differential operators 
are defined on to yield exact Hilbert complexes making the theory of nonlocal $H$-convergence applicable, 
see also \cite[Section 8]{W2018b}. This then results in new homogenisation 
and compactness theorems for nonlocal homogenisation problems. 
We postpone further results in this direction to future studies.

We shortly outline the plan of the paper next.
The main results, that is, the dimensions of the cohomology groups 
and the indices of the operators involved, are summarised in Section \ref{sec:Overview}.
In Section \ref{sec:3}, we briefly recall the notion of Hilbert complexes of densely defined and closed linear operators. 
Also, we provide a small introduction to the construction principle for Fredholm operators provided in \cite{BL1992}. 
As we slightly deviate from the approach presented there we recall some of the proofs for convenience of the reader. 
As an addendum to Section \ref{sec:3}, we provide an abstract set of Poincar\'e--Friedrichs inequalities 
in Section \ref{sec:3-addstuff} and outline an abstract perspective to variable coefficents in Section \ref{sec:4-addstuff}.
In order to have a first non-trivial yet rather elementary example at hand, we present 
the so-called Picard's extended Maxwell system in Section \ref{sec:derham}. 
This sets the stage for the index theorem for the Dirac operator provided in Section \ref{sec:dirac}.
In Section \ref{sec:pz}, we recall the first \pz complex and provide a more explicit formulation 
of Theorem \ref{thm:main} (see Theorem \ref{thm:idxPZC}). 
Similar results will be presented in Section \ref{sec:pz2} for the second \pz complex
and in Section \ref{sec:ela} for the elasticity complex.
Section \ref{app:sec:DirNeuF} is concerned with the topological setting introduced in \cite{P1982} 
forming our main assumption on $\om$. The Sections \ref{app:sec:DirF} 
and \ref{app:sec:NeuF} address the computation of bases and hence the dimensions 
of the generalised Dirichlet and Neumann vector and tensor fields
for the different complexes, respectively, and thus concluding the proofs of our main results. 
In passing, we also provide partial differential equations 
whose unique solutions will correspond to the basis functions under consideration. 
This is particularly important for numerically computing these basis functions. 
Amongst these PDEs we recover the one in \cite{CCGK2007}, when the generalised Neumann fields 
for the elasticity complex are concerned (see in particular Remark \ref{rem:charPDEela}; 
the regularity assumptions on $\om$ are the same here). 
In Section \ref{sec:con} we provide a small conclusion.

Note that unlike to many research topics in the analysis of partial differential equations (and other topics), 
we shall use $\om$ being `open' and a `domain' as synonymous terms. 
In particular, we shall not imply $\om$ to satisfy any connectivity properties, 
when calling $\om$ a domain.

\section{A Brief Overview of the Main Results}
\label{sec:Overview}

We employ the notations and assumptions of Section \ref{sec:int}. 
In particular, we shall assume that $\om\subset \R^3$ is a bounded, strong Lipschitz domain. 
The number of connected components of $\om$ is $n$, 
the number of connected components of $\R^3\setminus\overline{\om}$ is $m$, 
and $p$ denotes the number of handles (see Assumption \ref{ass:curvesandsurfaces} below).
We introduce the cohomology groups 
$$K_{1}=\harmD^{\cdots}(\om),\qquad
K_{2}=\harmN^{\cdots}(\om),$$
i.e., the Dirichlet and Neumann fields
\begin{align*}
\harmDom
&=\ker(\rcurl)\cap\ker(\dive),
&
\harmNom
&=\ker(\rdive)\cap\ker(\curl),\\
\harmbihoneDom
&=\ker(\rCS)\cap\ker(\dDS),
&
\harmbihoneNom
&=\ker(\rDT)\cap\ker(\symCT),\\
\harmbihtwoDom
&=\ker(\rsymCT)\cap\ker(\DT),
&
\harmbihtwoNom
&=\ker(\rdDS)\cap\ker(\CS),\\
\harmelaDom
&=\ker(\rCCtS)\cap\ker(\DS),
&
\harmelaNom
&=\ker(\rDS)\cap\ker(\CCtS).
\end{align*}
We will compute the dimensions of the kernels $N_{0}$, $N_{2,*}$, i.e., 
\begin{align*}
\dim\ker(\rgrad)&=0,
&
\dim\ker(\grad)&=n,\\
\dim\ker(\rGg)&=0,
&
\dim\ker(\devGrad)&=4n,\\
\dim\ker(\rdevGrad)&=0,
&
\dim\ker(\Gg)&=4n,\\
\dim\ker(\rsymGrad)&=0,
&
\dim\ker(\symGrad)&=6n,
\intertext{and the dimensions of the cohomology groups $K_{1}$, $K_{2}$, i.e.,}
\dim\harmDom&=m-1,
&
\dim\harmNom&=p,\\
\dim\harmbihoneDom&=4(m-1),
&
\dim\harmbihoneNom&=4p,\\
\dim\harmbihtwoDom&=4(m-1),
&
\dim\harmbihtwoNom&=4p,\\
\dim\harmelaDom&=6(m-1),
&
\dim\harmelaNom&=6p,
\intertext{and the indices $\ind\calD$, $\ind\calD^{*}$ of the involved Fredholm operators, i.e.,}
\ind\calD^{\rhmp}&=p-m-n+1,
&
\ind(\calD^{\rhmp})^{*}&=-\ind\calD^{\rhmp},\\
\ind\calD^{\bihonep}&=4(p-m-n+1),
&
\ind(\calD^{\bihonep})^{*}&=-\ind\calD^{\bihonep},\\
\ind\calD^{\bihtwop}&=4(p-m-n+1),
&
\ind(\calD^{\bihtwop})^{*}&=-\ind\calD^{\bihtwop},\\
\ind\calD^{\elap}&=6(p-m-n+1),
&
\ind(\calD^{\elap})^{*}&=-\ind\calD^{\elap}.
\end{align*}

\begin{remark}
\label{rem:genindex}
We observe that in all of our examples,
where generally the operators $A_{j}$ carry the boundary condition
and the adjoints $A_{j}^{*}$ do not have boundary condition,
the dimensions of the first and second cohomology groups 
$K_{1}$ and $K_{2}$ (Dirichlet fields and Neumann fields) 
are given by
$$\dim K_{1}=\frac{\dim N_{2,*}}{n}\cdot(m-1),\qquad
\dim K_{2}=\frac{\dim N_{2,*}}{n}\cdot p,$$
respectively. The indices of $\calD$ (see \eqref{eq:calD}) and $\calD^{*}$ are 
\begin{equation}
\label{eq:genindex}
-\ind\calD^{*}=\ind\calD=\frac{\dim N_{2,*}}{n}\cdot(p-m-n+1).
\end{equation}
\end{remark}

Remark \ref{rem:genindex} leads to the following problem that seems to be open:

\begin{problem}
\label{prop:genindex} 
Is it possible to find differential operators on $\om\subset \R^3$ (bounded, strong Lipschitz domain) 
of the form \eqref{eq:calD} as discussed in Remark \ref{rem:genindex} 
that violate the general index formula for $\calD$ in \eqref{eq:genindex}?
\end{problem}

\section{The Construction Principle and the Index Theorem}
\label{sec:3}

In this section, we provide the basic construction principle, which is the basis for the operators in question. 
The theory in more general terms has been developed already in \cite{BL1992}. 
Here, we rephrase the situation with a slightly more particular viewpoint. 
For the convenience of the reader, we carry out the necessary proofs here. 

Throughout this section, we let $H_{0}, H_{1}, H_{2}, H_{3}$ be Hilbert spaces, and
\begin{align*}
A_{0} &:\dom A_{0}\subset H_{0}\To H_{1},\\ 
A_{1} &:\dom A_{1}\subset H_{1}\To H_{2},\\ 
A_{2} &:\dom A_{2}\subset H_{2}\To H_{3}
\end{align*} 
be densely defined and \underline{closed} linear operators.

\begin{definition} 
Let $A_{0},A_{1},A_{2}$ be defined as above.
\begin{itemize}
\item
We call a pair $(A_{0},A_{1})$ a \emph{complex} (\emph{Hilbert complex}), 
if $\ran A_{0}\subset\ker A_{1}$. In this situation we also write
\begin{align*}
H_{0}\xrightarrow{A_{0}}
H_{1}&\xrightarrow{A_{1}}
H_{2}.
\end{align*}\item
We say a complex $(A_{0},A_{1})$ is \emph{closed}, 
if $\ran A_{0}$ and $\ran A_{1}$ are closed.
\item
A complex $(A_{0},A_{1})$ is said to be \emph{compact}, 
if the embedding $\dom A_{1}\cap\dom A_{0}^{*}\hookrightarrow H_{1}$ is compact.
\item
The triple $(A_{0},A_{1},A_{2})$ is called a \emph{(closed/compact) complex}, 
if both $(A_{0},A_{1})$ and $(A_{1},A_{2})$ are (closed/compact) complexes. 
\item
We say that a complex $(A_{0},A_{1},A_{2})$ is \emph{maximal compact}, 
if $(A_{0},A_{1},A_{2})$ is a compact complex and 
both embeddings $\dom A_{0}\hookrightarrow H_{0}$ and $\dom A_{2}^{*}\hookrightarrow H_{3}$ 
are compact as well.
\end{itemize}
\end{definition}

\begin{remark}
\label{rem:dualcomplex}
The `FA-ToolBox' (`Functional-Analysis-Tool Box') from \cite{P2019a,P2019b,P2020a,PZ2020a,PZ2020b,PZ2020c} shows that 
$$(A_{0},A_{1})\;\text{(closed/compact) complex}\quad\iff\quad
(A_{1}^{*},A_{0}^{*})\;\text{(closed/compact) complex.}$$
As a consequence, we obtain $(A_{0},A_{1},A_{2})$ 
is a (closed/compact/maximal compact) complex if and only if $(A_{2}^{*},A_{1}^{*},A_{0}^{*})$ 
is a (closed/compact/maximal compact) complex.
\end{remark}

Throughout this section, we assume that $(A_{0},A_{1},A_{2})$ is a complex, i.e.,
\begin{align*}
H_{0}\xrightarrow{A_{0}}
H_{1}&\xrightarrow{A_{1}}
H_{2}\xrightarrow{A_{2}}
H_{3},\\
H_{0}\xleftarrow{A_{0}^{*}}
H_{1}&\xleftarrow{A_{1}^{*}}
H_{2}\xleftarrow{A_{2}^{*}}
H_{3}.
\end{align*}
We define the operator
\begin{align*}
\calD:(\dom A_{2}\cap\dom A_{1}^{*})\times\dom A_{0}\subset H_{2}\times H_{0}&\To H_{3}\times H_{1}\\
(x,y)\hspace*{32mm}&\longmapsto(A_{2}x,A_{1}^{*}x+A_{0}y).
\end{align*}
In block operator matrix notation, we have
$$\calD=\begin{pmatrix}A_{2}&0\\A_{1}^{*}&A_{0}\end{pmatrix}.$$
From the introduction, we recall the notation
\begin{equation}
\label{eq:ker} 
N_{0}:=\ker A_{0},\qquad
N_{2,*}:=\ker A_{2}^{*}
\end{equation}
and 
\begin{equation}
\label{eq:coh}
K_{1}:=\ker A_{1}\cap\ker A_{0}^{*},\qquad
K_{2}:=\ker A_{2}\cap\ker A_{1}^{*}.
\end{equation}
The aim of this section is to provide a proof of Theorem \ref{thm:index} below. As a standard tool for this and related results, we recall the standard orthogonal decompositions 
\begin{align}
\label{deco1}
H_{2}&=\overline{\ran A_{2}^{*}}\oplus_{H_{2}}\ker A_{2},
&
H_{2}&=\ker A_{1}^{*}\oplus_{H_{2}}\overline{\ran A_{1}},\\
\nonumber
\dom A_{2}&=(\dom A_{2}\cap\overline{\ran A_{2}^{*}})\oplus_{H_{2}}\ker A_{2},
&
\dom A_{1}^{*}&=\ker A_{1}^{*}\oplus_{H_{2}}(\dom A_{1}^{*} \cap\overline{\ran A_{1}}).
\end{align}
Using \eqref{eq:coh}, by the complex property we get
\begin{equation}\label{deco1.5}\ker A_{2}=K_{2}\oplus_{H_{2}}\overline{\ran A_{1}}\end{equation}
and hence we obtain the following (abstract) Helmholtz type decomposition
\begin{align}
\label{deco2}
\begin{aligned}
H_{2} 
&=\overline{\ran A_{2}^{*}}\oplus_{H_{2}}K_{2}\oplus_{H_{2}}\overline{\ran A_{1}},\\
\dom A_{2}\cap\dom A_{1}^{*}
&=(\dom A_{2}\cap\overline{\ran A_{2}^{*}})\oplus_{H_{2}}K_{2}\oplus_{H_{2}}(\dom A_{1}^{*}\cap\overline{\ran A_{1}}).
\end{aligned}
\end{align}
We gather some elementary facts about $\calD$.

\begin{proposition}
\label{thm:properties1} 
$\calD$ is a densely defined and closed linear operator.
\end{proposition}

\begin{proof}
For the closedness of $\calD$, we let $\big((x_{k},y_{k})\big)_{k}$ be a sequence in $\dom\calD$ 
with $\big((x_{k},y_{k})\big)_{k}$ converging to some $(x,y)$ in $H_{2}\times H_{0}$ 
and $(\calD(x_{k},y_{k}))_{k}$ converging to $(w,z)$ in $H_{3}\times H_{1}$. 
One readily sees using the closedness of $A_{2}$ that $x\in\dom A_{2}$ and $A_{2}x=w$. 
Next, we observe that $\ran A_{0}\subset\ker A_{1}\,\bot_{H_{1}}\,\ran A_{1}^{*}$. 
Hence, $(A_{1}^{*}x_{k})_{k}$ and $(A_{0} y_{k})$ are both convergent to some $z_{1}\in H_{1}$ and $z_{2}\in H_{1}$, respectively. 
By the closedness of both $A_{1}^{*}$ and $A_{0}$, we thus deduce that $x\in\dom A_{1}^{*}$ and $y\in\dom A_{0}$ 
with $z_{1}=A_{1}^{*}x$ and $z_{2}=A_{0}y$ as well as $z=z_{1}+z_{2}=A_{1}^{*}x+A_{0}y$. 

For $\calD$ being densely defined, we see that by assumption, $\dom A_{0}$ is dense in $H_{0}$. 
Hence, it suffices to show that $\dom A_{2}\cap\dom A_{1}^{*}$ is dense in $H_{2}$. 
%For this, we deduce that \eqref{deco1} implies
%\[
%  \dom A_{2} = (\dom A_{2} \cap \overline{\ran A_{2}^{*}})\oplus_{H_{2}} \ker A_{2},\quad  
%  \dom A_{1}^{*} = \ker A_{1}^{*} \oplus_{H_{2}} (\dom A_{1}^{*} \cap \overline{\ran A_{1}}).
%\]
Thus, as $\dom A_{2}$ and $\dom A_{1}^{*}$ are dense in $H_{2}$, 
we deduce by \eqref{deco1} that $\dom A_{2} \cap \overline{\ran A_{2}^{*}}$ and $\dom A_{1}^{*} \cap \overline{\ran A_{1}}$ 
are dense in $\overline{\ran A_{2}^{*}}$ and $\overline{\ran A_{1}}$, respectively. 
Thus, the decomposition in \eqref{deco2} implies that $\dom A_{2}\cap\dom A_{1}^{*}$ 
is dense in $H_{2}$, which yields the assertion.
\end{proof}

\begin{theorem}
\label{thm:properties2} 
$\calD^{*}=\begin{pmatrix}A_{2}^{*}&A_{1}\\0&A_{0}^{*}\end{pmatrix}$. More precisely,
\begin{align*}
\calD^{*}:\dom A_{2}^{*}\times(\dom A_{1}\cap\dom A_{0}^{*})\subset H_{3}\times H_{1}&\To H_{2}\times H_{0}\\
(w,z)\hspace*{53mm}&\longmapsto(A_{2}^{*}w+A_{1}z,A_{0}^{*}z).
\end{align*}
\end{theorem}

\begin{proof}
Note that 
$$\begin{pmatrix}A_{2}^{*}&A_{1}\\0&A_{0}^{*}\end{pmatrix}\subset\calD^{*}$$
holds by definition since for all $(x,y)\in\dom\calD=(\dom A_{2}\cap\dom A_{1}^{*})\times\dom A_{0}$
and for all $(w,z)\in\dom A_{2}^{*}\times(\dom A_{1}\cap\dom A_{0}^{*})$ 
\begin{align*}
\bscp{\calD(x,y)}{(w,z)}_{H_{3}\times H_{1}}
&=\scp{A_{2}x}{w}_{H_{3}}
+\scp{A_{1}^{*}x+A_{0}y}{z}_{H_{1}}\\
&=\scp{x}{A_{2}^{*}w+A_{1}z}_{H_{2}}
+\scp{y}{A_{0}^{*}z}_{H_{0}}
=\bscp{(x,y)}{\calD^{*}(w,z)}_{H_{2}\times H_{0}}.
\end{align*}

Let $(w,z)\in\dom\calD^{*}$ and set $(u,v):=\calD^{*}(w,z)$. 
For $y\in\dom A_{0}$ we have $(0,y)\in\dom\calD$. We obtain 
\begin{align*}
\scp{A_{0}y}{z}_{H_{1}}
=\bscp{\calD(0,y)}{(w,z)}_{H_{3}\times H_{1}}
=\bscp{(0,y)}{\calD^{*}(w,z)}_{H_{2}\times H_{0}}
=\scp{y}{v}_{H_{0}}.
\end{align*}
Hence, $z\in\dom A_{0}^{*}$ and $A_{0}^{*}z = v$. 

For all $x\in\dom A_{2}\cap\dom A_{1}^{*}$ we see $(x,0)\in\dom\calD$
and deduce that
\begin{align}
\label{adjcomp1}
\begin{aligned}
\scp{A_{2}x}{w}_{H_{3}}+\scp{A_{1}^{*}x}{z}_{H_{1}}
&=\bscp{\calD(x,0)}{(w,z)}_{H_{3}\times H_{1}}\\
&=\bscp{(x,0)}{\calD^{*}(w,z)}_{H_{2}\times H_{0}}
=\scp{x}{u}_{H_{2}}.
\end{aligned}
\end{align}
Let $\pi_{2}$ denote the orthonormal projector onto $\overline{\ran A_{2}^{*}}$ in \eqref{deco1}.
Then for $\widetilde{x}\in\dom A_{2}$ we have
$$x:=\pi_{2}\widetilde{x}
\in\dom A_{2}\cap\overline{\ran A_{2}^{*}}\subset\dom A_{2}\cap\ker A_{1}^{*}\subset\dom A_{2}\cap\dom A_{1}^{*},\quad
A_{2}x=A_{2}\widetilde{x}$$ 
and by \eqref{adjcomp1} 
$$\scp{A_{2}\widetilde{x}}{w}_{H_{3}}
=\scp{A_{2}x}{w}_{H_{3}}+\scp{A_{1}^{*}x}{z}_{H_{1}}
=\scp{x}{u}_{H_{2}}
=\scp{\widetilde{x}}{\pi_{2}u}_{H_{2}}.$$
Thus $w\in\dom A_{2}^{*}$ and $A_{2}^{*}w=\pi_{2}u$.
Analogously, let $\pi_{1}$ denote the orthonormal projector onto $\overline{\ran A_{1}}$ in \eqref{deco1}.
Then for $\widetilde{x}\in\dom A_{1}^{*}$ we have
$$x:=\pi_{1}\widetilde{x}
\in\dom A_{1}^{*}\cap\overline{\ran A_{1}}\subset\dom A_{1}^{*}\cap\ker A_{2}\subset\dom A_{2}\cap\dom A_{1}^{*},\quad
A_{1}^{*}x=A_{1}^{*}\widetilde{x}$$ 
and by \eqref{adjcomp1} 
$$\scp{A_{1}^{*}\widetilde{x}}{z}_{H_{1}}
=\scp{A_{2}x}{w}_{H_{3}}+\scp{A_{1}^{*}x}{z}_{H_{1}}
=\scp{x}{u}_{H_{2}}
=\scp{\widetilde{x}}{\pi_{1}u}_{H_{2}}.$$
Thus $z\in\dom A_{1}$ and $A_{1}z=\pi_{1}u$.
Therefore, $(w,z)\in\dom A_{2}^{*}\times(\dom A_{1}\cap\dom A_{0}^{*})$.
Moreover, using the orthonormal projector $\pi_{0}$ onto $K_{2}$ in \eqref{deco2} we see for $x\in K_{2}$ by \eqref{adjcomp1} 
$$\scp{x}{\pi_{0}u}_{H_{2}}
=\scp{\pi_{0}x}{u}_{H_{2}}
=\scp{x}{u}_{H_{2}}
=\scp{A_{2}x}{w}_{H_{3}}+\scp{A_{1}^{*}x}{z}_{H_{1}}
=0,$$
yielding $\pi_{0}u=0$. Finally, by \eqref{deco2} we arrive at
$$\calD^{*}(w,z)
=(u,v)
=(\pi_{0}u+\pi_{1}u+\pi_{2}u,A_{0}^{*}z)
=(A_{1}z+A_{2}^{*}w,A_{0}^{*}z),$$
completing the proof.
\end{proof}

\begin{lemma}
\label{lem:properties3} 
With the settings \eqref{eq:ker} and \eqref{eq:coh}, the kernels of $\calD$ and $\calD^{*}$ read
\begin{align*}
\ker\calD&=K_{2}\times N_{0}=(\ker A_{2}\cap\ker A_{1}^{*})\times\ker A_{0},\\
\ker\calD^{*}&=N_{2,*}\times K_{1}=\ker A_{2}^{*}\times(\ker A_{1}\cap\ker A_{0}^{*}).
\end{align*}
\end{lemma}

\begin{proof}
For $(x,y)\in\ker\calD$ we have $A_{2}x=0$ and $A_{1}^{*}x+A_{0}y=0$.
By orthogonality and the complex property, i.e.,
$\ran A_{0}\subset\ker A_{1}\,\bot_{H_{1}}\,\ran A_{1}^{*}$,
we see $A_{1}^{*}x=A_{0}y=0$. The assertion about $\ker\calD^{*}$ (use Theorem \ref{thm:properties2} and Remark \ref{rem:dualcomplex}) follows analogously.
\end{proof}
With Lemma \ref{lem:properties3} at hand, the following result is immediate.
\begin{corollary}
\label{lem:properties3a} 
The closures of the ranges of $\calD$ and $\calD^{*}$ are given by
\begin{align*}
\overline{\ran\calD}=(\ker\calD^{*})^{\bot_{H_{3}\times H_{1}}}
&=N_{2,*}^{\bot_{H_{3}}}\times K_{1}^{\bot_{H_{1}}},\\
\overline{\ran\calD^{*}}=(\ker\calD)^{\bot_{H_{2}\times H_{0}}}
&=K_{2}^{\bot_{H_{2}}}\times N_{0}^{\bot_{H_{0}}}.
\end{align*}
\end{corollary}

\begin{lemma}
\label{lem:properties4} 
Let $(A_{0},A_{1},A_{2})$ be a maximal compact Hilbert complex. 
Then the embedding $\dom\calD\hookrightarrow H_{2}\times H_{0}$ is compact,
and so is the embedding $\dom\calD^{*}\hookrightarrow H_{3}\times H_{1}$.
\end{lemma}

\begin{proof}
Let $\big((x_{k},y_{k})\big)_{k}$ be a $(\dom\calD)$-bounded sequence in $\dom\calD$.
Then, as in the proof of Lemma \ref{lem:properties3},
by orthogonality and the complex property 
$(x_{k})_{k}$ is a $(\dom A_{2}\cap\dom A_{1}^{*})$-bounded sequence in $\dom A_{2}\cap\dom A_{1}^{*}$
and $(y_{k})_{k}$ is a $(\dom A_{0})$-bounded sequence in $\dom A_{0}$.
Since $(A_{0},A_{1},A_{2})$ is maximal compact, we can extract converging subsequences of 
$(x_{k})_{k}$ and $(y_{k})_{k}$.
Analogously, using Theorem \ref{thm:properties2} and Remark \ref{rem:dualcomplex}, we see that also $\dom\calD^{*}\hookrightarrow H_{3}\times H_{1}$ is compact,
finishing the proof.
\end{proof}

We now recall the abstract index theorem taken from \cite{BL1992} formulated for the present situation.

\begin{theorem}
\label{thm:index} 
Let $(A_{0},A_{1},A_{2})$ be a maximal compact Hilbert complex. 
Then $\calD$ and $\calD^{*}$ are Fredholm operators with indices 
$$\ind\calD=\dim N_{0}-\dim K_{1}+\dim K_{2}-\dim N_{2,*},\qquad
\ind\calD^{*}=-\ind\calD.$$
\end{theorem}

\begin{proof}
Utilising the `FA-ToolBox' from, e.g., \cite{P2019a,P2019b,P2020a,PZ2020a,PZ2020b,PZ2020c},
and Lemma \ref{lem:properties4} we observe that both ranges
$\ran\calD$ and $\ran\calD^{*}$ are closed and that both kernels 
$\ker\calD$ and $\ker\calD^{*}$ are finite-dimensional.
Therefore, both $\calD$ and $\calD^{*}$ are Fredholm operators.
The index $\ind\calD=\dim\ker\calD-\dim\ker\calD^{*}$
is then easily computed with the help of Lemma \ref{lem:properties3}.
\end{proof}

\section{Abstract Poincar\'e--Friedrichs Type Inequalities}
\label{sec:3-addstuff}

Let us mention some additional features of the `FA-ToolBox' from \cite{P2019a,P2019b,P2020a,PZ2020a,PZ2020b,PZ2020c}.
Lemma \ref{lem:properties4} and Theorem \ref{thm:index} imply some additional results for the reduced operators 
$$\calD_{\red}:=\calD|_{\ran\calD^{*}}=\calD|_{(\ker\calD)^{\bot_{H_{2}\times H_{0}}}},\qquad
\calD^{*}_{\red}:=\calD^{*}|_{\ran\calD}=\calD^{*}|_{(\ker\calD^{*})^{\bot_{H_{3}\times H_{1}}}}.$$ 

\begin{corollary}
\label{cor:FP}
Let $(A_{0},A_{1},A_{2})$ be a maximal compact Hilbert complex. Then the inverse operators
$\calD_{\red}^{-1}:\ran\calD\to\ran\calD^{*}$ and $(\calD^{*}_{\red})^{-1}:\ran\calD^{*}\to\ran\calD$ 
are compact. Moreover, 
$\calD_{\red}^{-1}:\ran\calD\to\dom\calD_{\red}$ and $(\calD^{*}_{\red})^{-1}:\ran\calD^{*}\to\dom\calD^{*}_{\red}$ 
are continuous and, equivalently, the Friedrichs--Poincar\'e type estimates 
\begin{align*}
\bnorm{(x,y)}_{H_{2}\times H_{0}}
&\leq c_{\calD}\bnorm{\calD(x,y)}_{H_{3}\times H_{1}}
=c_{\calD}\big(\norm{A_{2}x}_{H_{3}}^{2}+\norm{A_{1}^{*}x}_{H_{1}}^{2}+\norm{A_{0}y}_{H_{1}}^{2}\big)^{1/2},\\
\bnorm{(w,z)}_{H_{3}\times H_{1}}
&\leq c_{\calD}\bnorm{\calD^{*}(w,z)}_{H_{2}\times H_{0}}
=c_{\calD}\big(\norm{A_{2}^{*}w}_{H_{2}}^{2}+\norm{A_{1}z}_{H_{2}}^{2}+\norm{A_{0}^{*}z}_{H_{0}}^{2}\big)^{1/2}
\end{align*}
hold for all $(x,y)\in\dom\calD_{\red}$ and for all $(w,z)\in\dom\calD^{*}_{\red}$ 
with the same optimal constant $c_{\calD}>0$.
\end{corollary}

The latter estimates are additive combinations of the 
corresponding estimates for $A_{0}$ and $(A_{2},A_{1}^{*})$
as well as $A_{2}^{*}$ and $(A_{1},A_{0}^{*})$, respectively.

\begin{remark}
\label{rem:nocpt}
The compactness assumptions (maximal compact) are not needed to render $\calD$ and $\calD^{*}$ Fredholm operators. 
It suffices to assume that $(A_{0},A_{1},A_{2})$ is a closed Hilbert complex 
with finite-dimensional kernels $N_{0}$ and $N_{2,*}$
and finite-dimensional cohomology groups $K_{1}$ and $K_{2}$.
In this case, the latter Friedrichs--Poincar\'e type estimates still hold
and $\calD_{\red}^{-1}$ and $(\calD_{\red}^{*})^{-1}$ are still continuous.
\end{remark}

\begin{remark}
\label{rem:primaldualD}
There are simple relations between the primal, dual, and adjoint complexes, when $\calD$ is considered.
More precisely, let us denote the latter primal operators $\calD$ and $\calD^{*}$
of the primal complex $(A_{0},A_{1},A_{2})$ by
\begin{align*}
\calD=\calD^{p}&=\begin{pmatrix}A_{2}&0\\A_{1}^{*}&A_{0}\end{pmatrix},
&
\calD^{*}=(\calD^{p})^{*}&=\begin{pmatrix}A_{2}^{*}&A_{1}\\0&A_{0}^{*}\end{pmatrix},
\intertext{and the dual operators corresponding to the dual complex $(A_{2}^{*},A_{1}^{*},A_{0}^{*})$ by}
\calD^{d}&=\begin{pmatrix}A_{0}^{*}&0\\A_{1}&A_{2}^{*}\end{pmatrix},
&
(\calD^{d})^{*}&=\begin{pmatrix}A_{0}&A_{1}^{*}\\0&A_{2}\end{pmatrix}.
\end{align*}
By Remark \ref{rem:dualcomplex} $(A_{0},A_{1},A_{2})$ is a maximal compact complex, if and only if
$(A_{2}^{*},A_{1}^{*},A_{0}^{*})$ is a maximal compact complex. 
Note that we may weaken the assumptions along the lines sketched in Remark \ref{rem:nocpt}.
Theorem \ref{thm:index} shows
that $\calD^{p}$, $(\calD^{p})^{*}$, $\calD^{d}$, $(\calD^{d})^{*}$ are Fredholm operators with indices 
\begin{align*}
\ind\calD^{p}&=\dim N_{0}^{p}-\dim K_{1}^{p}+\dim K_{2}^{p}-\dim N_{2,*}^{p},
&
\ind(\calD^{p})^{*}&=-\ind\calD^{p},\\
\ind\calD^{d}&=\dim N_{0}^{d}-\dim K_{1}^{d}+\dim K_{2}^{d}-\dim N_{2,*}^{d},
&
\ind(\calD^{d})^{*}&=-\ind\calD^{d}.
\end{align*}
Next we observe
\begin{align*}
N_{0}^{d}&=\ker A_{2}^{*}=N_{2,*}^{p},
&
N_{2,*}^{d}&=\ker A_{0}=N_{0}^{p},\\
K_{1}^{d}&=\ker A_{1}^{*}\cap\ker A_{2}=K_{2}^{p},
&
K_{2}^{d}&=\ker A_{0}^{*}\cap\ker A_{1}=K_{1}^{p}.
\end{align*}
Hence
$$-\ind(\calD^{d})^{*}=\ind\calD^{d}=-\ind\calD^{p}=\ind(\calD^{p})^{*}.$$
Note that basically $\calD^{d}$ and $(\calD^{p})^{*}$
as well as $\calD^{p}$ and $(\calD^{d})^{*}$ are the `same' operators.
\end{remark}

\section{The Case of Variable Coefficients}
\label{sec:4-addstuff}

Note that the Hilbert space adjoints $A_{l}^{*}$ 
depend on the particular choice of the inner products (metrics) of the underlying Hilbert spaces $H_{l}$.
A typical example is simply given by `weighted' inner products
induced by `weights' $\lambda_{l}$, $l\in\{0,1,2,3\}$, i.e., symmetric and positive topological isomorphisms 
(symmetric and positive bijective bounded linear operators)
$\lambda_{l}:H_{l}\to H_{l}$ inducing inner products
$$\scp{\,\cdot\,}{\,\cdot\,}_{\widetilde{H}_{l}}
:=\scp{\lambda_{l}\,\cdot\,}{\,\cdot\,}_{H_{l}}:\widetilde{H}_{l}\times\widetilde{H}_{l}\to\C,$$
where $\widetilde{H}_{l}:=H_{l}$ (as linear space) 
equipped with the inner product $\scp{\,\cdot\,}{\,\cdot\,}_{\widetilde{H}_{l}}$.
A sufficiently general situation is defined by 
$\lambda_{0}:=\id$, $\lambda_{3}:=\id$, and
$\lambda_{1},\lambda_{2}$ being symmetric and positive topological isomorphisms, as well as
$\widetilde{H}_{l}:=\big(H_{l},\scp{\lambda_{l}\,\cdot\,}{\,\cdot\,}_{H_{l}}\big)$, $l\in\{0,1,2,3\}$.
Then the modified operators\footnote{E.g., we compute $\widetilde{A}_{0}^{*}$.
Let $y\in\dom\widetilde{A}_{0}^{*}$. Then for $x\in\dom\widetilde{A}_{0}=\dom A_{0}$
$$\scp{x}{\widetilde{A}_{0}^{*}y}_{H_{0}}
=\scp{x}{\widetilde{A}_{0}^{*}y}_{\widetilde{H}_{0}}
=\scp{\widetilde{A}_{0}x}{y}_{\widetilde{H}_{1}}
=\scp{A_{0}x}{\lambda_{1}y}_{H_{1}},$$
showing that $\lambda_{1}y\in\dom A_{0}^{*}$ and $A_{0}^{*}\lambda_{1}y=\widetilde{A}_{0}^{*}y$.}
\begin{align*}
\widetilde{A}_{0}:\dom\widetilde{A}_{0}:=\dom A_{0}\subset\widetilde{H}_{0}
&\To\widetilde{H}_{1};
&
x&\longmapsto A_{0}x,\\
\widetilde{A}_{1}:\dom\widetilde{A}_{1}:=\dom A_{1}\subset\widetilde{H}_{1}
&\To\widetilde{H}_{2};
&
y&\longmapsto\lambda_{2}^{-1}A_{1}y,\\
\widetilde{A}_{2}:\dom\widetilde{A}_{2}:=\lambda_{2}^{-1}\dom A_{2}\subset\widetilde{H}_{2}
&\To\widetilde{H}_{3};
&
z&\longmapsto A_{2}\lambda_{2}z,\\
\widetilde{A}_{0}^{*}:\dom\widetilde{A}_{0}^{*}=\lambda_{1}^{-1}\dom A_{0}^{*}\subset\widetilde{H}_{1}
&\To\widetilde{H}_{0};
&
y&\longmapsto A_{0}^{*}\lambda_{1}y,\\
\widetilde{A}_{1}^{*}:\dom\widetilde{A}_{1}^{*}=\dom A_{1}^{*}\subset\widetilde{H}_{2}
&\To\widetilde{H}_{1};
&
z&\longmapsto\lambda_{1}^{-1}A_{1}^{*}z,\\
\widetilde{A}_{2}^{*}:\dom\widetilde{A}_{2}^{*}=\dom A_{2}^{*}\subset\widetilde{H}_{3}
&\To\widetilde{H}_{2};
&
x&\longmapsto A_{2}^{*}x
\end{align*}
form again a primal and dual Hilbert complex, i.e.,
\begin{align*}
\widetilde{H}_{0}\xrightarrow{\widetilde{A}_{0}}
\widetilde{H}_{1}&\xrightarrow{\widetilde{A}_{1}}
\widetilde{H}_{2}\xrightarrow{\widetilde{A}_{2}}
\widetilde{H}_{3},\\
\widetilde{H}_{0}\xleftarrow{\widetilde{A}_{0}^{*}}
\widetilde{H}_{1}&\xleftarrow{\widetilde{A}_{1}^{*}}
\widetilde{H}_{2}\xleftarrow{\widetilde{A}_{2}^{*}}
\widetilde{H}_{3},
\end{align*}
and we can define 
$$\widetilde{\calD}:=\begin{pmatrix}\widetilde{A}_{2}&0\\\widetilde{A}_{1}^{*}&\widetilde{A}_{0}\end{pmatrix},\qquad
\widetilde{\calD}^{*}=\begin{pmatrix}\widetilde{A}_{2}^{*}&\widetilde{A}_{1}\\0&\widetilde{A}_{0}^{*}\end{pmatrix}.$$
The closedness of the operators $\widetilde{A}_{l}$ and the complex properties are easily checked.
Moreover, it is not hard to see that the closedness of $(\widetilde{A}_{0},\widetilde{A}_{1},\widetilde{A}_{2})$
is implied by the closedness of $(A_{0},A_{1},A_{2})$.
Remark \ref{rem:dualcomplex}, Proposition \ref{thm:properties1}, Theorem \ref{thm:properties2}, 
Lemma \ref{lem:properties3}, and Corollary \ref{lem:properties3a} can be applied to
$(\widetilde{A}_{0},\widetilde{A}_{1},\widetilde{A}_{2})$ as well.
In particular,
\begin{align*}
\ker\widetilde{\calD}
&=\widetilde{K}_{2}\times\widetilde{N}_{0}
=(\ker\widetilde{A}_{2}\cap\ker\widetilde{A}_{1}^{*})\times\ker\widetilde{A}_{0}
=\big((\lambda_{2}^{-1}\ker A_{2})\cap\ker A_{1}^{*}\big)\times\ker A_{0},\\
\ker\widetilde{\calD}^{*}
&=\widetilde{N}_{2,*}\times\widetilde{K}_{1}
=\ker\widetilde{A}_{2}^{*}\times(\ker\widetilde{A}_{1}\cap\ker\widetilde{A}_{0}^{*})
=\ker A_{2}^{*}\times\big(\ker A_{1}\cap(\lambda_{1}^{-1}\ker A_{0}^{*})\big),\\
\overline{\ran\widetilde{\calD}}
&=(\ker\widetilde{\calD}^{*})^{\bot_{\widetilde{H}_{3}\times\widetilde{H}_{1}}}
=\widetilde{N}_{2,*}^{\bot_{\widetilde{H}_{3}}}\times\widetilde{K}_{1}^{\bot_{\widetilde{H}_{1}}},\\
\overline{\ran\widetilde{\calD}^{*}}
&=(\ker\widetilde{\calD})^{\bot_{\widetilde{H}_{2}\times\widetilde{H}_{0}}}
=\widetilde{K}_{2}^{\bot_{\widetilde{H}_{2}}}\times\widetilde{N}_{0}^{\bot_{\widetilde{H}_{0}}}.
\end{align*}
It is possible to relate the statements in Lemma \ref{lem:properties4} and Theorem \ref{thm:index} 
to the corresponding ones of the original complex $(A_{0},A_{1},A_{2})$. This will be done next.

\begin{lemma}
\label{lambdaindeplem}
The compactness properties and the dimensions of the kernels and cohomology groups
of the latter complexes are independent of the weights $\lambda_{l}$.
More precisely, 
\begin{itemize}
\item[\bf(i)]
$\widetilde{N}_{0}=N_{0}$ and $\widetilde{N}_{2,*}=N_{2,*}$,
\quad as \quad $\dom\widetilde{A}_{0}=\dom A_{0}$ and $\dom\widetilde{A}_{2,*}=\dom A_{2,*}$,
\item[\bf(ii$_{\mathbf1}$)]
$\dim\big(\ker A_{1}\cap(\lambda_{1}^{-1}\ker A_{0}^{*})\big)
=\dim\widetilde{K}_{1}
=\dim K_{1}
=\dim(\ker A_{1}\cap\ker A_{0}^{*})$,
\item[\bf(ii$_{\mathbf2}$)]
$\dim\big(\ker A_{2}\cap(\lambda_{2}^{-1}\ker A_{1}^{*})\big)
=\dim\widetilde{K}_{2}
=\dim K_{2}
=\dim(\ker A_{2}\cap\ker A_{1}^{*})$,
\item[\bf(iii$_{\mathbf1}$)]
$\dom\widetilde{A}_{1}\cap\dom\widetilde{A}_{0}^{*}
=\dom A_{1}\cap(\lambda_{1}^{-1}\dom A_{0}^{*})
\hookrightarrow\widetilde{H}_{1}$ compactly\\
$\Leftrightarrow\;\dom A_{1}\cap\dom A_{0}^{*}
\hookrightarrow H_{1}$ compactly,
\item[\bf(iii$_{\mathbf2}$)]
$\dom\widetilde{A}_{2}\cap\dom\widetilde{A}_{1}^{*}
=\dom A_{2}\cap(\lambda_{2}^{-1}\dom A_{1}^{*})
\hookrightarrow\widetilde{H}_{2}$ compactly\\
$\Leftrightarrow\;\dom A_{2}\cap\dom A_{1}^{*}
\hookrightarrow H_{2}$ compactly.
\end{itemize}
\end{lemma}

\begin{proof}
For the proof we follow in close lines the ideas of \cite[Theorem 6.1]{BPS2019a},
where \cite{BPS2019a} is the extended version of \cite{BPS2019b}.
(i) is trivial and it is sufficient to show only (ii$_{1}$) and (iii$_{1}$).

For (ii$_{1}$), 
let $\mu$ be another weight having the same properties as $\lambda_{1}$.
Similar to \eqref{deco1}, \eqref{deco2} we have by orthogonality 
in $\widetilde{H}_{1}$ and by the complex property
\begin{align}
\label{deco3}
\begin{aligned}
\widetilde{H}_{1}
&=\overline{\ran\widetilde{A}_{0}}\oplus_{\widetilde{H}_{1}}\ker\widetilde{A}_{0}^{*}
=\overline{\ran A_{0}}\oplus_{\widetilde{H}_{1}}\lambda_{1}^{-1}\ker A_{0}^{*},\\
\ker\widetilde{A}_{1}
&=\overline{\ran\widetilde{A}_{0}}\oplus_{\widetilde{H}_{1}}(\ker\widetilde{A}_{1}\cap\ker\widetilde{A}_{0}^{*})
=\overline{\ran A_{0}}\oplus_{\widetilde{H}_{1}}\big(\ker A_{1}\cap(\lambda_{1}^{-1}\ker A_{0}^{*})\big),
\end{aligned}
\end{align}
and we note that $\widetilde{H}_{1}=H_{1}$ and $\ker\widetilde{A}_{1}=\ker A_{1}$ as sets. 
We denote the $\widetilde{H}_{1}$-orthonormal projector along $\overline{\ran A_{0}}$ 
onto $\lambda_{1}^{-1}\ker A_{0}^{*}$ by $\pi$. 
Then, by \eqref{deco3}, we deduce 
$$\pi(\ker {A}_{1})=\pi(\ker \widetilde{A}_{1})=\ker A_{1}\cap(\lambda_{1}^{-1}\ker A_{0}^{*}).$$
We consider the linear mapping
$$\widehat{\pi}:\ker A_{1}\cap(\mu^{-1}\ker A_{0}^{*})\To\ker A_{1}\cap(\lambda_{1}^{-1}\ker A_{0}^{*});\qquad
y\To\pi y.$$
Then $\widehat{\pi}$ is injective. 
Indeed, let $y\in\ker A_{1}\cap(\mu^{-1}\ker A_{0}^{*})$ with $\widehat{\pi}y=\pi y=0$. 
Then $y\in\overline{\ran A_{0}}$ and $\mu y\in\ker A_{0}^{*}$. 
Since $\overline{\ran A_{0}}\ \bot_{H_{1}}\ker A_{0}^{*}$, 
using that $\mu\geq\mu_{0}$ in the sense of positive definiteness for some $d>0$, we infer
$\mu_{0}\norm{y}_{H_{1}}^{2}\leq\scp{\mu y}{y}_{H_{1}}=0$. Thus 
$$\dim\big(\ker A_{1}\cap(\mu^{-1}\ker A_{0}^{*})\big)
\leq\dim\big(\ker A_{1}\cap(\lambda_{1}^{-1}\ker A_{0}^{*})\big).$$
The other inequality $\geq$ is deduced by symmetry (in $\mu$ and $\lambda_{1}$)
and hence equality holds.

For (iii$_{1}$), 
we use a similar decomposition strategy.
Let $\mu$ be as before and let
\begin{align}
\label{cpt1proof}
\dom A_{1}\cap(\lambda_{1}^{-1}\dom A_{0}^{*})\hookrightarrow H_{1}
\end{align}
be compact. Moreover, let us consider a bounded sequence 
$$(y_{k})_{k}\subset\dom A_{1}\cap(\mu^{-1}\dom A_{0}^{*}),$$
i.e., $(y_{k})_{k}$, $(A_{1}y_{k})_{k}$, $(A_{0}^{*}\,\mu\,y_{k})_{k}$ are bounded.
Similar to \eqref{deco3} we get
\begin{align*}
%\label{deco4}
%\begin{aligned}
\dom\widetilde{A}_{1}
&=\overline{\ran\widetilde{A}_{0}}\oplus_{\widetilde{H}_{1}}(\dom\widetilde{A}_{1}\cap\ker\widetilde{A}_{0}^{*})
=\overline{\ran A_{0}}\oplus_{\widetilde{H}_{1}}\big(\dom A_{1}\cap(\lambda_{1}^{-1}\ker A_{0}^{*})\big),\\
\dom\widetilde{A}_{0}^{*}
&=(\overline{\ran\widetilde{A}_{0}}\cap\dom\widetilde{A}_{0}^{*})\oplus_{\widetilde{H}_{1}}\ker\widetilde{A}_{0}^{*}
=\big(\overline{\ran A_{0}}\cap(\lambda_{1}^{-1}\dom A_{0}^{*})\big)\oplus_{\widetilde{H}_{1}}\lambda_{1}^{-1}\ker A_{0}^{*},
%\end{aligned}
\end{align*}
and $\dom\widetilde{A}_{1}=\dom A_{1}$ and $\dom\widetilde{A}_{0}^{*}=\lambda_{1}^{-1}\dom A_{0}^{*}$ as sets.
Now, we apply these decompositions to $(y_{k})_{k}$.
First, we $\widetilde{H}_{1}$-orthogonally decompose $y_{k}\in\dom A_{1}$ into 
$$y_{k}=u_{k}+v_{k}$$
with 
$$u_{k}\in\overline{\ran A_{0}}\subset\ker A_{1},\quad
v_{k}\in\dom A_{1}\cap(\lambda_{1}^{-1}\ker A_{0}^{*}),\quad
A_{1}y_{k}=A_{1}v_{k}.$$
Therefore $(v_{k})_{k}$ is bounded in 
$\dom A_{1}\cap(\lambda_{1}^{-1}\ker A_{0}^{*})$ and
by \eqref{cpt1proof} we can extract a $H_{1}$-converging subsequence, again denoted by $(v_{k})_{k}$.
Second, we $\widetilde{H}_{1}$-orthogonally decompose $\lambda_{1}^{-1}\mu y_{k}\in\lambda_{1}^{-1}\dom A_{0}^{*}$ into 
$$\lambda_{1}^{-1}\mu y_{k}=w_{k}+z_{k}$$ 
with
$$w_{k}\in\overline{\ran A_{0}}\cap(\lambda_{1}^{-1}\dom A_{0}^{*})
\subset\ker A_{1}\cap(\lambda_{1}^{-1}\dom A_{0}^{*}),\;
z_{k}\in\lambda_{1}^{-1}\ker A_{0}^{*},\;
A_{0}^{*}\mu y_{k}=A_{0}^{*}\lambda_{1}w_{k}.$$
Hence $(w_{k})_{k}$ is bounded in 
$\ker A_{1}\cap(\lambda_{1}^{-1}\dom A_{0}^{*})$ and
by \eqref{cpt1proof} we can extract an $H_{1}$-converging subsequence, again denoted by $(w_{k})_{k}$.
Finally, again by $H_{1}$-orthogonality, i.e.,
$u_{k}\in\overline{\ran A_{0}}\,\bot_{H_{1}}\ker A_{0}^{*}\ni\lambda_{1}z_{k}$,
\begin{align*}
\bscp{\mu(y_{k}-y_{l})}{y_{k}-y_{l}}_{H_{1}}
&=\bscp{\mu(y_{k}-y_{l})}{u_{k}-u_{l})}_{H_{1}}
+\bscp{\mu(y_{k}-y_{l})}{v_{k}-v_{l}}_{H_{1}}\\
&=\bscp{\lambda_{1}(w_{k}-w_{l})}{u_{k}-u_{l}}_{H_{1}}
+\bscp{\mu(y_{k}-y_{l})}{v_{k}-v_{l}}_{H_{1}}\\
&\leq c\big(\norm{w_{k}-w_{l}}_{H_{1}}+\norm{v_{k}-v_{l}}_{H_{1}}\big)
\end{align*}
for some $c>0$ independently of $k,l$,
which shows that $(y_{k})_{k}$ is an $H_{1}$-Cauchy sequence in $H_{1}$.
Thus $\dom A_{1}\cap(\mu^{-1}\dom A_{0}^{*})\hookrightarrow H_{1}$ is compact.
\end{proof}

Now we can formulate the counterparts of Lemma \ref{lem:properties4} and Theorem \ref{thm:index}.
The proofs follow immediately by Lemma \ref{lambdaindeplem}.

\begin{lemma}
\label{lem:properties4Atilde} 
Maximal compactness does not depend on the weights $\lambda_{l}$. More precisely:
$(A_{0},A_{1},A_{2})$ is a maximal compact Hilbert complex, if and only if the Hilbert complex
$(\widetilde{A}_{0},\widetilde{A}_{1},\widetilde{A}_{2})$ is maximal compact. 
In either case, $\dom\widetilde{\calD}\hookrightarrow\widetilde{H}_{2}\times\widetilde{H}_{0}$ 
and $\dom\widetilde{\calD}^{*}\hookrightarrow\widetilde{H}_{3}\times\widetilde{H}_{1}$ 
are compact.
\end{lemma}

\begin{theorem}
\label{thm:indexAtilde}
The Fredholm indices do not depend on the weights $\lambda_{l}$. More precisely:
Let $(A_{0},A_{1},A_{2})$ be a maximal compact Hilbert complex. 
Then $\calD$, $\widetilde{\calD}$, $\calD^{*}$, and $\widetilde{\calD}^{*}$ are Fredholm operators with indices 
$$\ind\widetilde{\calD}=\ind\calD=\dim N_{0}-\dim K_{1}+\dim K_{2}-\dim N_{2,*},\quad
\ind\widetilde{\calD}^{*}=\ind\calD^{*}=-\ind\calD.$$
\end{theorem}

\section{The De Rham Complex and Its Indices}
\label{sec:derham}

As a first application of our abstract findings, in this section, 
we specialise to a particular choice of the operators $A_{0}$, $A_{1}$, $A_{2}$. 
Also, we will show that the assumptions of Theorem \ref{thm:index} are satisfied for this particular choice of operators. 
We will, thus, obtain an index formula. 
The computations of the dimensions of the occurring cohomology groups date back to \cite{P1982}.

\begin{definition} 
Let $\om\subset\R^3$ be an open set. We put
\begin{align*}
\grad_{c}:\Cicom\subset\Ltom&\To\Lttom,
&
\phi&\longmapsto\grad\phi,\\
\curl_{c}:\Cictom\subset\Lttom&\To\Lttom, 
&
\Phi&\longmapsto\curl\Phi,\\
\dive_{c}:\Cictom\subset\Lttom&\To\Ltom,
& 
\Phi&\longmapsto\dive\Phi,
\end{align*}
and further define the densely defined and closed linear operators
\begin{align*}
\grad&:=-\dive_{c}^{*},
&
\curl&:=\curl_{c}^{*},
&
\dive&:=-\grad_{c}^{*},\\
\rgrad&:=-\dive^{*}=\overline{\grad_{c}},
&
\rcurl&:=\curl^{*}=\overline{\curl_{c}},
&
\rdive&:=-\grad^{*}=\overline{\dive_{c}}.
\end{align*}
\end{definition}

In terms of classical definitions and notions, we record the following equalities (that are easily seen):
\begin{align*}
\dom(\grad) & =  H^1(\om),
 &
\dom(\rgrad) & = \overline{\Cicom}^{H^1(\om)} = H_{0}^1(\om),\\
\dom(\curl) & =   H(\curl,\om) ,
 &
\dom(\rcurl) & =  \overline{\Cictom}^{H(\curl,\om)} = H_{0}(\curl,\om), \\
\dom(\dive) & = H(\dive,\om),
 &
\dom(\rdive) & =  \overline{\Cictom}^{H(\dive,\om)} = H_{0}(\dive,\om).
\end{align*}

\subsection{Picard's Extended Maxwell System}
\label{sec:pems}

We want to apply the index theorem in the following situation of the classical de Rham complex:
\begin{align*}
A_{0}&:=\rgrad,
&
A_{1}&:=\rcurl,
&
A_{2}&:=\rdive,\\
A_{0}^{*}&\phantom{:}=-\dive,
&
A_{1}^{*}&\phantom{:}=\curl,
&
A_{2}^{*}&\phantom{:}=-\grad,
\end{align*}
\begin{align*}
\calD^{\rhmp}
:=\begin{pmatrix}
A_{2} & 0\\ A_{1}^{*} & A_{0}
\end{pmatrix}
&=\begin{pmatrix}
\rdive & 0\\
\curl &\rgrad
\end{pmatrix},
&
(\calD^{\rhmp})^{*}
=\begin{pmatrix}A_{2}^{*}& A_{1}\\0&A_{0}^{*}\end{pmatrix}
&=\begin{pmatrix}-\grad&\rcurl\\0&-\dive\end{pmatrix},
\end{align*}
\begin{align}
\label{derhamcompl2}
\footnotesize
\begin{aligned}
\{0\}\xrightarrow{A_{-1}=\iota_{\{0\}}}
\Ltom\xrightarrow{A_{0}=\rgrad}
\Lttom&\xrightarrow{A_{1}=\rcurl}
\Lttom\xrightarrow{A_{2}=\rdive}
\Ltom\xrightarrow{A_{3}=\pi_{\R_{\pw}}}\R_{\pw},\\
\{0\}\xleftarrow{A_{-1}^{*}=\pi_{\{0\}}}
\Ltom\xleftarrow{A_{0}^{*}=-\dive}
\Lttom&\xleftarrow{A_{1}^{*}=\curl}
\Lttom\xleftarrow{A_{2}^{*}=-\grad}
\Ltom\xleftarrow{A_{3}^{*}=\iota_{\R_{\pw}}}\R_{\pw}.
\end{aligned}
\end{align} 
We note
\begin{align*}
\dom\calD^{\rhmp}
&=(\dom A_{2}\cap\dom A_{1}^{*})\times\dom A_{0}
=\big(H_{0}(\dive,\om)\cap H(\curl,\om)\big)\times H^{1}_{0}(\om),\\
\dom(\calD^{\rhmp})^{*}
&=\dom A_{2}^{*}\times(\dom A_{1}\cap\dom A_{0}^{*})
=H^{1}(\om)\times\big(H_{0}(\curl,\om)\cap H(\dive,\om)\big).
\end{align*}
The complex properties, i.e., $A_{1}A_{0}\subset0$ and $A_{2}A_{1}\subset0$,
are based on Schwarz's lemma ensuring that 
$\curl_{c}\grad_{c}=0$ and $\dive_{c}\curl_{c}=0$.

\begin{proposition}
\label{prop:cP} 
Let $\om\subset\R^3$ be open. Then
\begin{align*}
\ran A_{0}=\ran(\rgrad)&\subset\ker(\rcurl)=\ker A_{1},\\
\ran A_{1}=\ran(\rcurl)&\subset\ker(\rdive)=\ker A_{2}
\end{align*}
and by Remark \ref{rem:dualcomplex} the same holds for the adjoints (operators without homogeneous boundary conditions).
\end{proposition}

\begin{proof}
See, e.g., \cite[Proposition 6.1.5]{STW2020}.
\end{proof}

\begin{theorem}[Picard--Weber--Weck selection theorem, \cite{P1984,W1980,W1974}]
\label{thm:PWS} 
Let $\om\subset\R^3$ be a bounded weak\footnote{The boundary
of a weak Lipschitz domain is a $2$-dimensional submanifold 
of the $3$-dimensional Lipschitz manifold $\overline{\om}$ with boundary.} 
Lipschitz domain. Then 
\begin{align*}
\dom A_{1}\cap\dom A_{0}^{*}&=\dom(\rcurl)\cap\dom(\dive),\\
\dom A_{2}\cap\dom A_{1}^{*}&=\dom({\rdive})\cap\dom(\curl)
\end{align*}
are both compactly embedded into $H_{1}=H_{2}=\Lttom$.
\end{theorem}

\begin{remark}
\label{maxcptcplx:derham}
Proposition \ref{prop:cP} in conjunction with Theorem \ref{thm:PWS} and Rellich's selection theorems 
show that $(\rgrad,\rcurl,\rdive)$ is a maximal compact complex.
By Remark \ref{rem:dualcomplex} so is the dual complex $(-\grad,\curl,-\dive)$.
\end{remark}

Note that
\begin{align}
\label{deRhamdim1}
\begin{aligned}
N_{0}^{\rhmp}&=\ker A_{0}=\ker(\rgrad),\\
N_{2,*}^{\rhmp}&=\ker A_{2}^{*} =\ker(\grad),\\
K_{1}^{\rhmp}&=\ker A_{1} \cap\ker A_{0}^{*}=\ker(\rcurl)\cap\ker(\dive)=:\harmDom,\\
K_{2}^{\rhmp}&=\ker A_{2} \cap\ker A_{1}^{*}=\ker(\rdive)\cap\ker(\curl)=:\harmNom,
\end{aligned}
\end{align}
where we recall from the introduction the classical harmonic Dirichlet and Neumann fields
$\harmDom$ and $\harmNom$, respectively.

\begin{definition}
\label{def:topology} 
Let $\om\subset\R^3$ be bounded and open. Then we denote by 
\begin{itemize}
\item
$n$ the number of connected components of $\om$,
\item
$m$ the number of connected components of the complement $\R^3\setminus\overline{\om}$,
\item
$p$ the number of handles of $\om$, see Assumption \ref{ass:curvesandsurfaces}.
\end{itemize}
For $p$ to be well-defined we suppose Assumption \ref{ass:curvesandsurfaces} to hold.
\end{definition}

The dimensions of the cohomology groups are given as follows.
 
\begin{theorem}[{\cite[Theorem 1]{P1982}}]
\label{thm:DNF} 
Let $\om\subset\R^3$ be open and bounded with continuous boundary. 
Moreover, suppose Assumption \ref{ass:curvesandsurfaces}. Then
$$\dim\harmDom=m-1,\qquad
\dim\harmNom=p.$$
\end{theorem} 

In comparison to \cite[Theorem 1]{P1982} a modified proof of Theorem \ref{thm:DNF} 
is provided in the Sections \ref{app:sec:DirVFdeRham} and \ref{app:sec:NeuVFdeRham}. 
Note that in \cite{P1982} unbounded domains where considered as well, which necessitates a slightly different rationale.

\begin{remark} 
\label{segmentcont}
Note that for $\om$ to have a continuous boundary\footnote{A boundary 
being locally representable as the graph of a continuous function.} 
is equivalent for it to have the segment property, see, e.g., \cite[Remark 7.8 (a)]{AVV2020}.
\end{remark}

Let us introduce the space of piecewise constants by
\begin{equation}
\label{eq:R_pw}
\R_{\pw}:=\big\{u\in\Ltom:\forall\,C \in\cc(\om)\,\exists\,\alpha_{C}\in\R:u|_{C}=\alpha_{C}\big\},
\end{equation}
where
\begin{equation}
\label{eq:ccO}
\cc(\om):=\{C\subset\om:\text{ C is a connected component of }\om\}.
\end{equation}

\begin{theorem}
\label{thm:idx} 
Let $\om\subset\R^3$ be a bounded weak Lipschitz domain. 
Then $\calD^{\rhmp}$ is a Fredholm operator with index
$$\ind\calD^{\rhmp}
=\dim N_{0}^{\rhmp}-\dim K_{1}^{\rhmp}+\dim K_{2}^{\rhmp}-\dim N_{2,*}^{\rhmp}.$$
If additionally $\ga$ is continuous and Assumption \ref{ass:curvesandsurfaces} holds, then
$$\ind\calD^{\rhmp}
=p-m-n+1.$$
\end{theorem}

\begin{proof}
We recall Remark \ref{maxcptcplx:derham}
and apply Theorem \ref{thm:index} together with \eqref{deRhamdim1}, the observations
\begin{align}
\label{kernelsderham}
N_{0}^{\rhmp}=\ker(\rgrad)=\{0\},\qquad
N_{2,*}^{\rhmp}=\ker(\grad)=\R_{\pw},
\end{align}
and Theorem \ref{thm:DNF}.
\end{proof}

\begin{remark}
\label{picardextmaxwell}
By Theorem \ref{thm:index} the adjoint of the de Rham operator $(\calD^{\rhmp})^{*}$
is Fredholm as well with index $\ind(\calD^{\rhmp})^{*}=-\ind\calD^{\rhmp}$.
Moreover, Picard's extended Maxwell system is given by
$$\calM^{\rhmp}
:=\begin{pmatrix}0&\calD^{\rhmp}\\-(\calD^{\rhmp})^{*}&0\end{pmatrix}
=\begin{pmatrix}0&0&A_{2}&0\\0&0&A_{1}^{*}&A_{0}\\
-A_{2}^{*}&-A_{1}&0&0\\0&-A_{0}^{*}&0&0\end{pmatrix}
=\begin{pmatrix}0&0&\rdive&0\\0&0&\curl&\rgrad\\
\grad&-\rcurl&0&0\\0&\dive&0&0\end{pmatrix}$$
with $(\calM^{\rhmp})^{*}=-\calM^{\rhmp}$ and 
$\ind\calM^{\rhmp}=\dim\ker\calM^{\rhmp}-\dim\ker(\calM^{\rhmp})^{*}=0$.
Moreover, $\dim\ker\calM^{\rhmp}=n+m+p-1$ as
\begin{align*}
\ker\calM^{\rhmp}
&=\ker(\calD^{\rhmp})^{*}\times\ker\calD^{\rhmp}\\
&=N_{2,*}^{\rhmp}\times K_{1}^{\rhmp}\times K_{2}^{\rhmp}\times N_{0}^{\rhmp}\\
&=\ker A_{2}^{*}\times(\ker A_{1}\cap\ker A_{0}^{*})\times(\ker A_{2}\cap\ker A_{1}^{*})\times\ker A_{0}\\
&=\ker(\grad)\times\big(\ker(\rcurl)\cap\ker(\dive)\big)\times\big(\ker(\rdive)\cap\ker(\curl)\big)\times\ker(\rgrad)\\
&=\R_{\pw}\times\harmDom\times\harmNom\times\{0\}.
\end{align*}
\end{remark}

\subsection{Variable Coefficients and Poincar\'e--Friedrichs Type Inequalities}
\label{sec:mmorerederham}

The construction of a maximal compact Hilbert complex 
is also possible for mixed boundary conditions
as well as for inhomogeneous and anisotropic media, such as constitutive material laws,
see, e.g., \cite{BPS2016,P2019b,P2020a}.
For mixed boundary conditions we note the following:

\begin{problem}
\label{rem:mixbcderham}
In order to provide a greater variety of index theorems,
it would be interesting to compute the dimensions of the harmonic Dirichlet and Neumann fields 
also in the situation of mixed boundary conditions.
At least for the authors of this article it is completely beyond their expertise in geometry and topology 
and it appears to be an open problem as to which index formulas could be expected
in terms of subcohomologies and related concepts. Note that Fredholmness is guaranteed 
by the compactness result in \cite{BPS2016} in conjunction with Theorem \ref{thm:index}
for a suitably large class of underlying sets and boundaries.
\end{problem}

For inhomogeneous and anisotropic media (constitutive material laws) we have:

\begin{remark}
\label{rem:epsmuderham}
As mentioned before, a maximal compact Hilbert complex can also be constructed 
for inhomogeneous and anisotropic media.
These may be considered as weights $\lambda_{l}$ as presented in Theorem \ref{thm:indexAtilde}.
For Maxwell's equations a typical situation is given by the choices
$\lambda_{0}:=\id$, $\lambda_{3}:=\id$, and
$\lambda_{1}:=\eps,\lambda_{2}:=\mu:\om\to\R^{3\times3}$ 
being symmetric and uniformly positive definite
$L^{\infty}(\om)$-matrix (tensor) fields. Let us introduce the Hilbert spaces
$\Ltt_{\eps}(\om):=\widetilde{H}_{1}:=\big(\Lttom,\scp{\eps\,\cdot\,}{\,\cdot\,}_{\Lttom}\big)$
and similarly $\Ltt_{\mu}(\om):=\widetilde{H}_{2}$
as well as $\widetilde{H}_{0}=\widetilde{H}_{3}=H_{0}=H_{3}=\Ltom$.
We look at
\begin{align*}
\widetilde{A}_{0}&:=\rgrad,
&
\widetilde{A}_{1}&:=\mu^{-1}\rcurl,
&
\widetilde{A}_{2}&:=\rdive\mu,\\
\widetilde{A}_{0}^{*}&\phantom{:}=-\dive\eps,
&
\widetilde{A}_{1}^{*}&\phantom{:}=\eps^{-1}\curl,
&
\widetilde{A}_{2}^{*}&\phantom{:}=-\grad,
\end{align*}
\begin{align*}
\widetilde{\calD}^{\rhmp}&:=
\begin{pmatrix}
\widetilde{A}_{2}&0\\\widetilde{A}_{1}^{*}&\widetilde{A}_{0}
\end{pmatrix}
=\begin{pmatrix}
\rdive\mu & 0\\
\eps^{-1}\curl &\rgrad
\end{pmatrix},\\
(\widetilde{\calD}^{\rhmp})^{*}
&\phantom{:}=\begin{pmatrix}\widetilde{A}_{2}^{*}&\widetilde{A}_{1}\\0&\widetilde{A}_{0}^{*}\end{pmatrix}
=\begin{pmatrix}-\grad&\mu^{-1}\rcurl\\0&-\dive\eps\end{pmatrix},
\end{align*}
i.e., the de Rham complex, cf. \eqref{derhamcompl2},
\begin{align}
\label{derhamcompl3}
\footnotesize
\begin{aligned}
\{0\}\xrightarrow{\widetilde{A}_{-1}=\iota_{\{0\}}}
\Ltom\xrightarrow{\widetilde{A}_{0}=\rgrad}
\Ltt_{\eps}(\om)&\xrightarrow{\widetilde{A}_{1}=\mu^{-1}\rcurl}
\Ltt_{\mu}(\om)\xrightarrow{\widetilde{A}_{2}=\rdive\mu}
\Ltom\xrightarrow{\widetilde{A}_{3}=\pi_{\R_{\pw}}}\R_{\pw},\\
\{0\}\xleftarrow{\widetilde{A}_{-1}^{*}=\pi_{\{0\}}}
\Ltom\xleftarrow{\widetilde{A}_{0}^{*}=-\dive\eps}
\Ltt_{\eps}(\om)&\xleftarrow{\widetilde{A}_{1}^{*}=\eps^{-1}\curl}
\Ltt_{\mu}(\om)\xleftarrow{\widetilde{A}_{2}^{*}=-\grad}
\Ltom\xleftarrow{\widetilde{A}_{3}^{*}=\iota_{\R_{\pw}}}\R_{\pw}.
\end{aligned}
\end{align} 
Lemma \ref{lambdaindeplem}, Lemma \ref{lem:properties4Atilde}, and Theorem \ref{thm:indexAtilde}
show that the compactness properties, the dimensions of the kernels and cohomology groups,
the maximal compactness, and the Fredholm indices
of the de Rham complex do not depend on the material weights $\eps$ and $\mu$.
More precisely,
\begin{itemize}
\item
$\dim\big(\ker(\rcurl)\cap\big(\eps^{-1}\ker(\dive)\big)\big)
\!=\!\dim\big(\ker(\rcurl)\cap\ker(\dive)\big)
\!=\!\dim\harmDom
=m-1$,
\item
$\dim\big(\big(\mu^{-1}\ker(\rdive)\big)\cap\ker(\curl)\big)
=\dim\big(\ker(\rdive)\cap\ker(\curl)\big)
=\dim\harmNom
=p$,
\item
$\dom(\rcurl)\cap\big(\eps^{-1}\dom(\dive)\big)
\hookrightarrow\Ltt_{\eps}(\om)$ compactly\\
$\Leftrightarrow\;\dom(\rcurl)\cap\dom(\dive)
\hookrightarrow\Lttom$ compactly,
\item
$\big(\mu^{-1}\dom(\rdive)\big)\cap\dom(\curl)
\hookrightarrow\Ltt_{\mu}(\om)$ compactly\\
$\Leftrightarrow\;\dom(\rdive)\cap\dom(\curl)
\hookrightarrow\Lttom$ compactly,
\item
$(\rgrad,\mu^{-1}\rcurl,\rdive\mu)$ is maximal compact iff
$(\rgrad,\rcurl,\rdive)$ is maximal compact,\item$-\ind(\widetilde{\calD}^{\rhmp})^{*}=\ind\widetilde{\calD}^{\rhmp}
=\ind\calD^{\rhmp}=p-m-n+1$.
\end{itemize}
\end{remark}

At this point, see Lemma \ref{lem:properties3}, Corollary \ref{lem:properties3a}, and \eqref{kernelsderham},
we note that the kernels and ranges are given by
\begin{align*}
\ker\calD^{\rhmp}
=K_{2}^{\rhmp}\times N_{0}^{\rhmp}
&=\harmNom\times\{0\},\\
\ker(\calD^{\rhmp})^{*}
=N_{2,*}^{\rhmp}\times K_{1}^{\rhmp}
&=\R_{\pw}\times\harmDom,\\
\ran\calD^{\rhmp}
=(\ker(\calD^{\rhmp})^{*})^{\bot_{\Ltom\times\Lttom}}
&=\R_{\pw}^{\bot_{\Ltom}}\times\harmDom^{\bot_{\Lttom}},\\
\ran(\calD^{\rhmp})^{*}
=(\ker\calD^{\rhmp})^{\bot_{\Lttom\times\Ltom}}
&=\harmNom^{\bot_{\Lttom}}\times\Ltom.
\end{align*}

Finally, Corollary \ref{cor:FP} yields additional results for the corresponding reduced operators 
\begin{align*}
\calD^{\rhmp}_{\red}
=\calD^{\rhmp}|_{(\ker\calD^{\rhmp})^{\bot_{H_{2}\times H_{0}}}}
&=\begin{pmatrix}\rdive&0\\\curl&\rgrad\end{pmatrix}\Big|_{\harmNom^{\bot_{\Lttom}}\times\Ltom},\\
(\calD^{\rhmp}_{\red})^{*}
=(\calD^{\rhmp})^{*}|_{(\ker(\calD^{\rhmp})^{*})^{\bot_{H_{3}\times H_{1}}}}
&=\begin{pmatrix}-\grad&\rcurl\\0&-\dive\end{pmatrix}\Big|_{\R_{\pw}^{\bot_{\Ltom}}\times\harmDom^{\bot_{\Lttom}}}.
\end{align*}

\begin{corollary}
\label{cor:FP-tb}
Let $\om\subset\R^3$ be a bounded weak Lipschitz domain with continuous boundary. Then
\begin{align*}
(\calD^{\rhmp}_{\red})^{-1}:\ran\calD^{\rhmp}&\to\ran(\calD^{\rhmp})^{*},\\
((\calD^{\rhmp}_{\red})^{*})^{-1}:\ran(\calD^{\rhmp})^{*}&\to\ran\calD^{\rhmp}
\intertext{are compact. Furthermore,}
(\calD^{\rhmp}_{\red})^{-1}:\ran\calD^{\rhmp}&\to\dom\calD^{\rhmp}_{\red},\\
((\calD^{\rhmp}_{\red})^{*})^{-1}:\ran(\calD^{\rhmp})^{*}&\to\dom(\calD^{\rhmp}_{\red})^{*}
\end{align*}
are continuous and, equivalently, the Friedrichs--Poincar\'e type estimate
\begin{align*}
\bnorm{(E,u)}_{\Lttom\times\Ltom}
&\leq c_{\calD^{\rhmp}}\big(\norm{\grad u}_{\Lttom}^{2}+\norm{\dive E}_{\Ltom}^{2}+\norm{\curl E}_{\Lttom}^{2}\big)^{1/2}
\end{align*}
holds for all $(E,u)$ in
\begin{align*}
\dom\calD^{\rhmp}_{\red}
&=\big(H_{0}(\dive,\om)\cap H(\curl,\om)\cap\harmNom^{\bot_{\Lttom}}\big)
\times H^{1}_{0}(\om)
\intertext{or $(u,E)$ in}
\dom(\calD^{\rhmp}_{\red})^{*}
&=\big(H^{1}(\om)\cap\R_{\pw}^{\bot_{\Ltom}}\big)
\times\big(H_{0}(\curl,\om)\cap H(\dive,\om)\cap\harmDom^{\bot_{\Lttom}}\big)
\end{align*}
with some optimal constant $c_{\calD^{\rhmp}}>0$.
\end{corollary}

Note that the latter estimate is an additive combination of 
the well-known Friedrichs--Poincar\'e estimates for $\grad$
and the well-known Maxwell estimates for $(\curl,\dive)$. 

\subsection{The Dirac Operator}
\label{sec:dirac}

In this section, we flag up a relationship of the Dirac operator and Picard's extended Maxwell system.
Let the assumptions of Theorem \ref{thm:idx} be satisfied.
The extended Maxwell operator is an operator that is surprisingly close to the Dirac operator, see \cite{PTW2017}. 
We shall carry out this construction in the following. 
Recall from Remark \ref{picardextmaxwell}
that Picard's extended Maxwell system is given by the operator 
$$\calM
:=\begin{pmatrix}0&\calD\\-\calD^{*}&0\end{pmatrix},\qquad
\calD:=\calD^{\rhmp}.$$
Next, we introduce the Dirac operator. For this, we define the Pauli matrices
\[
 \sigma_{1}:=\begin{pmatrix} 0 & 1\\ 1 & 0\end{pmatrix},\quad  
 \sigma_{2}:=\begin{pmatrix} 0 & -i\\ i & 0\end{pmatrix},\quad  
 \sigma_{3}:=\begin{pmatrix} 1 & 0\\ 0 & -1\end{pmatrix}.
\]
Setting 
\begin{align*}
  \calQ\colon\dom\calQ\subset L^{2,2}(\om) &\To L^{2,2}(\om)\\
  \psi \hspace*{25mm}&\longmapsto\sum_{j=1}^3\p_{j}\sigma_{j}\psi
  =\begin{pmatrix}\p_{3}&\p_{1}-i\p_{2}\\\p_{1}+i\p_{2}&-\p_{3}\end{pmatrix}\psi, 
\end{align*}
we define the Dirac operator
$$  \calL:=\begin{pmatrix} 0 &\calQ\\ -\calQ^{*} & 0\end{pmatrix}. $$
We have not specified the domain of definition of $\calQ$, yet. 
For now, we only record that $C^{\infty,2}_{c}(\om)\subset\dom\calQ$,
and the domain of definition of $\calQ$ corresponding to $\calM$ is provided below, see also Proposition \ref{prop:real}.
We introduce the unitary operators from $L^{2,4}(\om)$ into itself
\[ 
W:= \begin{pmatrix} 0 & 0 & -1 & 0\\ 0 & 0 & 0 & -1\\ -1 & 0 & 0 & 0\\ 0 & 1 & 0 & 0\end{pmatrix},\qquad
U:= \begin{pmatrix} 0 & 1 & 0 & 0\\ 0 & 0 & 1 & 0\\ 0 & 0 & 0 & 1\\ 1 & 0 & 0 & 0\end{pmatrix}.
\]
Then the operators $\calL$ (Dirac operator) 
and $\calM$ (Picard's extended Maxwell operator) 
are unitarily equivalent. More precisely, we have with $V$ from Proposition \ref{prop:real}
\[
  \calM 
  =\begin{pmatrix} U & 0\\ 0 & W\end{pmatrix}
  \begin{pmatrix} V & 0\\ 0 & V\end{pmatrix}
  \calL 
  \begin{pmatrix} V^{*} & 0\\ 0 & V^{*}\end{pmatrix}
  \begin{pmatrix} U^{*} & 0\\ 0 & W^{*}\end{pmatrix},
\]
$$\dom\calQ^{*}\times\dom\calQ 
:=\begin{pmatrix} V^{*} & 0\\ 0 & V^{*}\end{pmatrix}
\begin{pmatrix} U^{*} & 0\\ 0 & W^{*}\end{pmatrix}
\big(\dom\calD^{*}\times\dom\calD\big)
\begin{pmatrix} U & 0\\ 0 & W\end{pmatrix}
\begin{pmatrix} V & 0\\ 0 & V\end{pmatrix}$$
and, consequently, 
$\calQ$ with domain $\dom(V^{*}U^{*}\calD WV)=\dom(\calD WV)$ is a Fredholm operator. 
Moreover, we have $\ind\calL=0$ and
$$\ind\calQ=\ind\calD=p-m-n+1.$$

We conclude this section by stating the missing proposition used above. 
The proofs of which are straightforward and will therefore be omitted. 
In a slightly similar fashion, they can be found \cite{PTW2017}. 
For the next result we use $L^{2}_{\R}(\om)$ and $L^{2}_{\C}(\om)$ 
to denote the Hilbert space $\Ltom$ with the reals and the complex numbers 
as respective underlying field.

\begin{proposition}[Realification of $\calL$]
\label{prop:real}
It holds:
\begin{enumerate}
\item[\bf(i)]
$V:\Lt_{\C}(\om)\to L^{2,2}_{\R}(\om)$ with 
$Vf:=(\Re f,\Im f)$ is unitary. 
\item[\bf(ii)]
$V i V^{*} =\begin{pmatrix} 0 & 1\\ -1 & 0\end{pmatrix}$.
\item[\bf(iii)] 
$\widetilde{\calQ}
:= V\calQ V^{*} 
=\p_{1} \begin{pmatrix} 0 & 0 & 1 & 0\\ 0 & 0 & 0 & 1\\ 1 & 0 & 0 & 0\\ 0 & 1 & 0 & 0\end{pmatrix} 
+\p_{2}\begin{pmatrix} 0 & 0 & 0 & -1\\ 0 & 0 & 1 & 0\\ 0 & 1 & 0 & 0\\ -1 & 0 & 0 & 0\end{pmatrix} 
+\p_{3}\begin{pmatrix} 1 & 0 & 0 & 0\\ 0 & 1 & 0 & 0\\ 0 & 0 & -1 & 0\\ 0 & 0 & 0 & -1\end{pmatrix}$ 
with $\dom\widetilde{\calQ} = V\dom\calQ\,V^{*}$.
\end{enumerate}
\end{proposition}

\section{The First \PZ Complex and Its Indices}
\label{sec:pz}

In this section, we focus on our first main result and properly define 
the operators involved in the formulation of Theorem \ref{thm:main}. 
Thus, we introduce the first \pz complex (see \cite{PZ2016,PZ2020a}) constructed 
for \pz problems and general relativity, but also relevant in problems for elasticity. 
It will be interesting to see that the differential operator is apparently of mixed order rather than just of first order. 
It is worth noting that the apparently leading order term is \emph{not} 
dominating the lower order differential operators.

\begin{definition} 
\label{def:pzone}
Let $\om\subset\R^3$ be an open set. We put
\begin{align*}
\Gg_{c}:\Cicom\subset\Ltom&\To\LtttSom,
&
\phi&\longmapsto\Gg\phi,\\
\Curl_{c}:\CicttSom\subset\LtttSom&\To\LtttTom, 
&
\Phi&\longmapsto\Curl\Phi,\\
\Dive_{c}:\CicttTom\subset\LtttTom&\To\Lttom,
& 
\Phi&\longmapsto\Dive\Phi,
\end{align*}
and further define the densely defined and closed linear operators
\begin{align*}
\dDS&:=\Gg_{c}^{*},
&
\rGg&:=\dDS^{*}=\overline{\Gg_{c}},\\
\symCT&:=\Curl_{c}^{*},
&
\rCS&:=\symCT^{*}=\overline{\Curl_{c}},\\
\devGrad&:=-\Dive_{c}^{*},
&
\rDT&:=-\devGrad^{*}=\overline{\Dive_{c}}.
\end{align*}
\end{definition}

We shall apply the index theorem in the following situation of the first \pz complex:
\begin{align*}
A_{0}&:=\rGg,
&
A_{1}&:=\rCS,
&
A_{2}&:=\rDT,\\
A_{0}^{*}&\phantom{:}=\dDS,
&
A_{1}^{*}&\phantom{:}=\symCT,
&
A_{2}^{*}&\phantom{:}=-\devGrad,
\end{align*}
\begin{align*}
\calD^{\bihonep}
:=\begin{pmatrix}A_{2}&0\\A_{1}^{*}&A_{0}\end{pmatrix}
&=\begin{pmatrix}\rDT& 0\\\symCT&\rGg\end{pmatrix},\\
(\calD^{\bihonep})^{*}
=\begin{pmatrix}A_{2}^{*}& A_{1}\\0&A_{0}^{*}\end{pmatrix}
&=\begin{pmatrix}-\devGrad&\rCS\\0&\dDS\end{pmatrix}.
\end{align*}
Introducing the space of piecewise Raviart--Thomas fields by
\begin{equation}
\label{eq:def_RT_pw}
\RT_{\pw}:=\big\{v\in\Lttom:\forall\,C\in\cc(\om)\quad\exists\,\alpha_{C}\in\R,\,\beta_{C}\in\R^{3}:
u|_{C}(x)=\alpha_{C}x+\beta_{C}\big\},
\end{equation}
for $\cc(\om)$ see \eqref{eq:ccO}, we can write the first \pz complex as
\begin{align}
\footnotesize
\begin{aligned}
\label{bih1compl1}
\{0\}\xrightarrow{\iota_{\{0\}}}
\Ltom\xrightarrow{\rGg}
\LtttSom&\xrightarrow{\rCS}
\LtttTom\xrightarrow{\rDT}
\Lttom\xrightarrow{\pi_{\RT_{\pw}}}\RT_{\pw},\\
\{0\}\xleftarrow{\pi_{\{0\}}}
\Ltom\xleftarrow{\dDS}
\LtttSom&\xleftarrow{\symCT}
\LtttTom\xleftarrow{-\devGrad}
\Lttom\xleftarrow{\iota_{\RT_{\pw}}}\RT_{\pw}.
\end{aligned}
\end{align}

The foundation of the index theorem to hold is the following compactness result 
established by Pauly and Zulehner. Note that it holds
$\dom(\rGg)=H^{2}_{0}(\om)$ and $\dom(\devGrad)=H^{1,3}(\om)$. 

\begin{theorem}[{\cite[Lemma 3.22, Theorem 3.23]{PZ2020a}}]
\label{thm:lHc} 
Let $\om\subset\R^3$ be a bounded strong Lipschitz domain. 
Then $(\rGg,\rCS,\rDT)$ is a maximal compact Hilbert complex.
\end{theorem}

We observe and define
\begin{align} 
\label{bihdim1}
\begin{aligned}
N_{0}^{\bihonep}&=\ker A_{0}=\ker(\rGg),\\
N_{2,*}^{\bihonep}&=\ker A_{2}^{*} =\ker(\devGrad),\\
K_{1}^{\bihonep}&=\ker A_{1} \cap\ker A_{0}^{*}=\ker(\rCS)\cap\ker(\dDS)=:\harmbihoneDom,\\
K_{2}^{\bihonep}&=\ker A_{2} \cap\ker A_{1}^{*}=\ker(\rDT)\cap\ker(\symCT)=:\harmbihoneNom.
\end{aligned}
\end{align}

The dimensions of the cohomology groups are given as follows.
 
\begin{theorem}
\label{thm:DNTbih} 
Let $\om\subset\R^3$ be open and bounded with continuous boundary. 
Moreover, suppose Assumption \ref{ass:curvesandsurfaces}. Then
$$\dim\harmbihoneDom=4(m-1),\qquad
\dim\harmbihoneNom=4p.$$
\end{theorem} 

\begin{proof}
We postpone the proof to Sections \ref{app:sec:DirTFbih1} and \ref{app:sec:NeuTFbih}.
\end{proof}

The proper formulation of the first main result, Theorem \ref{thm:main}, reads as follows.

\begin{theorem}
\label{thm:idxPZC} 
Let $\om\subset\R^3$ be a bounded strong Lipschitz domain. 
Then $\calD^{\bihonep}$ is a Fredholm operator with index
$$\ind\calD^{\bihonep}
=\dim N_{0}^{\bihonep}-\dim K_{1}^{\bihonep}+\dim K_{2}^{\bihonep}-\dim N_{2,*}^{\bihonep}.$$
If additionally Assumption \ref{ass:curvesandsurfaces} holds, then
$$\ind\calD^{\bihonep}
=4(p-m-n+1).$$
\end{theorem}

\begin{proof}
Using Theorem \ref{thm:lHc}, we apply Theorem \ref{thm:index} together with \eqref{bihdim1} and the observations
\begin{align}
\label{kernelsbih}
N_{0}^{\bihonep}=\ker(\rGg)=\{0\},\qquad
N_{2,*}^{\bihonep}=\ker(\devGrad)=\RT_{\pw},
\end{align}
see \cite[Lemma 3.2, Lemma 3.3]{PZ2020a}, and Theorem \ref{thm:DNTbih}.
\end{proof}

\begin{remark}
\label{thm:idxPZC:rem}
By Theorem \ref{thm:index} the adjoint $(\calD^{\bihonep})^{*}$
is Fredholm as well with index simply given by $\ind(\calD^{\bihonep})^{*}=-\ind\calD^{\bihonep}$.
Similar to Remark \ref{picardextmaxwell} we define the extended first \pz operator
$$\calM^{\bihonep}
:=\begin{pmatrix}0&\calD^{\bihonep}\\-(\calD^{\bihonep})^{*}&0\end{pmatrix}
=\begin{pmatrix}0&0&\rDT&0\\0&0&\symCT&\rGg\\
\devGrad&-\rCS&0&0\\0&-\dDS&0&0\end{pmatrix}$$
with $(\calM^{\bihonep})^{*}=-\calM^{\bihonep}$ and $\ind\calM^{\bihonep}=0$.
Moreover, $\dim\ker\calM^{\bihonep}=4(n+m+p-1)$ as
$\ker\calM^{\bihonep}
=\RT_{\pw}\times\harmbihoneDom\times\harmbihoneNom\times\{0\}$.
\end{remark}

%\subsection*{Variable Coefficients and Poincar\'e--Friedrichs Type Inequalities}
\noindent
\textbf{Variable Coefficients and Poincar\'e--Friedrichs Type Inequalities.}
\label{sec:mmorerebih1} 
Inhomogeneous and anisotropic media may also be considered 
for the first \pz complex, cf. Remark \ref{rem:epsmuderham}.

\begin{remark}
\label{rem:epsmubih1}
Let $\lambda_{0}:=\id$, $\lambda_{3}:=\id$,
$\lambda_{1}:=\eps:\om\to\R^{3\times3\times3\times3}$, and $\lambda_{2}:=\mu:\om\to\R^{3\times3\times3\times3}$ 
be $L^{\infty}(\om)$-tensor fields such that the induced respective operators on 
$\Lttt_{\Sb}(\om)$ and $\Lttt_{\Tb}(\om)$ are symmetric and strictly positive definite. 
Moreover, let us introduce
$$\Lttt_{\Sb,\eps}(\om):=\widetilde{H}_{1}:=\big(\LtttSom,\scp{\eps\,\cdot\,}{\,\cdot\,}_{\LtttSom}\big)$$
and similarly $\Lttt_{\Tb,\mu}(\om):=\widetilde{H}_{2}$
as well as $\widetilde{H}_{0}=H_{0}=\Ltom$, $\widetilde{H}_{3}=H_{3}=\Lttom$.
We look at
\begin{align*}
\widetilde{A}_{0}&:=\rGg,
&
\widetilde{A}_{1}&:=\mu^{-1}\rCS,
&
\widetilde{A}_{2}&:=\rDT\mu,\\
\widetilde{A}_{0}^{*}&\phantom{:}=\dDS\eps,
&
\widetilde{A}_{1}^{*}&\phantom{:}=\eps^{-1}\symCT,
&
\widetilde{A}_{2}^{*}&\phantom{:}=-\devGrad,
\end{align*}
\begin{align*}
\widetilde{\calD}^{\bihone}&:=
\begin{pmatrix}
\widetilde{A}_{2}&0\\\widetilde{A}_{1}^{*}&\widetilde{A}_{0}
\end{pmatrix}
=\begin{pmatrix}
\rDT\mu & 0\\
\eps^{-1}\symCT &\rGg
\end{pmatrix},\\
(\widetilde{\calD}^{\bihone})^{*}
&\phantom{:}=\begin{pmatrix}\widetilde{A}_{2}^{*}&\widetilde{A}_{1}\\0&\widetilde{A}_{0}^{*}\end{pmatrix}
=\begin{pmatrix}-\devGrad&\mu^{-1}\rCS\\0&\dDS\eps\end{pmatrix},
\end{align*}
i.e., the first \pz complex, cf. \eqref{bih1compl1},
\begin{align}
\footnotesize
\begin{aligned}
\label{bih1compl2}
\{0\}\xrightarrow{\iota_{\{0\}}}
\Ltom\xrightarrow{\rGg}
\Lttt_{\Sb,\eps}(\om)&\xrightarrow{\mu^{-1}\rCS}
\Lttt_{\Tb,\mu}(\om)\xrightarrow{\rDT\mu}
\Lttom\xrightarrow{\pi_{\RT_{\pw}}}\RT_{\pw},\\
\{0\}\xleftarrow{\pi_{\{0\}}}
\Ltom\xleftarrow{\dDS\eps}
\Lttt_{\Sb,\eps}(\om)&\xleftarrow{\eps^{-1}\symCT}
\Lttt_{\Tb,\mu}(\om)\xleftarrow{-\devGrad}
\Lttom\xleftarrow{\iota_{\RT_{\pw}}}\RT_{\pw}.
\end{aligned}
\end{align}
Lemma \ref{lambdaindeplem}, Lemma \ref{lem:properties4Atilde}, and Theorem \ref{thm:indexAtilde}
show that the compactness properties, the dimensions of the kernels and cohomology groups,
the maximal compactness, and the Fredholm indices
of the first \pz complex do not dependent of the material weights $\eps$ and $\mu$.
More precisely,
\begin{align*}
&\bullet&
\dim\big(\ker(\rCS)\cap\big(\eps^{-1}\ker(\dDS)\big)\big)
&=\dim\big(\ker(\rCS)\cap\ker(\dDS)\big)\\
&&&=\dim\harmbihoneDom
=4(m-1),\\
&\bullet&
\dim\big(\big(\mu^{-1}\ker(\rDT)\big)\cap\ker(\symCT)\big)
&=\dim\big(\ker(\rDT)\cap\ker(\symCT)\big)\\
&&&=\dim\harmbihoneNom
=4p,\\
&\bullet&
\dom(\rCS)\cap\big(\eps^{-1}\dom(\dDS)\big)
&\hookrightarrow\Lttt_{\Sb,\eps}(\om)\text{ compactly}\\
&&\Leftrightarrow\quad
\dom(\rCS)\cap\dom(\dDS)
&\hookrightarrow\LtttSom\text{ compactly},\\
&\bullet&
\big(\mu^{-1}\dom(\rDT)\big)\cap\dom(\symCT)
&\hookrightarrow\Lttt_{\Tb,\mu}(\om)\text{ compactly}\\
&&\Leftrightarrow\quad
\dom(\rDT)\cap\dom(\symCT)
&\hookrightarrow\LtttTom\text{ compactly},\\
&\bullet&
(\rGg,\mu^{-1}\rCS,\rDT\mu)&\text{ maximal compact}\\
&&\Leftrightarrow \quad (\rGg,\rCS,\rDT)&\text{ maximal compact},\\
&\bullet&
-\ind(\widetilde{\calD}^{\bihone})^{*}=\ind\widetilde{\calD}^{\bihone}
=\ind\calD^{\bihone}
&=4(p-m-n+1).
\end{align*}
\end{remark}

Note that the kernels and ranges are given by
\begin{align*}
\ker\calD^{\bihonep}
=K_{2}^{\bihonep}\times N_{0}^{\bihonep}
&=\harmbihoneNom\times\{0\},\\
\ker(\calD^{\bihonep})^{*}
=N_{2,*}^{\bihonep}\times K_{1}^{\bihonep}
&=\RT_{\pw}\times\harmbihoneDom,\\
\ran\calD^{\bihonep}
=(\ker(\calD^{\bihonep})^{*})^{\bot_{\Lttom\times\LtttSom}}
&=\RT_{\pw}^{\bot_{\Lttom}}\times\harmbihoneDom^{\bot_{\LtttSom}},\\
\ran(\calD^{\bihonep})^{*}
=(\ker\calD^{\bihonep})^{\bot_{\LtttTom\times\Ltom}}
&=\harmbihoneNom^{\bot_{\LtttTom}}\times\Ltom,
\end{align*}
see Lemma \ref{lem:properties3}, Corollary \ref{lem:properties3a}, and \eqref{kernelsbih}.
Corollary \ref{cor:FP} shows additional results for the corresponding reduced operators 
\begin{align*}
\calD^{\bihonep}_{\red}
=\calD^{\bihonep}|_{(\ker\calD^{\bihonep})^{\bot_{H_{2}\times H_{0}}}}
&=\begin{pmatrix}\rDT& 0\\\symCT&\rGg\end{pmatrix}\Big|_{\harmbihoneNom^{\bot_{\LtttTom}}\times\Ltom},\\
(\calD^{\bihonep}_{\red})^{*}
=(\calD^{\bihonep})^{*}|_{(\ker(\calD^{\bihonep})^{*})^{\bot_{H_{3}\times H_{1}}}}
&=\begin{pmatrix}-\devGrad&\rCS\\0&\dDS\end{pmatrix}\Big|_{\RT_{\pw}^{\bot_{\Lttom}}\times\harmbihoneDom^{\bot_{\LtttSom}}}.
\end{align*}

\begin{corollary}
\label{cor:FP-tb-bih}
Let $\om\subset\R^3$ be a bounded strong Lipschitz domain. Then
\begin{align*}
(\calD^{\bihonep}_{\red})^{-1}:\ran\calD^{\bihonep}&\to\ran(\calD^{\bihonep})^{*},\\
((\calD^{\bihonep}_{\red})^{*})^{-1}:\ran(\calD^{\bihonep})^{*}&\to\ran\calD^{\bihonep}
\intertext{are compact. Furthermore,}
(\calD^{\bihonep}_{\red})^{-1}:\ran\calD^{\bihonep}&\to\dom\calD^{\bihonep}_{\red},\\
((\calD^{\bihonep}_{\red})^{*})^{-1}:\ran(\calD^{\bihonep})^{*}&\to\dom(\calD^{\bihonep}_{\red})^{*}
\end{align*}
are continuous and, equivalently, the Friedrichs--Poincar\'e type estimates 
\begin{align*}
\bnorm{(T,u)}_{\LtttTom\times\Ltom}
&\leq c_{\calD^{\bihonep}}\big(\norm{\Gg u}_{\LtttSom}^{2}\\
&\qquad\qquad\qquad+\norm{\Dive T}_{\Lttom}^{2}+\norm{\symCurl T}_{\LtttSom}^{2}\big)^{1/2},\\
\bnorm{(v,S)}_{\Lttom\times\LtttSom}
&\leq c_{\calD^{\bihonep}}\big(\norm{\devGrad v}_{\LtttTom}^{2}\\
&\qquad\qquad\qquad+\norm{\dD S}_{\Ltom}^{2}+\norm{\Curl S}_{\LtttTom}^{2}\big)^{1/2}
\end{align*}
hold for all $(T,u)$ in
\begin{align*}
\dom\calD^{\bihonep}_{\red}
&=\big(\dom(\rDT)\cap\dom(\symCT)\cap\harmbihoneNom^{\bot_{\LtttTom}}\big)
\times H^{2}_{0}(\om)
\intertext{for all $(v,S)$ in}
\dom(\calD^{\bihonep}_{\red})^{*}
&=\big(H^{1,3}(\om)\cap\RT_{\pw}^{\bot_{\Lttom}}\big)\\
&\qquad\qquad\times\big(\dom(\rCS)\cap\dom(\dDS)\cap\harmbihoneDom^{\bot_{\LtttSom}}\big)
\end{align*}
with some optimal constant $c_{\calD^{\bihonep}}>0$.
\end{corollary}

\section{The Second \PZ Complex and Its Indices}
\label{sec:pz2}

The second major application of the abstract findings 
in the Sections \ref{sec:3}, \ref{sec:3-addstuff}, and \ref{sec:4-addstuff} 
is concerned with the second \pz complex. 
The needed operators are provided next. 
It is worth recalling the definitions of the operators $\devGrad$, $\symCT$, 
and $\dDS$ from Definition \ref{def:pzone}.

\begin{definition} 
Let $\om\subset\R^3$ be an open set. We put
\begin{align*}
\devGrad_{c}:\Cictom\subset\Lttom&\To\LtttTom,
&
\phi&\longmapsto\devGrad\phi,\\
\symCurl_{c}:\CicttTom\subset\LtttTom&\To\LtttSom, 
&
\Phi&\longmapsto\symCT\Phi,\\
\dD_{c}:\CicttSom\subset\LtttSom&\To\Ltom,
& 
\Phi&\longmapsto\dDS\Phi,
\end{align*}
and further define the densely defined and closed linear operators
\begin{align*}
\DT&:=-\devGrad_{c}^{*},
&
\rdevGrad&:=-\DT^{*}=\overline{\devGrad_{c}},\\
\CS&:=\symCurl_{c}^{*},
&
\rsymCT&:=\CS^{*}=\overline{\symCurl_{c}},\\
\Gg&:=\dD_{c}^{*},
&
\rdDS&:=\Gg^{*}=\overline{\dD_{c}}.
\end{align*}
\end{definition}

We shall apply the index theorem in the following situation of the second \pz complex:
\begin{align*}
A_{0}&:=\rdevGrad,
&
A_{1}&:=\rsymCT,
&
A_{2}&:=\rdDS,\\
A_{0}^{*}&\phantom{:}=-\DT,
&
A_{1}^{*}&\phantom{:}=\CS,
&
A_{2}^{*}&\phantom{:}=\Gg,
\end{align*}
\begin{align*}
\calD^{\bihtwop}
&:=\begin{pmatrix}A_{2}&0\\A_{1}^{*}&A_{0}\end{pmatrix}
=\begin{pmatrix}\rdDS&0\\\CS&\rdevGrad\end{pmatrix},\\
(\calD^{\bihtwop})^{*}
&\phantom{:}=\begin{pmatrix}A_{2}^{*}& A_{1}\\0&A_{0}^{*}\end{pmatrix}
=\begin{pmatrix}\Gg&\rsymCT\\0&-\DT\end{pmatrix},
\end{align*}
\begin{align}
\footnotesize
\begin{aligned}
\label{bih2compl1}
\{0\}\xrightarrow{\iota_{\{0\}}}
\Lttom\xrightarrow{\rdevGrad}
\LtttTom&\xrightarrow{\rsymCT}
\LtttSom\xrightarrow{\rdDS}
\Ltom\xrightarrow{\pi_{\Pol^{1}_{\pw}}}\Pol^{1}_{\pw},\\
\{0\}\xleftarrow{\pi_{\{0\}}}
\Lttom\xleftarrow{-\DT}
\LtttTom&\xleftarrow{\CS}
\LtttSom\xleftarrow{\Gg}
\Ltom\xleftarrow{\iota_{\Pol^{1}_{\pw}}}\Pol^{1}_{\pw},
\end{aligned}
\end{align}
where we used the space of piecewise first order polynomials (for $\cc(\om)$ see \eqref{eq:ccO})
\begin{equation}
\label{eq:P1_pw}
\Pol^{1}_{\pw}:=\big\{v\in\Ltom:\forall\,C\in\cc(\om)\quad\exists\,\alpha_{C}\in\R,\,\beta_{C}\in\R^{3}:
u|_{C}(x)=\alpha_{C}+\beta_{C}\cdot x\big\}.
\end{equation}
Note that $\dom(\rdevGrad)=H^{1,3}_{0}(\om)$ by \cite[Lemma 3.2]{PZ2020a}.

\begin{lemma}
\label{lem:regGg}
Let $\om\subset\R^3$ be a bounded strong Lipschitz domain. 
Then the regularity $\dom(\Gg)=H^{2}(\om)$ holds and 
there exists $c>0$ such that for all $u\in H^{2}(\om)$
$$c\,\norm{u}_{H^{2}(\om)}
\leq\norm{u}_{\Ltom}+\norm{\Grad\grad u}_{\Ltttom}.$$
\end{lemma}

\begin{proof}
Let $u\in\dom(\Gg)$. Then $\grad u\in H^{-1,3}(\om)$ and 
$\Grad\grad u\in\Ltttom$. Ne\v{c}as' regularity yields
$\grad u\in\Lttom$ and thus $u\in H^{1}(\om)$ and $\grad u\in H^{1,3}(\om)$.
Hence $u\in H^{2}(\om)$ and by Ne\v{c}as' inequality (see \cite{N2012}) we have
\begin{align*}
\norm{\grad u}_{\Lttom}
&\leq c\big(\norm{\grad u}_{H^{-1,3}(\om)}+\norm{\Grad\grad u}_{H^{-1,3\times3}(\om)}\big)\\
&\leq c\big(\norm{u}_{\Ltom}+\norm{\Grad\grad u}_{\Ltttom}\big),
\end{align*}
showing the desired estimate.
\end{proof}

\begin{theorem}
\label{thm:lHc2} 
Let $\om\subset\R^3$ be a bounded strong Lipschitz domain. 
Then the second \pz complex 
$(\rdevGrad,\rsymCT,\rdDS)$ is a maximal compact Hilbert complex.
\end{theorem}

\begin{proof}
The assertions can be shown by using the `FA-ToolBox' from \cite{P2019a,P2019b,P2020a,PZ2020a,PZ2020b,PZ2020c}.
In particular, the crucial compact embeddings can be shown 
by the same techniques used in the proof of \cite[Theorem 4.7]{PZ2020c}.
\end{proof}

We observe and define
\begin{align} 
\label{bihdim12}
\begin{aligned}
N_{0}^{\bihtwop}&=\ker A_{0}=\ker(\rdevGrad),\\
N_{2,*}^{\bihtwop}&=\ker A_{2}^{*} =\ker(\Gg),\\
K_{1}^{\bihtwop}&=\ker A_{1} \cap\ker A_{0}^{*}=\ker(\rsymCT)\cap\ker(\DT)=:\harmbihtwoDom,\\
K_{2}^{\bihtwop}&=\ker A_{2} \cap\ker A_{1}^{*}=\ker(\rdDS)\cap\ker(\CS)=:\harmbihtwoNom.
\end{aligned}
\end{align}

\begin{theorem}
\label{thm:DNTbih2} 
Let $\om\subset\R^3$ be open and bounded with continuous boundary. 
Moreover, suppose Assumption \ref{ass:curvesandsurfaces}. Then
$$\dim\harmbihtwoDom=4(m-1),\qquad
\dim\harmbihtwoNom=4p.$$
\end{theorem} 

\begin{proof}
We postpone the proof to Sections \ref{app:sec:DirTFbih2} and \ref{app:sec:NeuTFbih2}.
\end{proof}

\begin{theorem}
\label{thm:idxPZC2} 
Let $\om\subset\R^3$ be a bounded strong Lipschitz domain. 
Then $\calD^{\bihtwop}$ is a Fredholm operator with index
$$\ind\calD^{\bihtwop}
=\dim N_{0}^{\bihtwop}-\dim K_{1}^{\bihtwop}+\dim K_{2}^{\bihtwop}-\dim N_{2,*}^{\bihtwop}.$$
If additionally Assumption \ref{ass:curvesandsurfaces} holds, then
$$\ind\calD^{\bihtwop}
=4(p-m-n+1).$$
\end{theorem}

\begin{proof}
Using Theorem \ref{thm:lHc2} apply Theorem \ref{thm:index} together with \eqref{bihdim12}, the observations
\begin{align}
\label{kernelsbih2}
N_{0}^{\bihtwop}=\ker(\rdevGrad)=\{0\},\qquad
N_{2,*}^{\bihtwop}=\ker(\Gg)=\Pol^{1}_{\pw}
\end{align}
by using \cite[Lemma 3.2 (i)]{PZ2020a}, and Theorem \ref{thm:DNTbih2}.
\end{proof}

\begin{remark}
\label{thm:idxPZC2:rem}
By Theorem \ref{thm:index} the adjoint $(\calD^{\bihtwop})^{*}$
is Fredholm as well with index simply given by $\ind(\calD^{\bihtwop})^{*}=-\ind\calD^{\bihtwop}$.
Similar to Remark \ref{picardextmaxwell} and Remark \ref{thm:idxPZC:rem} 
we define the extended second \pz operator
$$\calM^{\bihtwop}
:=\begin{pmatrix}0&\calD^{\bihtwop}\\-(\calD^{\bihtwop})^{*}&0\end{pmatrix}
=\begin{pmatrix}0&0&\rdDS&0\\0&0&\CS&\rdevGrad\\
-\Gg&-\rsymCT&0&0\\0&\DT&0&0\end{pmatrix}$$
with $(\calM^{\bihtwop})^{*}=-\calM^{\bihtwop}$ and $\ind\calM^{\bihtwop}=0$.
Moreover, $\dim\ker\calM^{\bihtwop}=4(n+m+p-1)$ as
$\ker\calM^{\bihtwop}
=\Pol_{\pw}^{1}\times\harmbihtwoDom\times\harmbihtwoNom\times\{0\}$.
\end{remark}

%\subsection*{Variable Coefficients and Poincar\'e--Friedrichs Type Inequalities}
\noindent
\textbf{Variable Coefficients and Poincar\'e--Friedrichs Type Inequalities.}
\label{sec:mmorerebih2}
Inhomogeneous and anisotropic media may also be considered 
for the second \pz complex, cf. Remark \ref{rem:epsmuderham} and Remark \ref{rem:epsmubih1}.

\begin{remark}
\label{rem:epsmubih2}
Recall the notations from Remark \ref{rem:epsmubih1} and set
$\lambda_{0}:=\id$, $\lambda_{3}:=\id$,
$\lambda_{1}:=\eps$, $\lambda_{2}:=\mu$, and
$\widetilde{H}_{1}:=\Lttt_{\Tb,\eps}(\om)$,
$\widetilde{H}_{2}:=\Lttt_{\Sb,\mu}(\om)$,
$\widetilde{H}_{0}=H_{0}=\Lttom$, $\widetilde{H}_{3}=H_{3}=\Ltom$.
We look at
\begin{align*}
\widetilde{A}_{0}&:=\rdevGrad,
&
\widetilde{A}_{1}&:=\mu^{-1}\rsymCT,
&
\widetilde{A}_{2}&:=\rdDS\mu,\\
\widetilde{A}_{0}^{*}&\phantom{:}=-\DT\eps,
&
\widetilde{A}_{1}^{*}&\phantom{:}=\eps^{-1}\CS,
&
\widetilde{A}_{2}^{*}&\phantom{:}=\Gg,
\end{align*}
\begin{align*}
\widetilde{\calD}^{\bihtwop}
&:=\begin{pmatrix}\widetilde{A}_{2}&0\\\widetilde{A}_{1}^{*}&\widetilde{A}_{0}\end{pmatrix}
=\begin{pmatrix}\rdDS\mu&0\\\eps^{-1}\CS&\rdevGrad\end{pmatrix},\\
(\widetilde{\calD}^{\bihtwop})^{*}
&\phantom{:}=\begin{pmatrix}\widetilde{A}_{2}^{*}& \widetilde{A}_{1}\\0&\widetilde{A}_{0}^{*}\end{pmatrix}
=\begin{pmatrix}\Gg&\mu^{-1}\rsymCT\\0&-\DT\eps\end{pmatrix},
\end{align*}
i.e., the second \pz complex, cf. \eqref{bih2compl1},
\begin{align}
\footnotesize
\begin{aligned}
\label{bih2compl2}
\{0\}\xrightarrow{\iota_{\{0\}}}
\Lttom\xrightarrow{\rdevGrad}
\Lttt_{\Tb,\eps}(\om)&\xrightarrow{\mu^{-1}\rsymCT}
\Lttt_{\Sb,\mu}(\om)\xrightarrow{\rdDS\mu}
\Ltom\xrightarrow{\pi_{\Pol^{1}_{\pw}}}\Pol^{1}_{\pw},\\
\{0\}\xleftarrow{\pi_{\{0\}}}
\Lttom\xleftarrow{-\DT\eps}
\Lttt_{\Tb,\eps}(\om)&\xleftarrow{\eps^{-1}\CS}
\Lttt_{\Sb,\mu}(\om)\xleftarrow{\Gg}
\Ltom\xleftarrow{\iota_{\Pol^{1}_{\pw}}}\Pol^{1}_{\pw}.
\end{aligned}
\end{align}
Lemma \ref{lambdaindeplem}, Lemma \ref{lem:properties4Atilde}, and Theorem \ref{thm:indexAtilde}
show that the compactness properties, the dimensions of the kernels and cohomology groups,
the maximal compactness, and the Fredholm indices
of the second \pz complex do not dependent of the material weights $\eps$ and $\mu$.
More precisely,
\begin{align*}
&\bullet&
\dim\big(\ker(\rsymCT)\cap\big(\eps^{-1}\ker(\DT)\big)\big)
&=\dim\big(\ker(\rsymCT)\cap\ker(\DT)\big)\\
&&&=\dim\harmbihtwoDom
=4(m-1),\\
&\bullet&
\dim\big(\big(\mu^{-1}\ker(\rdDS)\big)\cap\ker(\CS)\big)
&=\dim\big(\ker(\rdDS)\cap\ker(\CS)\big)\\
&&&=\dim\harmbihtwoNom
=4p,\\
&\bullet&
\dom(\rsymCT)\cap\big(\eps^{-1}\dom(\DT)\big)
&\hookrightarrow\Lttt_{\Tb,\eps}(\om)\text{ compactly}\\
&&\Leftrightarrow\quad
\dom(\rsymCT)\cap\dom(\DT)
&\hookrightarrow\LtttTom\text{ compactly},\\
&\bullet&
\big(\mu^{-1}\dom(\rdDS)\big)\cap\dom(\CS)
&\hookrightarrow\Lttt_{\Sb,\mu}(\om)\text{ compactly}\\
&&\Leftrightarrow\quad
\dom(\rdDS)\cap\dom(\CS)
&\hookrightarrow\LtttSom\text{ compactly},\\
&\bullet&
(\rdevGrad,\mu^{-1}\rsymCT,\rdDS\mu)&\text{ maximal compact}\\
&&\Leftrightarrow\quad(\rdevGrad,\rsymCT,\rdDS)&\text{ maximal compact},\\
&\bullet&
-\ind(\widetilde{\calD}^{\bihtwo})^{*}=\ind\widetilde{\calD}^{\bihtwo}
=\ind\calD^{\bihtwo}
&=4(p-m-n+1).
\end{align*}
\end{remark}

Note that the kernels and ranges are given by
\begin{align*}
\ker\calD^{\bihtwop}
=K_{2}^{\bihtwop}\times N_{0}^{\bihtwop}
&=\harmbihtwoNom\times\{0\},\\
\ker(\calD^{\bihtwop})^{*}
=N_{2,*}^{\bihtwop}\times K_{1}^{\bihtwop}
&=\Pol^{1}_{\pw}\times\harmbihtwoDom,\\
\ran\calD^{\bihtwop}
=(\ker(\calD^{\bihtwop})^{*})^{\bot_{\Ltom\times\LtttTom}}
&=(\Pol^{1}_{\pw})^{\bot_{\Ltom}}\times\harmbihtwoDom^{\bot_{\LtttTom}},\\
\ran(\calD^{\bihtwop})^{*}
=(\ker\calD^{\bihtwop})^{\bot_{\LtttSom\times\Lttom}}
&=\harmbihtwoNom^{\bot_{\LtttSom}}\times\Lttom,
\end{align*}
see Lemma \ref{lem:properties3}, Corollary \ref{lem:properties3a}, and \eqref{kernelsbih2}.
Corollary \ref{cor:FP} shows additional results for the corresponding reduced operators 
\begin{align*}
\calD^{\bihtwop}_{\red}
=\calD^{\bihtwop}|_{(\ker\calD^{\bihtwop})^{\bot_{H_{2}\times H_{0}}}}
&=\begin{pmatrix}\rdDS&0\\\CS&\rdevGrad\end{pmatrix}\Big|_{\harmbihtwoNom^{\bot_{\LtttSom}}\times\Lttom},\\
(\calD^{\bihtwop}_{\red})^{*}
=(\calD^{\bihtwop})^{*}|_{(\ker(\calD^{\bihtwop})^{*})^{\bot_{H_{3}\times H_{1}}}}
&=\begin{pmatrix}\Gg&\rsymCT\\0&-\DT\end{pmatrix}\Big|_{(\Pol^{1}_{\pw})^{\bot_{\Ltom}}\times\harmbihtwoDom^{\bot_{\LtttTom}}}.
\end{align*}

\begin{corollary}
\label{cor:FP-tb-bih2}
Let $\om\subset\R^3$ be a bounded strong Lipschitz domain. Then
\begin{align*}
(\calD^{\bihtwop}_{\red})^{-1}:\ran\calD^{\bihtwop}&\to\ran(\calD^{\bihtwop})^{*},\\
((\calD^{\bihtwop}_{\red})^{*})^{-1}:\ran(\calD^{\bihtwop})^{*}&\to\ran\calD^{\bihtwop}
\intertext{are compact. Furthermore,}
(\calD^{\bihtwop}_{\red})^{-1}:\ran\calD^{\bihtwop}&\to\dom\calD^{\bihtwop}_{\red},\\
((\calD^{\bihtwop}_{\red})^{*})^{-1}:\ran(\calD^{\bihtwop})^{*}&\to\dom(\calD^{\bihtwop}_{\red})^{*}
\end{align*}
are continuous and, equivalently, the Friedrichs-Poincar\'e type estimates 
\begin{align*}
\bnorm{(S,v)}_{\LtttSom\times\Lttom}
&\leq c_{\calD^{\bihtwop}}\big(\norm{\devGrad v}_{\LtttTom}^{2}\\
&\qquad\qquad\qquad+\norm{\dD S}_{\Ltom}^{2}+\norm{\Curl S}_{\LtttTom}^{2}\big)^{1/2},\\
\bnorm{(u,T)}_{\Ltom\times\LtttTom}
&\leq c_{\calD^{\bihtwop}}\big(\norm{\Gg u}_{\LtttSom}^{2}\\
&\qquad\qquad\qquad+\norm{\Dive T}_{\Lttom}^{2}+\norm{\symCurl T}_{\LtttSom}^{2}\big)^{1/2}
\end{align*}
hold for all $(S,v)$ in
\begin{align*}
\dom\calD^{\bihtwop}_{\red}
&=\big(\dom(\rdDS)\cap\dom(\CS)\cap\harmbihtwoNom^{\bot_{\LtttSom}}\big)
\times H^{1,3}_{0}(\om)
\intertext{for all $(u,T)$ in}
\dom(\calD^{\bihtwop}_{\red})^{*}
&=\big(H^{2}(\om)\cap(\Pol^{1}_{\pw})^{\bot_{\Ltom}}\big)\\
&\qquad\qquad\times\big(\dom(\rsymCT)\cap\dom(\DT)\cap\harmbihtwoDom^{\bot_{\LtttTom}}\big)
\end{align*}
with some optimal constant $c_{\calD^{\bihtwop}}>0$.
\end{corollary}

\section{The Elasticity Complex and Its Indices}
\label{sec:ela}

This section is devoted to adapt our main results Theorem \ref{thm:main}, 
Theorem \ref{thm:idxPZC}, and Theorem \ref{thm:idxPZC2}, to the elasticity complex, see \cite{PZ2020b,PZ2020c} for details.
Its elasticity differential operator is of mixed order as well,
this time in the center of the complex.
As before for the \pz operators, the leading order term is \emph{not} 
dominating the lower order differential operators.

\begin{definition} 
Let $\om\subset\R^3$ be an open set. We put
\begin{align*}
\symGrad_{c}:\Cictom\subset\Lttom&\to\LtttSom,
&
\phi&\mapsto\sym\Grad\phi,\\
\CCt_{c}:\CicttSom\subset\LtttSom&\to\LtttSom, 
&
\Phi&\mapsto\CCt\Phi:=\Curl(\Curl\Phi)^{\top},\\
\Dive_{c}:\CicttSom\subset\LtttSom&\to\Lttom,
& 
\Phi&\mapsto\Dive\Phi,
\end{align*}
and further define the densely defined and closed linear operators
\begin{align*}
\DS&:=-\symGrad_{c}^{*},
&
\rsymGrad&:=-\DS^{*}=\overline{\symGrad_{c}},\\
\CCtS&:=(\CCt_{c})^{*},
&
\rCCtS&:=(\CCtS)^{*}=\overline{\CCt_{c}},\\
\symGrad&:=-\Dive_{c}^{*},
&
\rDS&:=-\symGrad^{*}=\overline{\Dive_{c}}.
\end{align*}
\end{definition}

We want to apply the index theorem in the following situation of the elasticity complex:
\begin{align*}
A_{0}&:=\rsymGrad,
&
A_{1}&:=\rCCtS,
&
A_{2}&:=\rDS,\\
A_{0}^{*}&\phantom{:}=-\DS,
&
A_{1}^{*}&\phantom{:}=\CCtS,
&
A_{2}^{*}&\phantom{:}=-\symGrad,
\end{align*}
\begin{align*}
\calD^{\elap}
&:=\begin{pmatrix}A_{2}&0\\A_{1}^{*}&A_{0}\end{pmatrix}
=\begin{pmatrix}\rDS& 0\\\CCtS&\rsymGrad\end{pmatrix},\\
(\calD^{\elap})^{*}
&\phantom{:}=\begin{pmatrix}A_{2}^{*}& A_{1}\\0&A_{0}^{*}\end{pmatrix}
=\begin{pmatrix}-\symGrad&\rCCtS\\0&-\DS\end{pmatrix},
\end{align*}
\begin{align}
\footnotesize
\begin{aligned}
\label{elacompl1}
\{0\}\xrightarrow{\iota_{\{0\}}}
\Lttom\xrightarrow{\rsymGrad}
\LtttSom&\xrightarrow{\rCCtS}
\LtttSom\xrightarrow{\rDS}
\Lttom\xrightarrow{\pi_{\RM_{\pw}}}\RM_{\pw},\\
\{0\}\xleftarrow{\pi_{\{0\}}}
\Lttom\xleftarrow{-\DS}
\LtttSom&\xleftarrow{\CCtS}
\LtttSom\xleftarrow{-\symGrad}
\Lttom\xleftarrow{\iota_{\RM_{\pw}}}\RM_{\pw}.
\end{aligned}
\end{align}

The foundation of the index theorem to follow is the following compactness result established in \cite{PZ2020b,PZ2020c}. 
Note that we have by Korn's inequalities
$\dom(\rsymGrad)=H^{1,3}_{0}(\om)$ and $\dom(\symGrad)=H^{1,3}(\om)$.

\begin{theorem}[{\cite[Theorem 4.7]{PZ2020c}}]
\label{thm:lHc-ela} 
Let $\om\subset\R^3$ be a bounded strong Lipschitz domain. 
Then $(\rsymGrad,\rCCtS,\rDS)$ is a maximal compact Hilbert complex.
\end{theorem}

We observe and define
\begin{align} 
\label{eladim1}
\begin{aligned}
N_{0}^{\elap}&=\ker A_{0}=\ker(\rsymGrad),\\
N_{2,*}^{\elap}&=\ker A_{2}^{*} =\ker(\symGrad),\\
K_{1}^{\elap}&=\ker A_{1} \cap\ker A_{0}^{*}=\ker(\rCCtS)\cap\ker(\DS)=:\harmelaDom,\\
K_{2}^{\elap}&=\ker A_{2} \cap\ker A_{1}^{*}=\ker(\rDS)\cap\ker(\CCtS)=:\harmelaNom.
\end{aligned}
\end{align}

The dimensions of the cohomology groups are given as follows.
 
\begin{theorem}
\label{thm:DNTela} 
Let $\om\subset\R^3$ be open and bounded with continuous boundary. 
Moreover, suppose Assumption \ref{ass:curvesandsurfaces}. Then
$$\dim\harmelaDom=6(m-1),\qquad
\dim\harmelaNom=6p.$$
\end{theorem} 

\begin{proof}
We postpone the proof to the Sections \ref{app:sec:DirTFela} and \ref{app:sec:NeuTFela}.
\end{proof}

Let us introduce the space of piecewise rigid motions by (for $\cc(\om)$ see \eqref{eq:ccO})
\begin{equation}
\label{eq:RM_pw}
\RM_{\pw}:=\big\{v\in\Lttom:\forall\,C\in\cc(\om)\quad\exists\,\alpha_{C},\beta_{C}\in\R^{3}:
u|_{C}(x)=\alpha_{C}\times x+\beta_{C}\big\}.
\end{equation}

\begin{theorem}
\label{thm:idxelaC}
Let $\om\subset\R^3$ be a bounded strong Lipschitz domain. 
Then $\calD^{\elap}$ is a Fredholm operator with index
$$\ind\calD^{\elap}
=\dim N_{0}^{\elap}-\dim K_{1}^{\elap}+\dim K_{2}^{\elap}-\dim N_{2,*}^{\elap}.$$
If additionally Assumption \ref{ass:curvesandsurfaces} holds, then
$$\ind\calD^{\elap}
=6(p-m-n+1).$$
\end{theorem}

\begin{proof}
Using Theorem \ref{thm:lHc-ela} apply Theorem \ref{thm:index} together with \eqref{eladim1}, the observations
\begin{align}
\label{kernelsela}
N_{0}^{\elap}=\ker(\rsymGrad)=\{0\},\qquad
N_{2,*}^{\elap}=\ker(\symGrad)=\RM_{\pw},
\end{align}
see \cite[Lemma 3.2]{PZ2020b}, and Theorem \ref{thm:DNTela}.
\end{proof}

\begin{remark}
\label{thm:idxelaC:rem}
By Theorem \ref{thm:index} the adjoint $(\calD^{\elap})^{*}$
is Fredholm as well with index simply given by $\ind(\calD^{\elap})^{*}=-\ind\calD^{\elap}$.
Similar to Remark \ref{picardextmaxwell}, Remark \ref{thm:idxPZC:rem}, and Remark \ref{thm:idxPZC2:rem} 
we define the extended elasticity operator
$$\calM^{\ela}
:=\begin{pmatrix}0&\calD^{\ela}\\-(\calD^{\ela})^{*}&0\end{pmatrix}
=\begin{pmatrix}0&0&\rDS&0\\0&0&\CCtS&\rsymGrad\\
\symGrad&-\rCCtS&0&0\\0&\DS&0&0\end{pmatrix}$$
with $(\calM^{\ela})^{*}=-\calM^{\ela}$ and $\ind\calM^{\ela}=0$.
Moreover, $\dim\ker\calM^{\ela}=6(n+m+p-1)$ as
$\ker\calM^{\ela}
=\RM_{\pw}\times\harmelaDom\times\harmelaNom\times\{0\}$.
\end{remark}

%\subsection*{Variable Coefficients and Poincar\'e--Friedrichs Type Inequalities}
\noindent
\textbf{Variable Coefficients and Poincar\'e--Friedrichs Type Inequalities.}
\label{sec:mmorereela} 
Inhomogeneous and anisotropic media may also be considered 
for the elasticity complex, cf. Remark \ref{rem:epsmuderham}, Remark \ref{rem:epsmubih1}, 
and Remark \ref{rem:epsmubih2}.

\begin{remark}
\label{rem:epsmuela}
Recall the notations from Remark \ref{rem:epsmubih1} and Remark \ref{rem:epsmubih2} and set
$\lambda_{0}:=\id$, $\lambda_{3}:=\id$,
$\lambda_{1}:=\eps$, $\lambda_{2}:=\mu$, and
$\widetilde{H}_{3}=\widetilde{H}_{0}=H_{3}=H_{0}=\Lttom$,
$\widetilde{H}_{1}:=\Lttt_{\Sb,\eps}(\om)$,
$\widetilde{H}_{2}:=\Lttt_{\Sb,\mu}(\om)$.
We look at
\begin{align*}
\widetilde{A}_{0}&:=\rsymGrad,
&
\widetilde{A}_{1}&:=\mu^{-1}\rCCtS,
&
\widetilde{A}_{2}&:=\rDS\mu,\\
\widetilde{A}_{0}^{*}&\phantom{:}=-\DS\eps,
&
\widetilde{A}_{1}^{*}&\phantom{:}=\eps^{-1}\CCtS,
&
\widetilde{A}_{2}^{*}&\phantom{:}=-\symGrad,
\end{align*}
\begin{align*}
\widetilde{\calD}^{\ela}
&:=\begin{pmatrix}\widetilde{A}_{2}&0\\\widetilde{A}_{1}^{*}&\widetilde{A}_{0}\end{pmatrix}
=\begin{pmatrix}\rDS\mu&0\\\eps^{-1}\CCtS&\rsymGrad\end{pmatrix},\\
(\widetilde{\calD}^{\ela})^{*}
&\phantom{:}=\begin{pmatrix}\widetilde{A}_{2}^{*}& \widetilde{A}_{1}\\0&\widetilde{A}_{0}^{*}\end{pmatrix}
=\begin{pmatrix}-\symGrad&\mu^{-1}\rCCtS\\0&-\DS\eps\end{pmatrix},
\end{align*}
i.e., the elasticity complex, cf. \eqref{elacompl1},
\begin{align}
\footnotesize
\begin{aligned}
\label{elacompl2}
\{0\}\xrightarrow{\iota_{\{0\}}}
\Lttom\xrightarrow{\rsymGrad}
\Lttt_{\Sb,\eps}(\om)&\xrightarrow{\mu^{-1}\rCCtS}
\Lttt_{\Sb,\mu}(\om)\xrightarrow{\rDS\mu}
\Lttom\xrightarrow{\pi_{\RM_{\pw}}}\RM_{\pw},\\
\{0\}\xleftarrow{\pi_{\{0\}}}
\Lttom\xleftarrow{-\DS\eps}
\Lttt_{\Sb,\eps}(\om)&\xleftarrow{\eps^{-1}\CCtS}
\Lttt_{\Sb,\mu}(\om)\xleftarrow{-\symGrad}
\Lttom\xleftarrow{\iota_{\RM_{\pw}}}\RM_{\pw}.
\end{aligned}
\end{align}
Lemma \ref{lambdaindeplem}, Lemma \ref{lem:properties4Atilde}, and Theorem \ref{thm:indexAtilde}
show that the compactness properties, the dimensions of the kernels and cohomology groups,
the maximal compactness, and the Fredholm indices
of the elasticity complex do not dependent of the material weights $\eps$ and $\mu$.
More precisely,
\begin{align*}
&\bullet&
\dim\big(\ker(\rCCtS)\cap\big(\eps^{-1}\ker(\DS)\big)\big)
&=\dim\big(\ker(\rCCtS)\cap\ker(\DS)\big)\\
&&&=\dim\harmelaDom
=6(m-1),\\
&\bullet&
\dim\big(\big(\mu^{-1}\ker(\rDS)\big)\cap\ker(\CCtS)\big)
&=\dim\big(\ker(\rDS)\cap\ker(\CCtS)\big)\\
&&&=\dim\harmelaNom
=6p,\\
&\bullet&
\dom(\rCCtS)\cap\big(\eps^{-1}\dom(\DS)\big)
&\hookrightarrow\Lttt_{\Sb,\eps}(\om)\text{ compactly}\\
&&\Leftrightarrow\quad
\dom(\rCCtS)\cap\dom(\DS)
&\hookrightarrow\LtttSom\text{ compactly},\\
&\bullet&
\big(\mu^{-1}\dom(\rDS)\big)\cap\dom(\CCtS)
&\hookrightarrow\Lttt_{\Sb,\mu}(\om)\text{ compactly}\\
&&\Leftrightarrow\quad
\dom(\rDS)\cap\dom(\CCtS)
&\hookrightarrow\LtttSom\text{ compactly},\\
&\bullet&
(\rsymGrad,\mu^{-1}\rCCtS,\rDS\mu)&\text{ maximal compact}\\
&&\Leftrightarrow\quad(\rsymGrad,\rCCtS,\rDS)&\text{ maximal compact},\\
&\bullet&
-\ind(\widetilde{\calD}^{\ela})^{*}=\ind\widetilde{\calD}^{\ela}
=\ind\calD^{\ela}
&=6(p-m-n+1).
\end{align*}
\end{remark}

Note that the kernels and ranges are given by
\begin{align*}
\ker\calD^{\elap}
=K_{2}^{\elap}\times N_{0}^{\elap}
&=\harmelaNom\times\{0\},\\
\ker(\calD^{\elap})^{*}
=N_{2,*}^{\elap}\times K_{1}^{\elap}
&=\RM_{\pw}\times\harmelaDom,\\
\ran\calD^{\elap}
=(\ker(\calD^{\elap})^{*})^{\bot_{\Lttom\times\LtttSom}}
&=\RM_{\pw}^{\bot_{\Lttom}}\times\harmelaDom^{\bot_{\LtttSom}},\\
\ran(\calD^{\elap})^{*}
=(\ker\calD^{\elap})^{\bot_{\LtttSom\times\Lttom}}
&=\harmelaNom^{\bot_{\LtttSom}}\times\Lttom,
\end{align*}
see Lemma \ref{lem:properties3}, Corollary \ref{lem:properties3a}, and \eqref{kernelsela}.
Corollary \ref{cor:FP} shows additional results for the corresponding reduced operators 
\begin{align*}
\calD^{\elap}_{\red}
=\calD^{\elap}|_{(\ker\calD^{\elap})^{\bot_{H_{2}\times H_{0}}}}
&=\begin{pmatrix}\rDS& 0\\\CCtS&\rsymGrad\end{pmatrix}\Big|_{\harmelaNom^{\bot_{\LtttSom}}\times\Lttom},\\
(\calD^{\elap}_{\red})^{*}
=(\calD^{\elap})^{*}|_{(\ker(\calD^{\elap})^{*})^{\bot_{H_{3}\times H_{1}}}}
&=\begin{pmatrix}-\symGrad&\rCCtS\\0&-\DS\end{pmatrix}\Big|_{\RM_{\pw}^{\bot_{\Lttom}}\times\harmelaDom^{\bot_{\LtttSom}}}.
\end{align*}

\begin{corollary}
\label{cor:FP-tb-ela}
Let $\om\subset\R^3$ be a bounded strong Lipschitz domain. Then
\begin{align*}
(\calD^{\elap}_{\red})^{-1}:\ran\calD^{\elap}&\to\ran(\calD^{\elap})^{*},\\
((\calD^{\elap}_{\red})^{*})^{-1}:\ran(\calD^{\elap})^{*}&\to\ran\calD^{\elap}
\intertext{are compact. Furthermore,}
(\calD^{\elap}_{\red})^{-1}:\ran\calD^{\elap}&\to\dom\calD^{\elap}_{\red},\\
((\calD^{\elap}_{\red})^{*})^{-1}:\ran(\calD^{\elap})^{*}&\to\dom(\calD^{\elap}_{\red})^{*}
\end{align*}
are continuous and, equivalently, the Friedrichs--Poincar\'e type estimate
\begin{align*}
\bnorm{(S,v)}_{\LtttSom\times\Lttom}
&\leq c_{\calD^{\elap}}\big(\norm{\symGrad v}_{\LtttSom}^{2}\\
&\qquad\qquad\qquad+\norm{\Dive S}_{\Lttom}^{2}+\norm{\CCt S}_{\LtttSom}^{2}\big)^{1/2}
\end{align*}
holds for all $(S,v)$ in
\begin{align*}
\dom\calD^{\elap}_{\red}
&=\big(\dom(\rDS)\cap\dom(\CCtS)\cap\harmelaNom^{\bot_{\LtttSom}}\big)
\times H^{1,3}_{0}(\om)
\intertext{or $(v,S)$ in}
\dom(\calD^{\elap}_{\red})^{*}
&=\big(H^{1,3}(\om)\cap\RM_{\pw}^{\bot_{\Lttom}}\big)\\
&\qquad\qquad\times\big(\dom(\rCCtS)\cap\dom(\DS)\cap\harmelaDom^{\bot_{\LtttSom}}\big)
\end{align*}
with some optimal constant $c_{\calD^{\elap}}>0$.
\end{corollary}

\section{The Main Topological Assumptions}
\label{app:sec:DirNeuF}

In Theorem \ref{thm:DNF}, Theorem \ref{thm:DNTbih}, Theorem \ref{thm:DNTbih2}, and Theorem \ref{thm:DNTela}
we have seen that the dimensions of the harmonic Dirichlet and Neumann fields 
are given by the topological invariants of the open and bounded set $\om$
and its complement 
$$\Xi:=\R^3\setminus\overline{\om},$$ 
i.e., by
\begin{itemize}
\item
$n$, the number of connected components $\om_{k}$ of $\om$, i.e.,
$\om=\dot\bigcup_{k=1}^{n}\,\om_{k}$,
\item
$m$, the number of connected components $\Xi_{\ell}$ of $\Xi$, i.e.,
$\Xi=\dot\bigcup_{\ell=0}^{m-1}\,\Xi_{\ell}$,
\item
$p$, the number of handles of $\om$, see Assumption \ref{ass:curvesandsurfaces}.
\end{itemize}
Note that $\cc(\om)=\{\om_{1},\dots,\om_{n}\}$ and $\cc(\Xi)=\{\Xi_{0},\dots,\Xi_{m-1}\}$.
We have claimed
\begin{align*}
\dim\harmDom&=m-1,
&
\dim\harmNom&=p,\\
\dim\harmbihoneDom&=4(m-1),
&
\dim\harmbihoneNom&=4p,\\
\dim\harmbihtwoDom&=4(m-1),
&
\dim\harmbihtwoNom&=4p,\\
\dim\harmelaDom&=6(m-1),
&
\dim\harmelaNom&=6p.
\end{align*}

The concluding sections of this manuscript are devoted to provide the corresponding proofs in detail.
For the de Rham complex we follow in close lines the arguments of Picard in \cite{P1982}
introducing some simplifications for bounded domains and trivial material tensors $\eps$ and $\mu$.
These ideas will be adapted and modified for the proofs 
of the corresponding results of the other Hilbert complexes.

\begin{assumption}
\label{ass:segmentprop}
$\om\subset\R^{3}$ is open and bounded with segment property, i.e.,
$\om$ has a continuous boundary $\ga:=\p\om$, see Remark \ref{segmentcont}.
\end{assumption}

\begin{assumption}
\label{ass:stronglip}
$\om\subset\R^{3}$ is open, bounded, and $\ga$ is strong Lipschitz.
\end{assumption}

In view of Assumption \ref{ass:segmentprop} and Assumption \ref{ass:stronglip} we note:
\begin{itemize}
\item
Assumption \ref{ass:segmentprop} guarantees that $m,n\in\N$ are well-defined.
In particular, we have $\mathrm{int}\,\Xi_{\ell}\neq\emptyset$ for all $\ell\in\{0,\dots,m-1\}$.
\item
Assumption \ref{ass:stronglip} implies Assumption \ref{ass:segmentprop}.
\item
Assumption \ref{ass:stronglip} simplifies some arguments, in particular,
all ranges in the crucial Helmholtz type decompositions used in our proofs
are closed, cf. Remark \ref{rem:Helmholtz-derhamN}, Remark \ref{rem:Helmholtz-bih1N}, 
Remark \ref{rem:Helmholtz-bih2N}, and Remark \ref{rem:Helmholtz-elaN}.
We emphasise that all our results presented in the following still hold with 
Assumption \ref{ass:stronglip} replaced by the weaker Assumption \ref{ass:segmentprop}.
In this case, however, the computation (and verification of the existence of) 
the Fredholm index in the sections above is more involved. 
In fact, it is not clear if the mentioned ranges are closed and 
in some of our arguments we need to use some additional density
and approximation arguments. 
\item
Our results concerning the bases and dimensions
of the generalised Dirichlet and Neumann fields extend 
naturally to exterior domains, i.e., domains with bounded complement $\Xi$.
For simplicity and to avoid even longer and more technical proofs
we restrict ourselves to the case of bounded domains $\om$ here.
\end{itemize}

The key topological assumptions to be satisfied by $\om$
to compute a basis for the Neumann fields and for $p$ to be well-defined, is described in detail next. 
For this, we recall the construction in \cite{P1982}. 

\begin{assumption}[{\cite[Section 1]{P1982}}]
\label{ass:curvesandsurfaces}
Let $\om\subset\R^3$ be open and bounded. 
There are $p\in\N_{0}$ piecewise $C^1$-curves $\zeta_{j}$
and $p$ $C^{2}$-surfaces $F_{j}$, $j\in\{1,\ldots, p\}$, with the following properties:
\begin{enumerate}
\item[\bf(A1)]
The curves $\zeta_{j}$, $j\in\{1,\ldots, p\}$, are pairwise disjoint 
and given any closed piecewise $C^1$-curve $\zeta$ in $\om$ there exists uniquely determined 
$\alpha_{j}\in\Z$, $j\in\{1,\ldots, p\}$, 
such that for all $\Phi\in\ker(\curl)$ being continuously differentiable we have
$$\int_{\zeta}\scp{\Phi}{\intd\lambda} = \sum_{j=1}^{p}\alpha_{j}\int_{\zeta_{j}}\scp{\Phi}{\intd\lambda}.$$
\item[\bf(A2)]
$ F_{j}$, $j\in\{1,\ldots, p\}$, are pairwise disjoint 
and $ F_{j}\cap \zeta_{k}$ is a singleton, if $j=k$, and empty, if $j\neq k$.
\item[\bf(A3)] 
If $\om_{c}\in\cc(\om)$, i.e., $\om_{c}$ is a connected component of $\om$, 
then $\om_{c}\setminus\bigcup_{j=1}^{p} F_{j}$ is simply connected.
\end{enumerate}
\end{assumption}

$p$ is called the topological genus of $\om$
and the curves $\zeta_{j}$ are said to represent a basis 
of the respective homology group of $\om$.

It is worth mentioning the following local regularity results 
for the Dirichlet and Neumann fields (see Lemma \ref{lem:harmreg} below), 
which are crucial for the construction of the Neumann fields,
\begin{align}
\label{app:harmreg}
\begin{aligned}
\harmDom,\harmNom
&\subset\Citom\cap\Lttom,\\
\harmbihoneDom,\harmelaDom,\harmbihtwoNom,\harmelaNom
&\subset\Cittom\cap\LtttSom,\\
\harmbihtwoDom,\harmbihoneNom
&\subset\Cittom\cap\LtttTom.
\end{aligned}
\end{align}
In particular, all Dirichlet and Neumann fields of the respective cohomology groups
are continuous and square integrable.

\section{The Construction of the Dirichlet Fields}
\label{app:sec:DirF}

Let us denote the unbounded connected component of $\Xi$ by $\Xi_{0}$
and its boundary by $\ga_{0}:=\p\Xi_{0}$. The remaining connected components of $\Xi$ 
are $\Xi_{1},\dots,\Xi_{m-1}$ with boundaries $\ga_{\ell}:=\p\Xi_{\ell}$.
Note that none of $\ga_{0},\dots,\ga_{m-1}$ need to be connected.
Furthermore, let us introduce an open (and bounded) ball 
$B\supset\overline{\om}$ and set $\widetilde{\Xi}_{0}:=B\cap\Xi_{0}$.
Then the connected components of $B\setminus\overline{\om}$ are
$\widetilde{\Xi}_{0}$ and $\Xi_{1},\dots,\Xi_{m-1}$. 
Moreover, let 
\begin{align}
\label{def:xi-ell}
\xi_{\ell}\in\Cic(\R^{3}),\qquad
\ell\in\{1,\dots,m-1\}, 
\end{align}
with disjoint supports
such that $\xi_{\ell}=0$ in a neighbourhood of $\Xi_{0}$ 
and in a neighbourhood of $\Xi_{k}$ for all $k\in\{1,\dots,m-1\}$, $k\neq\ell$,
as well as $\xi_{\ell}=1$ in a neighbourhood of $\Xi_{\ell}$.
In particular, $\xi_{\ell}=0$ in a neighbourhood of $\ga_{0}$ 
and in a neighbourhood of $\ga_{k}$ for all $k\in\{1,\dots,m-1\}$, $k\neq\ell$,
and $\xi_{\ell}=1$ in a neighbourhood of $\ga_{\ell}$.
Theses indicator type functions $\xi_{\ell}$ will be used to construct 
a basis for the respective Dirichlet fields.

\subsection{\except{toc}{Dirichlet Vector Fields of the Classical }De Rham Complex}
\label{app:sec:DirVFdeRham}

In this section, we rephrase the core arguments of \cite{P1982} in the simplified setting of bounded domains
and trivial materials $\eps$ and $\mu$. 
In order to highlight the apparent similarities and to motivate our rationale carried out 
for more involved situations later on, we shall present the construction for Dirichlet fields 
(and similarly for Neumann fields) in a seemingly great detail.

For the de Rham complex, see also \eqref{deco1} and \eqref{deco1.5}, 
we have the orthogonal decompositions 
\begin{align}
\label{app:deco1-derham}
\begin{aligned}
\Lttom=H_{1}
=\ran A_{0}\oplus_{H_{1}}\ker A_{0}^{*}
&=\ran(\rgrad,\om)\oplus_{\Lttom}\ker(\dive,\om),\\
\ker(\rcurl,\om)=\ker(A_{1})
=\ran A_{0}\oplus_{H_{1}}K_{1}
&=\ran(\rgrad,\om)\oplus_{\Lttom}\harmDom.
\end{aligned}
\end{align}

\begin{remark}
\label{rem:Helmholtz-derham}
We have $\dom(\rgrad,\om)=H^{1}_{0}(\om)$.
Moreover, the range in \eqref{app:deco1-derham} is closed due to the Friedrichs estimate
$$\exists\,c>0\quad\forall\,\phi\in H^{1}_{0}(\om)\qquad
\norm{\phi}_{\Ltom}\leq c\norm{\grad\phi}_{\Lttom},$$
which follows from Assumption \ref{ass:segmentprop}.
We recall that for the Friedrichs estimate to hold it suffices to assume that $\om$ is open and bounded only.
\end{remark}

Let us denote by $\pi:\Lttom\to\ker(\dive,\om)$ 
the orthogonal projector along $\ran(\rgrad,\om)$ onto $\ker(\dive,\om)$, 
which is well-defined according to \eqref{app:deco1-derham}. 
Moreover, we observe by \eqref{app:deco1-derham} that $\pi(\ker(\rcurl,\om))=\harmDom$.
Recall $\xi_{\ell}$ from \eqref{def:xi-ell}. 
Then for $\ell\in\{1,\dots,m-1\}$ 
$$\grad\xi_{\ell}\in\Cictom\cap\ker(\curl,\om)\subset\ker(\rcurl,\om).$$ 
Again relying on \eqref{app:deco1-derham} (and Remark \ref{rem:Helmholtz-derham}) 
for all $\ell\in\{1,\ldots,m-1\}$, we find uniquely determined $\psi_{\ell}\in H^{1}_{0}(\om)$ such that 
\begin{equation}
\label{eq:psiDVF}
\harmDom\ni\pi\grad\xi_{\ell}=\grad(\xi_{\ell}-\psi_{\ell})=\grad u_{\ell},\qquad
u_{\ell}:=\xi_{\ell}-\psi_{\ell}\in H^{1}(\om).
\end{equation}

We will show that 
\begin{align}
\label{basis:rhmD}
\calB^{\rhm}_{D}:=\{\grad u_{1},\dots,\grad u_{m-1}\}
\subset\harmDom
\end{align}
defines a basis of $\harmDom$. The first step for showing this statement is the next lemma.

\begin{lemma}
\label{dimDFrhm:lem1}
Let Assumption \ref{ass:segmentprop} be satisfied. Then
$\harmDom=\lin\calB^{\rhm}_{D}$.
\end{lemma}

\begin{proof}
Let $H\in\harmDom=\ker(\rcurl,\om)\cap\ker(\dive,\om)$.
In particular, by the homogeneous boundary condition its extension by zero, $\widetilde{H}$,
to $B$ belongs to $\ker(\rcurl,B)$.
As $B$ is topologically trivial (and smooth and bounded), there exists (a unique) 
$u\in H^{1}_{0}(B)$ such that $\grad u=\widetilde{H}$ in $B$,
see, e.g., \cite[Lemma 2.24]{PZ2020a}. As $\grad u=\widetilde{H}=0$ in $B\setminus\overline{\om}$,
$u$ must be constant in each connected component 
$\widetilde{\Xi}_{0},\Xi_{1},\dots,\Xi_{m-1}$
of $B\setminus\overline{\om}$.
Due to the homogenous boundary condition at $\p\!B$, $u$ vanishes in $\widetilde{\Xi}_{0}$.
Therefore, $H=\grad u$ in $\om$ and $u\in H^{1}_{0}(B)$ 
such that $u|_{\widetilde{\Xi}_{0}}=0$
and $u|_{\Xi_{\ell}}=:\alpha_{\ell}\in\R$ for all $\ell\in\{1,\dots,m-1\}$.
Let us consider 
$$\widehat{H}:=H-\sum_{\ell=1}^{m-1}\alpha_{\ell}\grad u_{\ell}
=\grad\widehat{u}\in\harmDom,\qquad
\widehat{u}:=u-\sum_{\ell=1}^{m-1}\alpha_{\ell}u_{\ell}\in H^{1}(\om)$$
with $u_{\ell}$ from \eqref{eq:psiDVF}.
The extension by zero of $\psi_{\ell}$, $\widetilde{\psi}_{\ell}$, to the whole of $B$ belongs to $H^{1}_{0}(B)$.
Hence as an element of $H^{1}(B)$ we see that 
$$\widehat{u}_{B}
:=u-\sum_{\ell=1}^{m-1}\alpha_{\ell}\xi_{\ell}
+\sum_{\ell=1}^{m-1}\alpha_{\ell}\widetilde{\psi}_{\ell}\in H^{1}_{0}(B)$$
vanishes in $\Xi_{\ell}$ for all $\ell\in\{0,\ldots,m-1\}$.
Thus $\widehat{u}=\widehat{u}_{B}|_{\om}\in H^{1}_{0}(\om)$ by Assumption \ref{ass:segmentprop}, and we compute
$$\norm{\widehat{H}}_{\Lttom}^{2}
=\scp{\grad\widehat{u}}{\widehat{H}}_{\Lttom}=0,$$
finishing the proof.
\end{proof}

Before we show linear independence of the set $\calB^{\rhm}_{D}$, 
we highlight the possibility of determining the functions constructed here by solving certain PDEs. 
This can be used for numerically determining a basis for $\harmDom$.

\begin{remark}[Characterisation by PDEs]
\mbox{}
\begin{itemize}
\item[\bf(i)]
It is not difficult to see that 
$\psi_{\ell}\in H^{1}_{0}(\om)$ as in \eqref{eq:psiDVF} 
can be found as the solution of the standard variational formulation
$$\forall\,\phi\in H^{1}_{0}(\om)\qquad
\scp{\grad\psi_{\ell}}{\grad\phi}_{\Lttom}
=\scp{\grad\xi_{\ell}}{\grad\phi}_{\Lttom},$$
i.e., $\psi_{\ell}=\Delta_{D}^{-1}\Delta\xi_{\ell}$, 
where $\Delta_{D}=\dive\rgrad$ denotes the Laplacian 
with standard homogeneous Dirichlet boundary conditions on $\om$.
\item[\bf(ii)]
As a consequence of (i) and \eqref{eq:psiDVF}, we obtain 
$u_{\ell}
=\xi_{\ell}-\psi_{\ell}
=(1-\Delta_{D}^{-1}\Delta)\xi_{\ell}\in H^{1}(\om)$
and
\begin{align*}
\grad u_{\ell}
%&=\grad(1-\Delta_{D}^{-1}\Delta)\xi_{\ell}\\
%&=(\grad-\grad\Delta_{D}^{-1}\Delta)\xi_{\ell}\\
&=(1-\grad\Delta_{D}^{-1}\dive)\grad\xi_{\ell}.
\end{align*}
Let us also mention that for $\ell\in\{1,\ldots,m-1\}$, $u_{\ell}$
solves in classical terms the Dirichlet Laplace problem
\begin{align}
\label{DirLap2}
\begin{aligned}
-\Delta u_{\ell}=-\dive\grad u_{\ell}&=0
&&\text{in }\om,\\
u_{\ell}&=1
&&\text{on }\ga_{\ell},\\
u_{\ell}&=0
&&\text{on }\ga_{k},\, k\in\{0,\dots,m-1\}\setminus\{\ell\},
\end{aligned}
\end{align}
which is uniquely solvable.
In particular, for all $\ell\in\{1,\dots,m-1\}$, $u_{\ell}=0$ on $\ga_{0}$. 
\item[\bf(iii)]
$u$ (representing $H=\grad u$)
constructed in the proof of Lemma \ref{dimDFrhm:lem1} 
solves the linear Dirichlet Laplace problem
\begin{align*}
-\Delta u=-\dive\grad u=-\dive H&=0
&&\text{in }\om,\\
u&=0
&&\text{on }\ga_{0},\\
u&=\alpha_{\ell}\in\R
&&\text{on }\ga_{\ell},\,\ell \in\{1,\dots,m-1\},
\end{align*}
which is uniquely solvable as long as the constants, $\alpha_{\ell}$, are prescribed. 
\end{itemize}
\end{remark}

\begin{lemma}
\label{dimDFrhm:lem2}
Let Assumption \ref{ass:segmentprop} be satisfied. Then
$\calB^{\rhm}_{D}$ is linearly independent.
\end{lemma}

\begin{proof}
Let $\alpha_{\ell}\in\R$ for all $\ell\in\{1,\ldots,m-1\}$ such that
$$\sum_{\ell=1}^{m-1}\alpha_{\ell}\grad u_{\ell}=0;\qquad\text{set }
u:=\sum_{\ell=1}^{m-1}\alpha_{\ell}u_{\ell}.$$
Then $\grad u=0$ in $\om$, i.e., $u$ is constant in each connected component of $\om$.
We show $u=0$. Since $\psi_{\ell}\in H_{0}^1(\om)$ and $\xi_{\ell}\in H_{0}^1(B)$ 
we can extend $u_{\ell}=\xi_{\ell}-\psi_{\ell}$ from \eqref{eq:psiDVF} to $B$ by setting
$$\widetilde{u}_{\ell}
:=\begin{cases}
u_{\ell}&\text{ in }\om,\\
\xi_{\ell}&\text{ in }B\setminus\overline{\om},
\end{cases}\qquad
\grad\widetilde{u}_{\ell}
=\begin{cases}
\grad u_{\ell}&\text{ in }\om,\\
\grad\xi_{\ell}=0&\text{ in }B\setminus\overline{\om}.
\end{cases}$$
Note $\widetilde{u}_{\ell}\in H_{0}^1(B)$.
Moreover, for all $\ell\in\{1,\ldots,m-1\}$, 
we have $\widetilde{u}_{\ell}=\xi_{\ell}=1$ in $\Xi_{\ell}$ 
and $\widetilde{u}_{\ell}=\xi_{\ell}=0$ 
in $\widetilde{\Xi}_{0}\cup\bigcup_{k\in\{1,\ldots,m-1\}\setminus\{\ell\}}\Xi_{k}$. Then
$$\widetilde{u}:=\sum_{\ell=1}^{m-1}\alpha_{\ell}\widetilde{u}_{\ell}\in H^{1}_{0}(B)$$
with $\widetilde{u}=0$ in $\widetilde{\Xi}_{0}$ and $\grad\widetilde{u}=0$ in $B\setminus\overline{\om}$ 
as well as $\grad\widetilde{u}=\grad u=0$ in $\om$ by assumption.
Hence, $\grad\widetilde{u}=0$ in $B$, showing $\widetilde{u}=0$ in $B$.
In particular, $u=0$ in $\om$, 
and $\alpha_{\ell}=\widetilde{u}|_{\Xi_{\ell}}=0$ for all $\ell\in\{1,\dots,m-1\}$,
finishing the proof.
\end{proof}

\begin{theorem}
\label{dimDFrhm}
Let Assumption \ref{ass:segmentprop} be satisfied.
Then $\dim\harmDom=m-1$ and a basis of $\harmDom$ is given by \eqref{basis:rhmD}.
\end{theorem}

\begin{proof}
Use Lemma \ref{dimDFrhm:lem1} and Lemma \ref{dimDFrhm:lem2}.
\end{proof}

\subsection{\except{toc}{Dirichlet Tensor Fields of the }First \PZ Complex}
\label{app:sec:DirTFbih1}

For the first \pz complex, see also \eqref{deco1}, \eqref{deco2}, and \eqref{app:deco1-derham},
we have the orthogonal decompositions 
\begin{align}
\label{app:deco1-bih1}
\begin{aligned}
\LtttSom&=\ran(\rGg,\om)\oplus_{\LtttSom}\ker(\dDS,\om),\\
\ker(\rCS,\om)&=\ran(\rGg,\om)\oplus_{\LtttSom}\harmbihoneDom.
\end{aligned}
\end{align}

\begin{remark}
\label{rem:Helmholtz-bih1}
By \cite[Lemma 3.3]{PZ2020a} we have $\dom(\rGg,\om)=H^{2}_{0}(\om)$.
Moreover, the range in \eqref{app:deco1-bih1} is closed by the Friedrichs type estimate
$$\exists\,c>0\quad\forall\,\phi\in H^{2}_{0}(\om)\qquad
\norm{\phi}_{H^{1}(\om)}\leq c\norm{\Gg\phi}_{\Ltttom},$$
which holds by Assumption \ref{ass:segmentprop}.
Similar to Remark \ref{rem:Helmholtz-derham} it suffices to have $\om$ to be open and bounded.
\end{remark}

We define $\pi:\LtttSom\to \ker(\dDS,\om)$ to be the projector onto 
$\ker(\dDS,\om)$ along $\ran(\rGg,\om)$. 
By \eqref{app:deco1-bih1} we obtain $\pi\big(\ker(\rCS,\om)]\big)=\harmbihoneDom$. 
We recall the functions $\xi_{\ell}$ from \eqref{def:xi-ell}. 
In contrast to the derivation for the de Rham complex, 
here the second order nature of $\Gg$ necessitates the introduction of polynomials $\widehat{p}_{j}$
given by 
$$\widehat{p}_{0}(x):=1,\quad
\widehat{p}_{j}(x):=x_{j}\qquad
\big(x=(x_{1},x_{2},x_{3})^{\top}\in\R^3\big)$$
for $j\in\{1,2,3\}$. We define
$\xi_{\ell,j}:=\xi_{\ell}\widehat{p}_{j}$ for all $\ell\in\{1,\dots,m-1\}$
and $j\in\{0,\dots,3\}$.
In particular, for all $j\in\{0,\dots,3\}$ and $\ell\in\{1,\dots,m-1\}$ we have 
$\xi_{\ell,j}=0$ in a neighbourhood of $\Xi_{k}$ for all $k\in\{0,\dots,m-1\}\setminus\{\ell\}$
and $\xi_{\ell,j}=\widehat{p}_{j}$ in a neighbourhood of $\Xi_{\ell}$. Then 
$$\Gg\xi_{\ell,j}\in\CicttSom\cap\ker(\CS,\om)\subset\ker(\rCS,\om).$$
By \eqref{app:deco1-bih1} (and the Friedrichs type estimate for $\rGg$, see Remark \ref{rem:Helmholtz-bih1}) 
there exists a unique $\psi_{\ell,j}\in H^{2}_{0}(\om)$ such that 
$$\harmbihoneDom\ni\pi\Gg\xi_{\ell,j}=\Gg(\xi_{\ell,j}-\psi_{\ell,j})=\Gg u_{\ell,j},$$
where
\begin{equation}
\label{eq:ulj}
u_{\ell,j}:=\xi_{\ell,j}-\psi_{\ell,j}\in H^{2}(\om).
\end{equation}
We shall show that
\begin{align}
\label{basis:bih1D}
\calB^{\bihone}_{D}:=\big\{\Gg u_{\ell,j}: \ell\in\{1,\ldots,m-1\},\,j\in\{0,\ldots,3\}\big\}
\subset\harmbihoneDom
\end{align}
defines a basis of $\harmbihoneDom$. 
In order to show that the linear hull of $\calB^{\bihone}_{D}$ 
generates $\harmbihoneDom$, we cite the following prerequisite.

\begin{lemma}[{{\cite[Theorem 3.10 (i) and Remark 3.11 (i)]{PZ2020a}}}]
\label{lem:toptrivbih1} 
Let $D\subset\R^3$ be a bounded strong Lipschitz domain. 
Assume $D$ is \emph{topologically trivial}, i.e., 
$D$ is simply connected and $\R^3\setminus D$ is connected. Then 
$$\ker(\rCS,D)=\ran(\rGg,D),\qquad
\ker(\CS,D)=\ran(\Gg,D).$$
\end{lemma}

\begin{lemma}
\label{dimDTbih1:lem1}
Let Assumption \ref{ass:segmentprop} be satisfied. Then
$\harmbihoneDom=\lin\calB^{\bihone}_{D}$.
\end{lemma}

\begin{proof}
We follow in close lines the arguments used in the proof of Lemma \ref{dimDFrhm:lem1}.
For this, let $S\in\harmbihoneDom=\ker(\rCS,\om)\cap\ker(\dDS,\om)$.
In particular, by the homogeneous boundary condition its extension by zero, $\widetilde{S}$,
to $B$ belongs to $\ker(\rCS,B)$.
As $B$ is topologically trivial (and smooth and bounded), there exists (a unique) 
$u\in H^{2}_{0}(B)$ such that $\Gg u=\widetilde{S}$ in $B$, 
see Lemma \ref{lem:toptrivbih1} and \eqref{app:deco1-bih1} applied to $\om=D=B$. 
Since $\Gg u=\widetilde{S}=0$ in $B\setminus\overline{\om}$,
$u$ must belong to $\Pol^{1}$, the polynomials of order at most $1$, 
in each connected component $\widetilde{\Xi}_{0},\Xi_{1},\dots,\Xi_{m-1}$
of $B\setminus\overline{\om}$.
Due to the homogenous boundary condition at $\p\!B$, $u$ vanishes in $\widetilde{\Xi}_{0}$.
Therefore, $S=\Gg u$ in $\om$ and $u\in H^{2}_{0}(B)$ 
is such that $u|_{\widetilde{\Xi}_{0}}=0$
and $u|_{\Xi_{\ell}}=:p_{\ell}=:\sum_{j=0}^{3}\alpha_{\ell,j}\widehat{p}_{j}\in\Pol^{1}$,
for some unique $\alpha_{\ell,j}\in\R$, for all $\ell\in\{1,\dots,m-1\}$ and $j\in\{0,\ldots,3\}$. 
Let us consider 
\begin{align*}
\widehat{S}&:=S-\sum_{\ell=1}^{m-1}\sum_{j=0}^{3}\alpha_{\ell,j}\Gg u_{\ell,j}
=\Gg\widehat{u}\in\harmbihoneDom,\\
\widehat{u}&:=u-\sum_{\ell=1}^{m-1}\sum_{j=0}^{3}\alpha_{\ell,j}u_{\ell,j}\in H^{2}(\om)
\end{align*}
with $u_{\ell,j}$ from \eqref{eq:ulj}.
The extension $\widetilde{\psi}_{\ell,j}$ of $\psi_{\ell,j}$ by zero to $B$ belongs to $H^{2}_{0}(B)$.
Hence as an element of $H^{2}(B)$ we see that 
$$\widehat{u}_{B}
:=u-\sum_{\ell=1}^{m-1}\sum_{j=0}^{3}\alpha_{\ell,j}\xi_{\ell,j}
+\sum_{\ell=1}^{m-1}\sum_{j=0}^{3}\alpha_{\ell,j}\widetilde{\psi}_{\ell,j}\in H^{2}_{0}(B)$$
vanishes in all $\Xi_{\ell}$. 
Thus $\widehat{u}=\widehat{u}_{B}|_{\om}\in H^{2}_{0}(\om)$ by Assumption \ref{ass:segmentprop}, and we compute
$$\norm{\widehat{S}}_{\LtttSom}^{2}
=\scp{\Gg\widehat{u}}{\widehat{S}}_{\LtttSom}=0,$$
finishing the proof.
\end{proof}

Similar to the case of the de Rham complex, 
we have a look at a possible numerical implementation for the computation of the basis functions. 
Naturally, the PDEs in question differ from one another quite substantially.

\begin{remark}[Characterisation by PDEs]
\mbox{}
\begin{itemize}
\item[\bf(i)]
The functions $\psi_{\ell,j}\in H^{2}_{0}(\om)$ introduced just above \eqref{eq:ulj} 
can be characterised as solutions by the standard variational formulation
$$\forall\,\phi\in H^{2}_{0}(\om)\quad
\scp{\Gg\psi_{\ell,j}}{\Gg\phi}_{\LtttSom}
=\scp{\Gg\xi_{\ell,j}}{\Gg\phi}_{\LtttSom},$$
i.e., $\psi_{\ell,j}=({\Delta}_{DD}^{2})^{-1}\Delta^{2}\xi_{\ell,j}$, 
where $\Delta^{2}_{DD}=\dDS\rGg$ is the bi-Laplacian with both the functions as well as 
the derivatives satisfying homogeneous Dirichlet boundary conditions.
\item[\bf(ii)]
With the statement in (i) together with \eqref{eq:ulj}, 
we deduce for all $\ell\in\{1,\ldots,m-1\}$ and $j\in\{0,\ldots,3\}$
$$u_{\ell,j}
=\xi_{\ell,j}-\psi_{\ell,j}
=\big(1-({\Delta}_{DD}^{2})^{-1}\Delta^{2}\big)\xi_{\ell,j}\in H^{2}(\om).$$
Hence,
\begin{align*}
\Gg u_{\ell,j}
%&=\Gg\big(1-({\Delta}_{DD}^{2})^{-1}\Delta^{2}\big)\xi_{\ell,j}\\
%&=\big(\Gg-\Gg({\Delta}_{DD}^{2})^{-1}\Delta^{2}\big)\xi_{\ell,j}\\
&=\big(1-\Gg({\Delta}_{DD}^{2})^{-1}\dDS\big)\Gg\xi_{\ell,j}.
\end{align*}
For all $\ell\in\{1,\ldots,m-1\}$ and $j\in\{0,\ldots,3\}$, $u_{\ell,j}$ 
solve in classical terms the \pz Dirichlet problem
\begin{align}
\label{Dirbih2}
\begin{aligned}
\Delta^{2}u_{\ell,j}=\dDS\Gg u_{\ell,j}&=0
&&\text{in }\om,\\
u_{\ell,j}=\widehat{p}_{j},\quad
\grad u_{\ell,j}=\grad\widehat{p}_{j}&=e^{j}
&&\text{on }\ga_{\ell},\\
u_{\ell,j}=0,\quad
\grad u_{\ell,j}&=0
&&\text{on }\ga_{k},\,k\in\{0,\dots,m-1\}\setminus\{\ell\},
\end{aligned}
\end{align} 
which is uniquely solvable.
In particular, we have for all $\ell\in\{1,\dots,m-1\}$ and all $j\in\{0,\dots,3\}$ 
that $u_{\ell,j}=0$ and $\grad u_{\ell,j}=0$ on $\ga_{0}$, where
we denote by $e^{j}$, $j\in\{1,2,3\}$,
the Euclidean unit vectors in $\R^{3}$ and set $e^{0}:=0\in\R^{3}$.
\item[\bf(iii)]
In classical terms, $u$ (representing $S=\Gg u$)
derived in the proof of Lemma \ref{dimDTbih1:lem1} 
solves the linear \pz Dirichlet problem
\begin{align*}
\Delta^{2}u=\dDS\Gg u=\dDS S&=0
&&\text{in }\om,\\
u=0,\quad
\grad u&=0
&&\text{on }\ga_{0},\\
u=p_{\ell}\in\Pol^{1},\quad
\grad u&=\grad p_{\ell}\in\R^3
&&\text{on }\ga_{\ell},\,\ell\in\{1,\dots,m-1\},
\end{align*}
which is uniquely solvable as long as the polynomials, 
$p_{\ell}$ in $\Pol^{1}$, are prescribed.
\end{itemize}
\end{remark}

\begin{lemma}
\label{dimDTbih1:lem2}
Let Assumption \ref{ass:segmentprop} be satisfied. Then
$\calB^{\bihone}_{D}$ is linearly independent.
\end{lemma}

\begin{proof}
For $\ell\in\{1,\ldots,m-1\}$ and $j\in\{0,\ldots,3\}$
we take $\alpha_{\ell,j}\in\R$ such that
$$\sum_{\ell=1}^{m-1}\sum_{j=0}^{3}\alpha_{\ell,j}\Gg u_{\ell,j}=0;\qquad\text{we put }
u:=\sum_{\ell=1}^{m-1}\sum_{j=0}^{3}\alpha_{\ell,j}u_{\ell,j}.$$
Then $\Gg u=0$ in $\om$, i.e., $u$ belongs to $\Pol^{1}_{\pw}$, see \eqref{eq:P1_pw}.
We will show $u=0$.
For this we extend $u_{\ell,j}=\xi_{\ell,j}-\psi_{\ell,j}$ (see \eqref{eq:ulj})
to $B$ via (note that $\xi_{\ell,j}\in H_{0}^{2}(B)$ and $\psi_{\ell,j}\in H_{0}^{2}(\om)$)
$$\widetilde{u}_{\ell,j}
:=\begin{cases}
u_{\ell,j}&\text{ in }\om,\\
\xi_{\ell,j}&\text{ in }B\setminus\overline{\om},
\end{cases}\qquad
\Gg\widetilde{u}_{\ell,j}
=\begin{cases}
\Gg u_{\ell,j}&\text{ in }\om,\\
\Gg\xi_{\ell,j}=0&\text{ in }B\setminus\overline{\om}.
\end{cases}$$
Note that $\widetilde{u}_{\ell,j}\in H_{0}^{2}(B)$.
For all $\ell\in\{1,\ldots,m-1\}$, $j\in\{0,\ldots,3\}$ 
we have $\widetilde{u}_{\ell,j}=\xi_{\ell,j}=\widehat{p}_{j}$ in $\Xi_{\ell}$ 
and $\widetilde{u}_{\ell,j}=\xi_{\ell,j}=0$ 
in $\widetilde{\Xi}_{0}\cup\bigcup_{k\in\{1,\ldots,m-1\}\setminus\{\ell\}}\Xi_{k}$. Then
$$\widetilde{u}:=\sum_{\ell=1}^{m-1}\sum_{j=0}^{3}\alpha_{\ell,j}\widetilde{u}_{\ell,j}\in H^{2}_{0}(B)$$
with $\widetilde{u}=0$ in $\widetilde{\Xi}_{0}$ and $\Gg\widetilde{u}=0$ in $B\setminus\overline{\om}$ 
as well as $\Gg\widetilde{u}=\Gg u=0$ in $\om$ by assumption.
Hence, $\Gg\widetilde{u}=0$ in $B$, showing $\widetilde{u}=0$ in $B$.
In particular, $u=0$ in $\om$, and
$\sum_{j=0}^{3}\alpha_{\ell,j}\widehat{p}_{j}=\widetilde{u}|_{\Xi_{\ell}}=0$ 
for all $\ell\in\{1,\dots,m-1\}$. We conclude $\alpha_{\ell,j}=0$ for all $j\in\{0,\dots,3\}$ and all $\ell\in\{1,\ldots,m-1\}$,
finishing the proof.
\end{proof}

\begin{theorem}
\label{dimDTbih1}
Let Assumption \ref{ass:segmentprop} be satisfied.
Then $\dim\harmbihoneDom=4(m-1)$ and a basis of $\harmbihoneDom$ is given by \eqref{basis:bih1D}.
\end{theorem}

\begin{proof}
Use Lemma \ref{dimDTbih1:lem1} and Lemma \ref{dimDTbih1:lem2}.
\end{proof}

\subsection{\except{toc}{Dirichlet Tensor Fields of the }Second \PZ Complex}
\label{app:sec:DirTFbih2}

The rationale to derive a set of basis functions for the second \pz complex is somewhat similar to the first one. 
For the second \pz complex, similar to \eqref{deco1}, \eqref{deco2}, 
and \eqref{app:deco1-derham}, \eqref{app:deco1-bih1}, 
we have the orthogonal decompositions 
\begin{align}
\label{app:deco1-bih2}
\begin{aligned}
\LtttTom&=\ran(\rdevGrad,\om)\oplus_{\LtttTom}\ker(\DT,\om),\\
\ker(\rsymCT,\om)&=\ran(\rdevGrad,\om)\oplus_{\LtttTom}\harmbihtwoDom.
\end{aligned}
\end{align}

\begin{remark}
\label{rem:Helmholtz-bih2}
\cite[Lemma 3.2]{PZ2020a} yields $\dom(\rdevGrad,\om)=H^{1,3}_{0}(\om)$.
Moreover, the range in \eqref{app:deco1-bih2} is closed by the Friedrichs type estimate
\begin{equation}\label{eq:Hbih2}\exists\,c>0\quad\forall\,\phi\in H^{1,3}_{0}(\om)\qquad
\norm{\phi}_{\Lttom}\leq c\norm{\devGrad\phi}_{\Ltttom},\end{equation}
which holds by Assumption \ref{ass:segmentprop}.
Again, $\om$ being open and bounded would be sufficient already. 
Indeed, the estimate mentioned here is based on the Friedrichs estimate provided 
in Remark \ref{rem:Helmholtz-derham} and the following observations
similar to the proof of Korn's inequality, cf.~Remark \ref{rem:Helmholtz-ela}:
From $\dom(\rdevGrad,\om)=H^{1,3}_{0}(\om)$ it suffices to show \eqref{eq:Hbih2} 
for smooth vector fields $v$ with compact support in $\om$. 
It is elementary to see that for matrices $T$ in $\R^{3\times 3}$ 
and the Frobenius norm $\norm{T}_{\R^{3\times 3}}$
we have $\norm{T}_{\R^{3\times 3}}^{2}=\norm{\dev T}_{\R^{3\times 3}}^{2}+\frac{1}{3}\norm{\tr T}_{\R}^{2}$,
where $\dev T=T-\frac13(\tr{T})\id$ is the \emph{deviatoric} (`trace-free') part of $T$ 
and $\tr T$ is the trace of $T$. Integration by parts shows 
$\norm{\Grad v}_{\Ltttom}^{2}=\norm{\curl v}_{\Lttom}^{2}+\norm{\dive v}_{\Ltom}^{2}\geq\norm{\dive v}_{\Ltom}^{2}$ 
for all $v\in C_{c}^{\infty,3}(\om)$. Thus, from $\tr \Grad v=\dive v$ we infer
\begin{align*}
\norm{\Grad v}_{\Ltttom}^{2}
&=\norm{\dev\Grad v}_{\Ltttom}^{2}+\frac13\norm{\dive v}_{\Ltom}^{2}\\
&\leq\norm{\dev\Grad v}_{\Ltttom}^{2}
+\frac13\norm{\Grad v}_{\Ltttom}^{2}.
\end{align*}
Hence, $2\norm{\Grad v}_{\Ltttom}^{2}\leq3\norm{\devGrad v}_{\Ltttom}^{2}$,
and inequality \eqref{eq:Hbih2} follows from Remark \ref{rem:Helmholtz-derham}.
\end{remark}

Using \eqref{app:deco1-bih2}, 
we define the orthogonal projector $\pi:\LtttTom\to\ker(\DT,\om)$ along $\ran(\rdevGrad,\om)$
and we have $\pi(\ker(\rsymCT,\om))=\harmbihtwoDom$.
Recalling $\xi_{\ell}\in\Cic(\R^{3})$ from \eqref{def:xi-ell} and introducing the Raviart--Thomas fields $\widehat{r}_{j}$ 
given by 
$$\widehat{r}_{0}(x):=x,\qquad
\widehat{r}_{j}(x):=e^{j}$$ 
for $j\in\{1,2,3\}$, we define
$\xi_{\ell,j}:=\xi_{\ell}\widehat{r}_{j}$ 
for all $\ell\in\{1,\dots,m-1\}$ 
and all $j\in\{0,\dots,3\}$. It is easy to see that 
$$\devGrad\xi_{\ell,j}\in\CicttTom\cap\ker(\symCT,\om)\subset\ker(\rsymCT,\om).$$
Due to Remark \ref{rem:Helmholtz-bih2} in conjunction with \eqref{app:deco1-bih2}, 
we find unique $\psi_{\ell,j}\in H^{1,3}_{0}(\om)$ such that 
$$\harmbihtwoDom\ni\pi\devGrad\xi_{\ell,j}=\devGrad(\xi_{\ell,j}-\psi_{\ell,j})=\devGrad v_{\ell,j}$$
with 
\begin{equation}
\label{eq:vlj}
v_{\ell,j}:=\xi_{\ell,j}-\psi_{\ell,j}\in H^{1,3}(\om).
\end{equation}
We shall show that
\begin{align}
\label{basis:bih2D}
\calB^{\bihtwo}_{D}:=\big\{\devGrad v_{\ell,j}:\ell\in\{1,\ldots,m-1\},\,j\in\{0,\ldots,3\}\big\}
\subset\harmbihtwoDom
\end{align}
defines a basis of $\harmbihtwoDom$.

\begin{lemma}[{{\cite[Theorem 3.10 (iv) and Remark 3.11 (i)]{PZ2020a}}}]
\label{lem:toptrivbih2} 
Let $D\subset\R^3$ be a bounded strong Lipschitz domain. 
Assume that $D$ is topologically trivial. Then 
$$\ker(\rsymCT,D)=\ran(\rdevGrad,D),\qquad
\ker(\symCT,D)=\ran(\devGrad,D).$$
\end{lemma}

\begin{lemma}
\label{dimDTbih2:lem1}
Let Assumption \ref{ass:segmentprop} be satisfied. Then
$\harmbihtwoDom=\lin\calB^{\bihtwo}_{D}$.
\end{lemma}

\begin{proof}
Let $T\in\harmbihtwoDom=\ker(\rsymCT,\om)\cap\ker(\DT,\om)$ and let $\widetilde{T}$ 
be the extension of $T$ by zero onto $B$. Then $\widetilde{T}\in\ker(\rsymCT,B)$.
As $B$ is topologically trivial (and smooth and bounded), by Lemma \ref{lem:toptrivbih2} there exists (a unique vector field) 
$v\in H^{1,3}_{0}(B)$ such that $\devGrad v=\widetilde{T}$ in $B$.
Since $\devGrad v=\widetilde{T}=0$ in $B\setminus\overline{\om}$,
$v$ is a Raviart--Thomas vector field, $v\in\RT$, 
in each connected component $\widetilde{\Xi}_{0},\Xi_{1},\dots,\Xi_{m-1}$
of $B\setminus\overline{\om}$. 
Due to the boundary condition of $v\in H^{1,3}_{0}(B)$, $v$ vanishes in $\widetilde{\Xi}_{0}$.
Therefore, $T=\devGrad v$ in $\om$ and $v\in H^{1,3}_{0}(B)$ 
is such that $v|_{\widetilde{\Xi}_{0}}=0$
and $v|_{\Xi_{\ell}}=:r_{\ell}=:\sum_{j=0}^{3}\alpha_{\ell,j}\widehat{r}_{j}\in\RT$, for some
$\alpha_{\ell,j}\in\R$, for all $\ell\in\{1,\dots,m-1\}$ and $j\in\{0,\ldots,3\}$. 
Define
\begin{align*}
\widehat{T}&:=T-\sum_{\ell=1}^{m-1}\sum_{j=0}^{3}\alpha_{\ell,j}\devGrad v_{\ell,j}
=\devGrad\widehat{v}\in\harmbihtwoDom,\\
\widehat{v}&:=v-\sum_{\ell=1}^{m-1}\sum_{j=0}^{3}\alpha_{\ell,j}v_{\ell,j}\in H^{1,3}(\om)
\end{align*}
with $v_{\ell,j}$ from \eqref{eq:vlj}.
Since $\widetilde{\psi}_{\ell,j}\in H^{1,3}_{0}(B)$, where $\widetilde{\psi}_{\ell,j}$ is the extension of $\psi_{\ell,j}$ by zero to $B$, as an element of $H^{1,3}(B)$ we see that 
$$\widehat{v}_{B}
:=v-\sum_{\ell=1}^{m-1}\sum_{j=0}^{3}\alpha_{\ell,j}\xi_{\ell,j}
+\sum_{\ell=1}^{m-1}\sum_{j=0}^{3}\alpha_{\ell,j}\widetilde{\psi}_{\ell,j}\in H^{1,3}_{0}(B)$$
vanishes in all $\Xi_{\ell}$. 
Thus $\widehat{v}=\widehat{v}_{B}|_{\om}\in H^{1,3}_{0}(\om)$ by Assumption \ref{ass:segmentprop}, and $$\norm{\widehat{T}}_{\LtttTom}^{2}
=\scp{\devGrad\widehat{v}}{\widehat{T}}_{\LtttTom}=0$$
yields the assertion.
\end{proof}

\begin{remark}[Characterisation by PDEs]
\mbox{}
\begin{itemize}
\item[\bf(i)]
Denoting $\Delta_{\Tb,D}:=\DT\rdevGrad$ the `deviatoric' Laplacian 
with homogeneous Dirichlet boundary conditions, 
we see that $\psi_{\ell,j}={\Delta}_{\Tb,D}^{-1}\Delta_{\Tb}\xi_{\ell,j}$ 
with $\Delta_{\Tb}:=\DT\devGrad$, which corresponds to the variational formulation
$$\forall\,\phi\in H^{1,3}_{0}(\om)\quad
\scp{\devGrad\psi_{\ell,j}}{\devGrad\phi}_{\LtttTom}
=\scp{\devGrad\xi_{\ell,j}}{\devGrad\phi}_{\LtttTom}.$$
\item[\bf(ii)]
For all $\ell\in\{1,\dots,m-1\}$ and all $j\in\{0,\dots,3\}$ we have
$$v_{\ell,j}
=\xi_{\ell,j}-\psi_{\ell,j}
=(1-\Delta_{\Tb,D}^{-1}\Delta_{\Tb})\xi_{\ell,j}\in H^{1,3}(\om)$$
and deduce
\begin{align*}
\devGrad v_{\ell,j}
%&=\devGrad(1-\Delta_{\Tb,D}^{-1}\Delta_{\Tb})\xi_{\ell,j}\\
%&=(\devGrad-\devGrad\Delta_{\Tb,D}^{-1}\Delta_{\Tb})\xi_{\ell,j}\\
&=(1-\devGrad\Delta_{\Tb,D}^{-1}\DT)\devGrad\xi_{\ell,j}.
\end{align*}
In classical terms, this reads
\begin{align}
\label{Dirbih22}
\begin{aligned}
-\Delta_{\Tb}v_{\ell,j}&=0
&&\text{in }\om,\\
v_{\ell,j}&=\widehat{r}_{j}
&&\text{on }\ga_{\ell},\\
v_{\ell,j}&=0
&&\text{on }\ga_{k},\,k\in\{0,\dots,m-1\}\setminus\{\ell\},
\end{aligned}
\end{align} 
which is uniquely solvable.
\item[\bf(iii)]
In classical terms, $v$ (representing $T=\devGrad v$) 
from the proof of Lemma \ref{dimDTbih2:lem1}
solves the linear elasticity type Dirichlet problem
\begin{align*}
-\Delta_{\Tb}v=-\DT\devGrad v=-\DT T&=0
&&\text{in }\om,\\
v&=0
&&\text{on }\ga_{0},\\
v&=r_{\ell}\in\RT
&&\text{on }\ga_{\ell},\,\ell\in\{1,\dots,m-1\},
\end{align*}
which is uniquely solvable given the knowledge of $r_{\ell}$ in $\RT$.
\end{itemize}
\end{remark}

\begin{lemma}
\label{dimDTbih2:lem2}
Let Assumption \ref{ass:segmentprop} be satisfied. Then
$\calB^{\bihtwo}_{D}$ is linearly independent.
\end{lemma}

\begin{proof}
Let $\alpha_{\ell,j}\in\R$ with $\ell\in\{1,\ldots,m-1\}$ and $j\in\{0,\ldots,3\}$ be such that
$$\sum_{\ell=1}^{m-1}\sum_{j=0}^{3}\alpha_{\ell,j}\devGrad v_{\ell,j}=0;\qquad\text{set }
v:=\sum_{\ell=1}^{m-1}\sum_{j=0}^{3}\alpha_{\ell,j}v_{\ell,j}.$$
Then $\devGrad v=0$ in $\om$, i.e., $v\in\RT$ in each connected component of $\om$.
We show $v=0$. Recalling
$v_{\ell,j}=\xi_{\ell,j}-\psi_{\ell,j}$ in $\om$ from \eqref{eq:ulj} 
and using $\xi_{\ell,j} \in H_{0}^1(B)$ and $\psi_{\ell,j}\in H_{0}^1(\om)$, we define
$$\widetilde{v}_{\ell,j}
:=\begin{cases}
v_{\ell,j}&\text{ in }\om,\\
\xi_{\ell,j}&\text{ in }B\setminus\overline{\om},
\end{cases}\qquad
\devGrad\widetilde{v}_{\ell,j}
=\begin{cases}
\devGrad v_{\ell,j}&\text{ in }\om,\\
\devGrad\xi_{\ell,j}=0&\text{ in }B\setminus\overline{\om}.
\end{cases}$$
Note that $\widetilde{v}_{\ell,j}\in H^{1,3}_{0}(B)$.
For all $\ell \in\{1,\dots,m-1\}$ and $j\in\{0,\dots,3\}$, 
we obtain $\widetilde{v}_{\ell,j}=\xi_{\ell,j}=\widehat{r}_{j}$ in $\Xi_{\ell}$ 
and $\widetilde{v}_{\ell,j}=\xi_{\ell,j}=0$ 
in $\widetilde{\Xi}_{0}\cup\bigcup_{k\in\{1,\ldots,m-1\}\setminus\{\ell\}}\Xi_{k}$. Then
$$\widetilde{v}:=\sum_{\ell=1}^{m-1}\sum_{j=0}^{3}\alpha_{\ell,j}\widetilde{v}_{\ell,j}\in H^{1,3}_{0}(B)$$
with $\widetilde{v}=0$ in $\widetilde{\Xi}_{0}$ and $\devGrad\widetilde{v}=0$ in $B\setminus\overline{\om}$ 
as well as $\devGrad\widetilde{v}=\devGrad v=0$ in $\om$ by assumption.
Hence, $\devGrad\widetilde{v}=0$ in $B$, showing $\widetilde{v}=0$ in $B$.
In particular, $v=0$ in $\om$, and
$\sum_{j=0}^{3}\alpha_{\ell,j}\widehat{r}_{j}=\widetilde{v}|_{\Xi_{\ell}}=0$ 
for all $\ell\in\{1,\dots,m-1\}$. 
We conclude $\alpha_{\ell,j}=0$ for all $\ell \in\{1,\dots,m-1\}$ and $j\in\{0,\dots,3\}$.
\end{proof}

\begin{theorem}
\label{dimDTbih2}
Let Assumption \ref{ass:segmentprop} be satisfied.
Then $\dim\harmbihtwoDom=4(m-1)$ and a basis of $\harmbihtwoDom$ is given by \eqref{basis:bih2D}.
\end{theorem}

\begin{proof}
Use Lemma \ref{dimDTbih2:lem1} and Lemma \ref{dimDTbih2:lem2}.
\end{proof}

\subsection{\except{toc}{Dirichlet Tensor Fields of the }Elasticity Complex}
\label{app:sec:DirTFela}

For the elasticity complex, similar to \eqref{deco1}, \eqref{deco2}, 
and \eqref{app:deco1-derham}, \eqref{app:deco1-bih1}, \eqref{app:deco1-bih2}, 
we have the orthogonal decompositions 
\begin{align}
\label{app:deco1-ela}
\begin{aligned}
\LtttSom&=\ran(\rsymGrad,\om)\oplus_{\LtttSom}\ker(\DS,\om),\\
\ker(\rCCtS,\om)&=\ran(\rsymGrad,\om)\oplus_{\LtttSom}\harmelaDom.
\end{aligned}
\end{align}

\begin{remark}
\label{rem:Helmholtz-ela}
\cite[Lemma 3.2]{PZ2020b} implies $\dom(\rsymGrad,\om)=H^{1,3}_{0}(\om)$. 
Moreover, the range in \eqref{app:deco1-ela} is closed by the Friedrichs type estimate 
(and follows from the standard first Korn's inequality and Remark \ref{rem:Helmholtz-derham})
\begin{equation}
\label{eq:Hela}
\exists\,c>0\quad\forall\,\phi\in H^{1,3}_{0}(\om)\qquad
\norm{\phi}_{\Lttom}\leq c\norm{\symGrad\phi}_{\Ltttom},
\end{equation}
which holds by Assumption \ref{ass:segmentprop}.
Again, $\om$ open and bounded is sufficient for \eqref{eq:Hela}.
Indeed, Korn's first inequality is easy to see as follows:
For a tensor $T\in\R^{3\times 3}$ we have 
$\norm{T}^{2}_{\R^{3\times 3}}=\norm{\sym T}^{2}_{\R^{3\times 3}}+\norm{\skw T}^{2}_{\R^{3\times 3}}$. Hence, 
$$\norm{\Grad v}^{2}_{\R^{3\times 3}}
=\norm{\symGrad v}^{2}_{\R^{3\times 3}}
+\norm{\skw\Grad v}^{2}_{\R^{3\times 3}}
=\norm{\symGrad v}^{2}_{\R^{3\times 3}}
+\frac{1}{2}\norm{\curl v}^{2}_{\R^{3}}.$$ 
By $\norm{\Grad v}_{\Ltttom}^{2}=\norm{\curl v}_{\Lttom}^{2}+\norm{\dive v}_{\Ltom}^{2}\geq\norm{\curl v}_{\Lttom}^{2}$
for all $v\in H^{1,3}_{0}(\om)$, we get Korn's first inequality
$\norm{\Grad v}_{\Ltttom}^{2}\leq2\norm{\symGrad v}_{\Ltttom}^{2}$.
\end{remark}

The orthogonal projector from $\LtttSom$ onto $\ker(\DS,\om)$ along $\ran(\rsymGrad,\om)$ is denoted by $\pi$. 
From \eqref{app:deco1-ela}, we deduce $\pi(\ker(\rCCtS,\om))=\harmelaDom$.
Recall the functions  $\xi_{\ell}\in\Cic(\R^{3})$ from \eqref{def:xi-ell} 
and introduce rigid motions $\widehat{r}_{j}$ given by 
$$\widehat{r}_{j}(x):=e^{j}\times x,\qquad
\widehat{r}_{j+3}(x):=e^{j}$$ 
for $j\in\{1,2,3\}$. We define $\xi_{\ell,j}:=\xi_{\ell}\widehat{r}_{j}$ 
for all $\ell\in\{1,\dots,m-1\}$ and for all $j\in\{1,\dots,6\}$. Then 
$$\symGrad\xi_{\ell,j}\in\CicttSom\cap\ker(\CCtS,\om)\subset\ker(\rCCtS,\om).$$
We find unique $\psi_{\ell,j}\in H^{1,3}_{0}(\om)$ such that 
$$\harmelaDom\ni\pi\symGrad\xi_{\ell,j}=\symGrad(\xi_{\ell,j}-\psi_{\ell,j})=\symGrad v_{\ell,j}$$
with 
\begin{equation}
\label{eq:vljela}
v_{\ell,j}:=\xi_{\ell,j}-\psi_{\ell,j}\in H^{1,3}(\om).
\end{equation}
We shall show that 
\begin{align}
\label{basis:elaD}
\calB^{\ela}_{D}:=\big\{\symGrad v_{\ell,j}:\ell\in\{1,\ldots,m-1\},\,j\in\{1,\ldots,6\}\big\}
\subset\harmelaDom
\end{align}
defines a basis of $\harmelaDom$.

\begin{lemma}[{\cite[Theorem 3.5]{PZ2020b}}]
\label{lem:toptrivela} 
Let $D\subset\R^3$ a bounded, topologically trivial, strong Lipschitz domain. Then 
$$\ker(\rCCtS,D)=\ran(\rsymGrad,D),\qquad
\ker(\CCtS,D)=\ran(\symGrad,D).$$
\end{lemma}

\begin{proof}
The result follows by \cite[Corollary 2.29]{PZ2020a} for $m=1$ 
in conjunction with the formulas in \cite[Appendix]{PZ2020a}; 
see \cite[Theorem 3.5]{PZ2020b} and \cite{PZ2020c} for the details.
\end{proof}

\begin{lemma}
\label{dimDTela:lem1}
Let Assumption \ref{ass:segmentprop} be satisfied. Then
$\harmelaDom=\lin\calB^{\ela}_{D}$.
\end{lemma}

\begin{proof}
We follow in close lines the arguments used in the proofs 
of Lemma \ref{dimDFrhm:lem1}, Lemma \ref{dimDTbih1:lem1}, and Lemma \ref{dimDTbih2:lem1}.
Let $S\in\harmelaDom=\ker(\rCCtS,\om)\cap\ker(\DS,\om)$.
and $\widetilde{S}$ its extension to $B$ by zero. Then $\widetilde{S}\in\ker(\rCCtS,B)$.
By Lemma \ref{lem:toptrivela}, as $B$ is topologically trivial (and smooth and bounded), 
there exists (a unique) 
$v\in H^{1,3}_{0}(B)$ such that $\symGrad v=\widetilde{S}$ in $B$. 
Since $\symGrad v=\widetilde{S}=0$ in $B\setminus\overline{\om}$,
$v$ is a rigid motion, i.e., $v\in\RM$,
in each connected component $\widetilde{\Xi}_{0},\Xi_{1},\dots,\Xi_{m-1}$ 
of $B\setminus\overline{\om}$.
Since $v\in H^{1,3}_{0}(B)$, $v$ vanishes in $\widetilde{\Xi}_{0}$.
Thus, $S=\symGrad v$ in $\om$ with some $v\in H^{1,3}_{0}(B)$ 
and we have $v|_{\Xi_{\ell}}=:r_{\ell}=:\sum_{j=1}^{6}\alpha_{\ell,j}\widehat{r}_{j}\in\RM$
for $\alpha_{\ell,j}\in\R$ and all $\ell\in\{1,\dots,m-1\}$, $j\in\{1,\ldots,6\}$. Let 
\begin{align*}
\widehat{S}&:=S-\sum_{\ell=1}^{m-1}\sum_{j=1}^{6}\alpha_{\ell,j}\symGrad v_{\ell,j}
=\symGrad\widehat{v}\in\harmelaDom,\\
\widehat{v}&:=v-\sum_{\ell=1}^{m-1}\sum_{j=1}^{6}\alpha_{\ell,j}v_{\ell,j}\in H^{1,3}(\om)
\end{align*}
with $v_{\ell,j}$ from \eqref{eq:vljela}.
With $\widetilde{\psi}_{\ell,j} \in H^{1,3}_{0}(B)$, the extension of $\psi_{\ell,j}$ by zero, we see, as an element of $H^{1,3}(B)$, that
$$\widehat{v}_{B}
:=v-\sum_{\ell=1}^{m-1}\sum_{j=1}^{6}\alpha_{\ell,j}\xi_{\ell,j}
+\sum_{\ell=1}^{m-1}\sum_{j=1}^{6}\alpha_{\ell,j}\widetilde{\psi}_{\ell,j}\in H^{1,3}_{0}(B)$$
vanishes in $\Xi_{\ell}$ for all $\ell$. 
Thus $\widehat{v}=\widehat{v}_{B}|_{\om}\in H^{1,3}_{0}(\om)$ by Assumption \ref{ass:segmentprop}, and we conclude
$$\norm{\widehat{S}}_{\LtttSom}^{2}
=\scp{\symGrad\widehat{v}}{\widehat{S}}_{\LtttSom}=0,$$
finishing the proof.
\end{proof}

For numerical purposes, we again highlight the partial differential equations 
satisfied by the functions constructed here.

\begin{remark}[Characterisation by PDEs]
\mbox{}
\begin{itemize}
\item[\bf(i)]
For all $\ell\in\{1,\ldots,m-1\}$, $j\in\{1,\ldots,6\}$, 
the vector field $\psi_{\ell,j}\in H^{1,3}_{0}(\om)$ can be found with the help of the standard variational formulation
$$\forall\,\phi\in H^{1,3}_{0}(\om)\quad
\scp{\symGrad\psi_{\ell,j}}{\symGrad\phi}_{\LtttSom}
=\scp{\symGrad\xi_{\ell,j}}{\symGrad\phi}_{\LtttSom},$$
i.e., $\psi_{\ell,j}={\Delta}_{\Sb,D}^{-1}\Delta_{\Sb}\xi_{\ell,j}$, 
where $\Delta_{\Sb,D}=\DS\rsymGrad$ and $\Delta_{\Sb}=\DS\symGrad$ 
are the `symmetric' Laplacians with homogeneous Dirichlet boundary conditions 
and no boundary conditions, respectively. 
\item[\bf(ii)]
For all $\ell\in\{1,\ldots,m-1\}$, $j\in\{1,\ldots,6\}$ we have
$$v_{\ell,j}
=\xi_{\ell,j}-\psi_{\ell,j}
=(1-\Delta_{\Sb,D}^{-1}\Delta_{\Sb})\xi_{\ell,j}\in H^{1,3}(\om)$$
and thus
\begin{align*}
\symGrad v_{\ell,j}
%&=\symGrad(1-\Delta_{\Sb,D}^{-1}\Delta_{\Sb})\xi_{\ell,j}\\
%&=(\symGrad-\symGrad\Delta_{\Sb,D}^{-1}\Delta_{\Sb})\xi_{\ell,j}\\
&=(1-\symGrad\Delta_{\Sb,D}^{-1}\DS)\symGrad\xi_{\ell,j}.
\end{align*}
In classical terms, $v_{\ell,j}$
solves the linear elasticity Dirichlet problem
\begin{align}
\label{Direla2}
\begin{aligned}
-\Delta_{\Sb}v_{\ell,j}&=0
&&\text{in }\om,\\
v_{\ell,j}&=\widehat{r}_{j}
&&\text{on }\ga_{\ell},\\
v_{\ell,j}&=0
&&\text{on }\ga_{k},\,k\in\{0,\dots,m-1\}\setminus\{\ell\}.
\end{aligned}
\end{align} 
which is uniquely solvable. 
\item[\bf(iii)]
In classical terms, $v$ (representing $S=\symGrad v$) 
from the proof of Lemma \ref{dimDTela:lem1}
solves the linear elasticity Dirichlet problem
\begin{align*}
-\Delta_{\Sb}v=-\DS S&=0
&&\text{in }\om,\\
v&=0
&&\text{on }\ga_{0},\\
v&=r_{\ell}\in\RM
&&\text{on }\ga_{\ell},\,\ell\in\{1,\dots,m-1\},
\end{align*}
which is uniquely solvable as long as the rigid motions
$r_{\ell}$ in $\RM$ are prescribed.
\end{itemize}
\end{remark}

\begin{lemma}
\label{dimDTela:lem2}
Let Assumption \ref{ass:segmentprop} be satisfied. Then
$\calB^{\ela}_{D}$ is linearly independent.
\end{lemma}

\begin{proof}
Let $\alpha_{\ell,j}\in\R$ for all $\ell\in\{1,\dots,m-1\}$, $j\in\{1,\ldots,6\}$ such that 
$$\sum_{\ell=1}^{m-1}\sum_{j=1}^{6}\alpha_{\ell,j}\symGrad v_{\ell,j}=0;\qquad\text{set }
v:=\sum_{\ell=1}^{m-1}\sum_{j=1}^{6}\alpha_{\ell,j}v_{\ell,j}.$$
Then $\symGrad v=0$ in $\om$, i.e., $v\in\RM$ in each connected component of $\om$.
We show $v=0$. Recall $v_{\ell,j}=\xi_{\ell,j}-\psi_{\ell,j}$ in $\om$. 
Using $\psi_{\ell,j}\in H_{0}^1(\om)$ and $\xi_{\ell,j}\in H_{0}^1(B)$
we extend $v_{\ell,j}$ to $B$ via 
$$\widetilde{v}_{\ell,j}
:=\begin{cases}
v_{\ell,j}&\text{ in }\om,\\
\xi_{\ell,j}&\text{ in }B\setminus\overline{\om},
\end{cases}\qquad
\symGrad\widetilde{v}_{\ell,j}
=\begin{cases}
\symGrad v_{\ell,j}&\text{ in }\om,\\
\symGrad\xi_{\ell,j}=0&\text{ in }B\setminus\overline{\om}.
\end{cases}$$
Note that $\widetilde{v}_{\ell,j}\in H_{0}^1(B)$.
Then, for all $\ell\in\{1,\dots,m-1\}$ and all $j\in\{1,\dots,6\}$ 
we have $\widetilde{v}_{\ell,j}=\xi_{\ell,j}=0$ 
in $\widetilde{\Xi}_{0}\cup\bigcup_{k\in\{1,\ldots,m-1\}\setminus\{\ell\}}\Xi_{k}$ 
and $\widetilde{v}_{\ell,j}=\xi_{\ell,j}=\widehat{r}_{j}$ in $\Xi_{\ell}$. Thus
$$\widetilde{v}:=\sum_{\ell=1}^{m-1}\sum_{j=1}^{6}\alpha_{\ell,j}\widetilde{v}_{\ell,j}\in H^{1,3}_{0}(B)$$
with $\widetilde{v}=0$ in $\widetilde{\Xi}_{0}$ and $\symGrad\widetilde{v}=0$ in $B\setminus\overline{\om}$ 
as well as $\symGrad\widetilde{v}=\symGrad v=0$ in $\om$ by assumption.
Hence, $\symGrad\widetilde{v}=0$ in $B$, showing $\widetilde{v}=0$ in $B$.
In particular, $v=0$ in $\om$, and
$\sum_{j=1}^{6}\alpha_{\ell,j}\widehat{r}_{j}=\widetilde{v}|_{\Xi_{\ell}}=0$ 
for all $\ell\in\{1,\dots,m-1\}$. 
We conclude $\alpha_{\ell,j}=0$ for all $\ell\in\{1,\dots,m-1\}$
and all $j\in\{1,\dots,6\}$, finishing the proof.
\end{proof}

\begin{theorem}
\label{dimDTela}
Let Assumption \ref{ass:segmentprop} be satisfied.
Then $\dim\harmelaDom=6(m-1)$ and a basis of $\harmelaDom$ is given by \eqref{basis:elaD}.
\end{theorem}

\begin{proof}
Use Lemma \ref{dimDTela:lem1} and Lemma \ref{dimDTela:lem2}.
\end{proof}

\section{The Construction of the Neumann Fields}
\label{app:sec:NeuF}

The construction of the Neumann fields is more involved 
than the one for the generalised harmonic Dirichlet fields. 
We start off with some general definitions and remarks all the basis constructions have in common.

Since $\om$ consists of the connected components $\om_{k}$, i.e., 
$\cc(\om)=\{\om_{1},\ldots,\om_{n}\}$, we have by Assumption \ref{ass:curvesandsurfaces} (A3) 
for all $k\in\{1,\dots,n\}$ that $\om_{k}\setminus\bigcup_{j=1}^{p} F_{j}$ is simply connected. 
We define
$$\om_{F}:=\om\setminus\bigcup_{j=1}^{p} F_{j}.$$

For $j\in\{1,\ldots,p\}$,
let $\widehat{F}_{j}\subset\widetilde{F}_{j}$ be two stacked, open, and simply connected neighbourhoods of $F_{j}$, i.e,
$$\overline{F_{j}}\subset\widehat{F}_{j}\subset\overline{\widehat{F}_{j}}\subset\widetilde{F}_{j},$$
let 
$$\Upsilon_{j}:=\widehat{F}_{j}\cap\om,\qquad
\widetilde{\Upsilon}_{j}:=\widetilde{F}_{j}\cap\om,$$
and let $\theta_{j}\in\Ci(\om_{F})$ be a bounded (together with all derivatives) 
function with the following properties:
\begin{itemize}
\item 
$F_{j}\subset\Upsilon_{j}\subset\widetilde{\Upsilon}_{j}$.
\item 
$\Upsilon_{j}$ and $\widetilde{\Upsilon}_{j}$ are (nonempty, open, and) simply connected. 
\item 
$\widetilde{F}_{j}$ are pairwise disjoint.
\item 
$\Upsilon_{j}\setminus F_{j}=\Upsilon_{j,0}\dot\cup\Upsilon_{j,1}$
and $\widetilde{\Upsilon}_{j}\setminus F_{j}=\widetilde{\Upsilon}_{j,0}\dot\cup\widetilde{\Upsilon}_{j,1}$
with $\Upsilon_{j,0}\subset\widetilde{\Upsilon}_{j,0}$ and 
$\Upsilon_{j,1}\subset\widetilde{\Upsilon}_{j,1}$ (which are all nonempty, open, and simply connected). 
\item 
$\overline{\Upsilon}_{j,0}\cap\overline{\Upsilon}_{j,1}=\overline{F}_{j}$.
\item 
$\supp\theta_{j}\subset\overline{\widetilde{\Upsilon}_{j,1}}$.
\item
$\theta_{j}|_{\Upsilon_{j,0}}=0$ 
and $\theta_{j}|_{\Upsilon_{j,1}}=1$.
\end{itemize}
Additionally, for all $l\in\{1,\dots,p\}$ we pick curves 
\begin{itemize}
\item 
$\zeta_{x_{l,0},x_{l,1}}\subset\zeta_{l}$
with fixed starting points $x_{l,0}\in\Upsilon_{l,0}$ 
and fixed endpoints $x_{l,1}\in\Upsilon_{l,1}$.
\end{itemize}

\begin{remark}
\label{rem:constrtheta1}
Roughly speaking, $\widetilde{\Upsilon}_{j}\setminus F_{j}$ 
consists of exactly two open, nonempty, and simply connected components 
$\widetilde{\Upsilon}_{j,0}$ and $\widetilde{\Upsilon}_{j,1}$,
on which subsets $\Upsilon_{j,0}$ (one side) and $\Upsilon_{j,1}$ (other side)
the indicator function $\theta_{j}$ is $0$ and $1$, respectively.
Note that $\Upsilon_{j,0}$ and $\Upsilon_{j,1}$ touch each other
at the whole surface $F_{j}$, i.e.,
$\overline{\Upsilon}_{j,0}\cap\overline{\Upsilon}_{j,1}=\overline{F}_{j}$.
Moreover, $\widetilde{\Upsilon}_{j}$ are pairwise disjoint and 
$\theta_{j}$ is supported in $\overline{\widetilde{\Upsilon}_{j,1}}$.
As a consequence, $\theta_{j}$ \emph{cannot} be continuously extended to $\om$. 
On the other hand, $\grad\theta_{j}=0$ in $\Upsilon_{j}\setminus F_{j}$,
and hence $\grad\theta_{j}$ \emph{can} be continuously extended to $\Upsilon_{j}$
and thus to $\om$. Note that for all $l,j\in\{1,\ldots,p\}$ 
it holds $\theta_{j}(x_{l,0})=0$ and $\theta_{j}(x_{l,1})=\delta_{l,j}$.
\end{remark}

For the construction of bases and to compute the dimensions 
of the Neumann fields it is crucial, 
that these fields are sufficiently regular, e.g., continuous in $\om$.
We even have the following local regularity results.

\begin{lemma}[local regularity of the cohomology groups]
\label{lem:harmreg}
Let $\om\subset\R^{3}$ be open. Then
\begin{align*}
\harmDom,\,\harmNom
&\subset\Citom\cap\Lttom,\\
\harmbihoneDom,\,\harmelaDom,\,\harmbihtwoNom,\,\harmelaNom
&\subset\Cittom\cap\LtttSom,\\
\harmbihtwoDom,\,\harmbihoneNom
&\subset\Cittom\cap\LtttTom.
\end{align*}
\end{lemma}

\begin{proof}
Vector fields in $\harmDom\cup\harmNom$ are harmonic and thus belong to $\Citom$.

Let 
$$S\in\harmbihoneDom\cup\harmbihtwoNom\subset\ker(\CS)\cap\ker(\dDS).$$ 
Then $S$ can be represented locally,
e.g., in any topologically trivial and smooth subdomain $D\subset \om$,
by $S=\Gg u$ with some $u\in H^{2}(D)$, see Lemma \ref{lem:toptrivbih1}.
Therefore, $\dDS\Gg u=0$ in $D$. 
Local regularity for the biharmonic equation shows $u\in\Ci(D)$
and hence $S=\Gg u\in\Citt(D)$, i.e.,
$S\in\Cittom$.

Next, let
$$T\in\harmbihtwoDom\cup\harmbihoneNom\subset\ker(\symCT)\cap\ker(\DT).$$ 
Then, for any topologically trivial and smooth subdomain $D\subset \om$ 
we find $v\in H^{1,3}(D)$ such that $T=\devGrad v$, see Lemma \ref{lem:toptrivbih2}.
Thus $\DT\devGrad v=0$ in $D$. 
Local elliptic regularity shows $v\in\Cit(D)$
and hence $T=\devGrad v\in\Citt(D)$, i.e.,
$T\in\Cittom$.

Finally, let
$$S\in\harmelaDom\cup\harmelaNom\subset\ker(\CCtS)\cap\ker(\DS).$$ 
For $D\subset \om$ smooth, bounded, and topologically trivial, 
we find $v\in H^{1,3}(D)$ representing $S=\symGrad v$, see Lemma \ref{lem:toptrivela}.
Thus, $\DS\symGrad v=0$ in $D$. 
Local elliptic regularity shows $v\in\Cit(D)$
and thus $S=\symGrad v\in\Citt(D)$, i.e.,
$S\in\Cittom$.
\end{proof}

\subsection{\except{toc}{Neumann Vector Fields of the Classical }De Rham Complex}
\label{app:sec:NeuVFdeRham}

Similar to our reasoning for the generalised harmonic Dirichlet fields, 
we start off with the arguably easiest case of the de Rham complex. 
Since we rely on the rephrasing of Picard's ideas in the forthcoming sections, 
we carry out the full construction of the harmonic Neumann fields. 
Note that we heavily use the functions and sets introduced at the beginning 
of Section \ref{app:sec:NeuF}, cf.~Remark \ref{rem:constrtheta1}.
Let $j\in\{1,\ldots,p\}$. By definition $\theta_{j}=0$ 
outside of $\widetilde{\Upsilon}_{j,1}$ and $\theta_{j}$ 
is constant on each connected component $\Upsilon_{j,0}$ and $\Upsilon_{j,1}$
of $\Upsilon_{j}\setminus F_{j}$.
Hence $\grad\theta_{j}=0$ in $\Upsilon_{j}\setminus F_{j}$ 
and -- due to the support condition -- also $\theta_{j}=0$ 
in $\dot\bigcup_{l\in\{1,\ldots,p\}\setminus\{j\}}\Upsilon_{l}\setminus F_{l}$.
Thus, $\grad\theta_{j}$ can be continuously extended by zero to $\Theta_{j}\in\Citom\cap\Lttom$
with $\Theta_{j}=0$ in $\dot\bigcup_{l\in\{1,\ldots,p\}}\Upsilon_{l}$.

\begin{lemma}
\label{dimNFrhm:lem0}
Let Assumption \ref{ass:curvesandsurfaces} be satisfied. Then
$\Theta_{j}\in\ker(\curl,\om)$.
\end{lemma}

\begin{proof}
Let $\Phi\in\Cictom$. 
As $\supp\Theta_{j}\subset\widetilde{\Upsilon}_{j}\setminus\Upsilon_{j}$
we can pick another cut-off function $\varphi\in\Cic(\om_{F})$ with $\varphi|_{\supp\Theta_{j}\cap\supp\Phi}=1$. Then
$$\scp{\Theta_{j}}{\curl\Phi}_{\Lttom}
=\scp{\Theta_{j}}{\curl\Phi}_{\Ltt(\supp\Theta_{j}\cap\supp\Phi)}
=\bscp{\grad\theta_{j}}{\curl(\varphi\Phi)}_{\Ltt(\om_{F})}
=0,$$
as $\varphi\Phi\in\Cict(\om_{F})$.
\end{proof}

%{\color{red}
%Das stimmt jetzt nicht mehr!
%Also doch kein kompakter Tr\"ager in ,,$F_{j}$``''!
%
%\begin{proof}
%Let $\Phi\in\Cictom$. 
%As $\supp\Theta_{j}\subset\widetilde{\Upsilon}_{j}\setminus\Upsilon_{j}$
%we can pick another cut-off function $\varphi\in\Cic(\om_{F})$ with $\varphi|_{\supp\Theta_{j}\cap\supp\Phi}=1$. Then
%$$\scp{\Theta_{j}}{\curl\Phi}_{\Lttom}
%=\scp{\Theta_{j}}{\curl\Phi}_{\Ltt(\supp\Theta_{j}\cap\supp\Phi)}
%=\bscp{\grad\theta_{j}}{\curl(\varphi\Phi)}_{\Ltt(\om_{F})}
%=0,$$
%as $\varphi\Phi\in\Cict(\om_{F})$.
%\end{proof}
%}

Let $l,j\in\{1,\ldots,p\}$. 
We recall from the latter proof and from Remark \ref{rem:constrtheta1} 
that $\supp\Theta_{j}\subset\widetilde{\Upsilon}_{j}\setminus\Upsilon_{j}$
and thus
%\begin{align*}
%\int_{\zeta_{l}}\scp{\Theta_{j}}{\intd\lambda}
%&=\int_{\zeta_{l}\setminus\Upsilon_{j}}\scp{\grad\theta_{j}}{\intd\lambda} 
%=\int_{\zeta_{x_{l,0},x_{l,1}}}\scp{\grad\theta_{j}}{\intd\lambda}\\
%&=\theta_{j}(x_{l,1})
%-\theta_{j}(x_{l,0})=\delta_{l,j}-0=\delta_{l,j},
%\end{align*}
$$\int_{\zeta_{l}}\scp{\Theta_{j}}{\intd\lambda}
=\int_{\zeta_{l}\setminus\Upsilon_{j}}\scp{\grad\theta_{j}}{\intd\lambda} 
=\int_{\zeta_{x_{l,0},x_{l,1}}}\scp{\grad\theta_{j}}{\intd\lambda}
=\theta_{j}(x_{l,1})
-\theta_{j}(x_{l,0})=\delta_{l,j},$$
where we recall $\zeta_{x_{l,0},x_{l,1}}\subset\zeta_{l}$
with chosen starting points $x_{l,0}\in\Upsilon_{l,0}$ 
and respective endpoints $x_{l,1}\in\Upsilon_{l,1}$.
Hence we define functionals $\beta_{l}$ in the way that
\begin{align}
\label{pathint1}
\beta_{l}(\Theta_{j})
:=\int_{\zeta_{l}}\scp{\Theta_{j}}{\intd\lambda}=\delta_{l,j},\qquad
l,j\in\{1,\ldots,p\}.
\end{align}

Let Assumption \ref{ass:segmentprop} be satisfied.
For the de Rham complex, similar to \eqref{deco1}, \eqref{deco2}, and \eqref{app:deco1-derham},
we have the orthogonal decompositions 
\begin{align}
\label{app:deco2-derham}
\begin{aligned}
\Lttom=H_{2}
=\ran A_{2}^{*}\oplus_{H_{2}}\ker A_{2}
&=\ran(\grad,\om)\oplus_{\Lttom}\ker(\rdive,\om),\\
\ker(\curl,\om)=\ker(A_{1}^{*})
=\ran A_{2}^{*}\oplus_{H_{2}}K_{2}
&=\ran(\grad,\om)\oplus_{\Lttom}\harmNom.
\end{aligned}
\end{align}

\begin{remark}
\label{rem:Helmholtz-derhamN}
By definition $\dom(\grad,\om)=H^{1}(\om)$, and the range in \eqref{app:deco2-derham} is closed by the Poincar\'e estimate
$$\exists\,c>0\quad\forall\,\phi\in H^{1}(\om)\cap\R_{\pw}^{\bot_{\Ltom}}\qquad
\norm{\phi}_{\Ltom}\leq c\norm{\grad\phi}_{\Lttom},$$
which is implied by a contradiction argument using Rellich's selection theorem 
as Assumption \ref{ass:segmentprop} holds.
\end{remark}

Similar to the case of harmonic Dirichlet fields, 
we denote in \eqref{app:deco2-derham} the orthogonal projector 
along $\ran(\grad,\om)$ onto $\ker(\rdive,\om)$ by $\pi$.
By Lemma \ref{dimNFrhm:lem0} for all $j\in\{1,\ldots,p\}$ there exists 
some $\psi_{j}\in H^{1}(\om)$ (unique up to $\R_{\pw}$) such that 
$$\harmNom\ni\pi\Theta_{j}=\Theta_{j}-\grad\psi_{j},\qquad
(\Theta_{j}-\grad\psi_{j})\big|_{\om_{F}}
=\grad(\theta_{j}-\psi_{j}).$$
By Lemma \ref{lem:harmreg} we have $\harmNom\subset\Citom$.
Therefore, $\Theta_{j}\in\Citom$ and we obtain $\grad\psi_{j}\in\Citom$. Hence,
$\psi_{j}\in H^{1}(\om)\cap\Ciom$ and the following path integrals are well-defined
and can be computed by \eqref{pathint1}, i.e., for all $l,j\in\{1,\dots,p\}$
\begin{align}
\label{pathint2}
%\begin{aligned}
\beta_{l}(\pi\Theta_{j})
=\int_{\zeta_{l}}\scp{\pi\Theta_{j}}{\intd\lambda}
=\int_{\zeta_{l}}\scp{\Theta_{j}}{\intd\lambda}
-\int_{\zeta_{l}}\scp{\grad\psi_{j}}{\intd\lambda}
=\delta_{l,j}+0=\delta_{l,j}.
%\end{aligned}
\end{align}

We will show that 
\begin{align}
\label{basis:rhmN}
\calB^{\rhm}_{N}:=\{\pi\Theta_{1},\dots,\pi\Theta_{p}\}
\subset\harmNom
\end{align}
defines a basis of $\harmNom$.

Also for the harmonic Neumann fields, 
we provide a possible variational formulation for obtaining $\psi_{j}$ constructed here:

\begin{remark}[Characterisation by PDEs]
Let $j\in\{1,\ldots,p\}$. 
Then $\psi_{j}\in H^{1}(\om)\cap\R_{\pw}^{\bot_{\Ltom}}$ satisfies
$$\forall\,\phi\in H^{1}(\om)\qquad
\scp{\grad\psi_{j}}{\grad\phi}_{\Lttom}
=\scp{\Theta_{j}}{\grad\phi}_{\Lttom},$$
i.e., $\psi_{j}=\Delta_{N}^{-1}\big(\dive\Theta_{j},\nu\cdot\Theta_{j}|_{\ga}\big)$, 
where $\Delta_{N}\subset\dive\grad$ 
is the Laplacian with inhomogeneous Neumann boundary conditions 
restricted to a subset of $H^{1}(\om)\cap\R_{\pw}^{\bot_{\Ltom}}$. 
%Note that due to the support constraints on $\theta_{j}$, we deduce $\dive\Theta_{j}=\rdive\Theta_{j}$.
Therefore, 
$$\pi\Theta_{j}
=\Theta_{j}-\grad\psi_{j}
=\Theta_{j}-\grad\Delta_{N}^{-1}\big(\dive\Theta_{j},\nu\cdot\Theta_{j}|_{\ga}\big).$$
In classical terms, $\psi_{j}$ solves the Neumann Laplace problem
\begin{align}
\label{NeuLap1}
\begin{aligned}
-\Delta\psi_{j}&=-\dive\Theta_{j}
&&\text{in }\om,\\
\nu\cdot\grad\psi_{j}&=\nu\cdot\Theta_{j}
&&\text{on }\ga,\\
\int_{\om_{k}}\psi_{j}&=0
&&\text{for }k\in\{1,\dots,n\},
\end{aligned}
\end{align}
which is uniquely solvable.
\end{remark}

\begin{lemma}
\label{dimNFrhm:lem1}
Let Assumption \ref{ass:segmentprop} and Assumption \ref{ass:curvesandsurfaces} be satisfied. 
Then it holds $\harmNom=\lin\calB^{\rhm}_{N}$.
\end{lemma}

\begin{proof}
Let $H\in\harmNom=\ker(\rdive,\om)\cap\ker(\curl,\om)\subset\Citom$
(see Lemma \ref{lem:harmreg}), and define the numbers
$$\gamma_{j}
:=\gamma_{j}(H)
:=\beta_{j}(H)
=\int_{\zeta_{j}}\scp{H}{\intd\lambda}\in\R,\qquad
j\in\{1,\dots,p\}.$$
We shall show that
$$\harmNom\ni\widehat{H}:=H-\sum_{j=1}^{p}\gamma_{j}\pi\Theta_{j}=0\quad\text{in }\om.$$
The aim is to prove that there exists $u\in H^{1}(\om)$
such that $\grad u=\widehat{H}$, since then
$$\norm{\widehat{H}}_{\Lttom}^{2}
=\scp{\grad u}{\widehat{H}}_{\Lttom}=0.$$
Using \eqref{pathint2}, we obtain
$$\int_{\zeta_{l}}\scp{\widehat{H}}{\intd\lambda}
=\int_{\zeta_{l}}\scp{H}{\intd\lambda}
-\sum_{j=1}^{p}\gamma_{j}\int_{\zeta_{l}}\scp{\pi\Theta_{j}}{\intd\lambda}
=\gamma_{l}
-\sum_{j=1}^{p}\gamma_{j}\beta_{l}(\pi\Theta_{j})=\gamma_{l}-\sum_{j=1}^{p}\gamma_{l}\delta_{l,j}
=0.$$
Note that $\widehat{H}\in\ker(\curl,\om)\cap\Citom$.
Hence, by Assumption \ref{ass:curvesandsurfaces} (A.1)
we have for any closed piecewise $C^1$-curve $\zeta$ in $\om$
\begin{align}
\label{pathint3}
\int_{\zeta}\scp{\widehat{H}}{\intd\lambda}
&=0.
\end{align}
Recall the connected components $\om_{1},\dots,\om_{n}$ of $\om$.
For $1\leq k\leq n$ let $\om_{k}$ and some $x_{0}\in\om_{k}$ be fixed.
By \eqref{pathint3} and the fundamental theorem of calculus 
the function $u:\om\to\R$ given by
$$u(x):=\int_{\zeta(x_{0},x)}\scp{\widehat{H}}{\intd\lambda},\qquad
x\in\om_{k},$$
where $\zeta(x_{0},x)$ is any piecewise $C^1$-curve connecting $x_{0}$ with $x$,
is well defined, i.e., independent of the choice of the respective curve $\zeta(x_{0},x)$, 
and belongs to $\Ci(\om_{k})$ with $\grad u=\widehat{H}\in\Ltt(\om_{k})$.
Thus\footnote{Indeed,
it is sufficient to assume $u\in\Lt_{\mathsf{loc}}(\om_{k})$, see, e.g.,
\cite[Satz 6.6.26, Beweis; Folgerung 6.3.2]{L1997} 
or \cite[Theorem 7.4]{W1987}.} 
$u\in\Lt(\om_{k})$, see, e.g., \cite[Theorem 2.6 (1)]{L1986} or \cite[Theorem 3.2 (2)]{A1965},
and hence $u\in H^{1}(\om_{k})$, showing $u\in H^{1}(\om)$.
\end{proof}

\begin{remark}
\label{dimNFrhm:lem1-rem}
Note that in the latter proof
the existence of $u\in H^{1}(\om_{k})$ with $\grad u=\widehat{H}$ in $\om_{k}$
is well-known, if the connected component $\om_{k}$ of $\om$ is even simply connected.
Indeed, in this case $\ker(\curl,\om_{k})=\ran(\grad,\om_{k})$.
\end{remark}

\begin{lemma}
\label{dimNFrhm:lem2}
Let Assumption \ref{ass:segmentprop} and Assumption \ref{ass:curvesandsurfaces} be satisfied. Then
$\calB^{\rhm}_{N}$ is linearly independent.
\end{lemma}

\begin{proof}
Let $\displaystyle\sum_{j=1}^{p}\gamma_{j}\pi\Theta_{j}=0$ for some $\gamma_{j}\in\R$. 
Then \eqref{pathint2} implies
$$0
=\sum_{j=1}^{p}\gamma_{j}\int_{\zeta_{l}}\scp{\pi\Theta_{j}}{\intd\lambda}
=\sum_{j=1}^{p}\gamma_{j}\beta_{l}(\pi\Theta_{j})
=\sum_{j=1}^{p}\gamma_{j}\delta_{l,j}
=\gamma_{l}$$
for all $l\in\{1,\ldots,p\}$.
\end{proof}

\begin{theorem}
\label{dimNFrhm}
Let Assumptions \ref{ass:segmentprop} and \ref{ass:curvesandsurfaces} be satisfied.
Then $\dim\harmNom=p$ and a basis of $\harmNom$ is given by \eqref{basis:rhmN}.
\end{theorem}

\begin{proof}
Use Lemma \ref{dimNFrhm:lem1} and Lemma \ref{dimNFrhm:lem2}.
\end{proof}

\subsection{\except{toc}{Neumann Tensor Fields of the }First \PZ Complex}
\label{app:sec:NeuTFbih}

The main difference of the constructions to come to the one in the previous section 
is the introduction of $\beta$: a suitable collection of functionals that very easily allows 
for testing of linear independence and for a straightforward application 
of Assumption \ref{ass:curvesandsurfaces} (A1). 
As a preparation for this, we need the next results. 
The first one -- also important for the sections to come -- 
is rather combinatorical and analyses the interplay between vector analysis and matrix calculus; 
the second and third one deal with so-called Poincar\'e maps, 
which form the foundation of the construction of the desired set of functionals. 
Note that for the subsequent sections Lemma \ref{PZformulalem} is of independent interest. 
For this, we introduce for $v\in\R^3$ the skew-symmetric matrix
$$\spn v
:=\begin{pmatrix}0&-v_{3}&v_{2}\\v_{3}&0&-v_{1}\\-v_{2}&v_{1}&0\end{pmatrix}$$
and the corresponding isometric mapping $\spn:\R^{3}\to\R^{3\times3}_{\skw}$.

\begin{lemma}
\label{PZformulalem}
Let $u\in\Cic(\R^{3})$, $v,w\in\Cict(\R^{3})$, and $S\in\Cictt(\R^{3})$. Then:
\begin{itemize}
\item
$(\spn v)\,w=v\times w=-(\spn w)\,v$ 
\quad and  
\quad
$(\spn v)(\spn^{-1}S)=-Sv$, if $\sym S=0$
\item
$\sym\spn v=0$
\quad and \quad
$\dev(u\id)=0$
\item
$\tr\Grad v=\dive v$
\quad and \quad
$2\skw\Grad v=\spn\curl v$
\item
$\Dive(u\id)=\grad u$
\quad and \quad
$\Curl(u\id)=-\spn\grad u$,\\
in particular, \quad $\curl\Dive(u\id)=0$ 
\quad and \quad $\curl\spn^{-1}\Curl(u\id)=0$\\
and \quad $\sym\Curl(u\id)=0$
\item
$\Dive\spn v=-\curl v$
\quad and \quad
$\Dive\skw S=-\curl\spn^{-1}\skw S$,\\
in particular, \quad $\dive\Dive\skw S=0$
\item
$\Curl\spn v=(\dive v)\id-(\Grad v)^{\top}$\\
and \quad $\Curl\skw S=(\dive\spn^{-1}\skw S)\id-(\Grad\spn^{-1}\skw S)^{\top}$
\item
$\dev\Curl\spn v=-(\dev\Grad v)^{\top}$
\item
$-2\Curl\sym\Grad v=2\Curl\skw\Grad v=-(\Grad\curl v)^{\top}$
\item
$2\spn^{-1}\skw\Curl S=\Dive S^{\top}-\grad\tr S=\Dive\big(S-(\tr S)\id\big)^{\top}$,\\
in particular, \quad $\curl\Dive S^{\top}=2\curl\spn^{-1}\skw\Curl S$\\
and \quad $2\skw\Curl S=\spn\Dive S^{\top}$, if $\tr S=0$
\item
$\tr\Curl S=2\dive\spn^{-1}\skw S$,
\quad in particular, \quad $\tr\Curl S=0$, if $\skw S=0$,\\
and \quad $\tr\Curl\sym S=0$ \quad and \quad $\tr\Curl\skw S=\tr\Curl S$
\item
$2(\Grad\spn^{-1}\skw S)^{\top}=(\tr\Curl\skw S)\id-2\Curl\skw S$
\item
$3\Dive(\dev\Grad v)^{\top}=2\grad\dive v$
\item
$2\Curl\sym\Grad v=-2\Curl\skw\Grad v=-\Curl\spn\curl v=(\Grad\curl v)^{\top}$
\item
$2\Dive\sym\Curl S=-2\Dive\skw\Curl S=\curl\Dive S^{\top}$
\item
$\Curl(\Curl\sym S)^{\top}=\sym\Curl(\Curl S)^{\top}$
\item
$\Curl(\Curl\skw S)^{\top}=\skw\Curl(\Curl S)^{\top}$
\end{itemize}
All formulas extend to distributions as well.
\end{lemma}

\begin{proof}
Almost all formulas can be found in \cite[Lemma 3.9]{PZ2020a} and \cite[Lemma A.1]{PZ2020a}.
It is elementary to show that $\skw T=0$ implies $\skw\Curl(\Curl T)^{\top}=0$,
and that $\sym T=0$ implies $\sym\Curl(\Curl T)^{\top}=0$. Note that the needed (straightforward-to-prove) formulas for this are provided in \cite[Appendix B]{PZ2020b}.
Hence $\sym$ commutes with $\Curl\Curl^{\top}$ as
$$\Curl(\Curl\sym T)^{\top}
=\sym\Curl(\Curl\sym T)^{\top}
=\sym\Curl(\Curl T)^{\top},$$
and so does $\skw$.
\end{proof}

In Lemma \ref{devGradHSlemma} below for a tensor field $T$ the operation 
$T\intd\lambda:=\big(\scp{\mathrm{row}_{\ell}T}{\intd\lambda}\big)_{\ell=1,2,3}$
has to be understood row-wise, i.e., the transpose of the $\ell$th row of $T$
is denoted by $\mathrm{row}_{\ell}T$, giving then the vector object $T\intd\lambda$.
More precisely,
$$\big(\int_{\zeta_{x_{0},x}}T\intd\lambda\big)_{\ell}
=\int_{\zeta_{x_{0},x}}\scp{\mathrm{row}_{\ell}T}{\intd\lambda}
=\int_{0}^{1}\Bscp{(\mathrm{row}_{\ell}T)\big(\varphi(t)\big)}{\varphi'(t)}\intd t,\quad
\ell\in\{1,2,3\},$$
with some parametrisation $\varphi\in C^{1,3}_{\pw}\big([0,1]\big)$ 
of a piecewise $C^1$-curve $\zeta_{x_{0},x}$ connecting $x_{0}\in\om$ and $x\in\om$.
Furthermore, we define
$$\int_{\zeta_{x_{0},x}}
(x-y)\bscp{(\Dive T^{\top})(y)}{\intd\lambda_{y}}
:=\int_{0}^{1}\big(x-\varphi(t)\big)\Bscp{(\Dive T^{\top})\big(\varphi(t)\big)}{\varphi'(t)}\intd t.$$

The first statement concerned with describing vector fields and their divergence
by means of curve integrals over tensor fields reads as follows.

\begin{lemma}
\label{devGradHSlemma}
Let $x,x_{0}\in\om$ and let $\zeta_{x_{0},x}\subset\om$ 
be a piecewise $C^{1}$-curve connecting $x_{0}$ and $x$.
\begin{itemize}
\item[\bf(i)]
Let $v\in\Citom$. Then $v$ and its divergence $\dive v$ can be represented by
\begin{align*}
&\qquad v(x)-v(x_{0})-\frac{1}{3}\big(\dive v(x_{0})\big)(x-x_{0})\\
&=\int_{\zeta_{x_{0},x}}\devGrad v\intd\lambda
+\frac{1}{2}\int_{\zeta_{x_{0},x}}
\Big(\int_{\zeta_{x_{0},y}}\bscp{\Dive(\devGrad v)^{\top}}{\intd\lambda}\Big)\id\intd\lambda_{y}
\end{align*}
and
$$\dive v(x)-\dive v(x_{0})
=\frac{3}{2}\int_{\zeta_{x_{0},x}}\bscp{\Dive(\devGrad v)^{\top}}{\intd\lambda}.$$
\item[\bf(ii)]
Let $T\in\Cittom$. Then
$$\int_{\zeta_{x_{0},x}}
\Big(\int_{\zeta_{x_{0},y}}\scp{\Dive T^{\top}}{\intd\lambda}\Big)\id\intd\lambda_{y}
=\int_{\zeta_{x_{0},x}}
(x-y)\bscp{(\Dive T^{\top})(y)}{\intd\lambda_{y}}.$$
\end{itemize}
\end{lemma}

\begin{proof}
For (i), let
$$T:=\devGrad v=\Grad v-\frac{1}{3}(\tr\Grad v)\id=\Grad v-\frac{1}{3}(\dive v)\id$$
and observe $3\Dive T^{\top}=2\grad\dive v$ by Lemma \ref{PZformulalem}. Thus 
\begin{align*}
v_{k}(x)-v_{k}(x_{0})
&=\int_{\zeta_{x_{0},x}}\scp{\grad v_{k}}{\intd\lambda},\qquad
k\in\{1,2,3\},\\
\dive v(x)-\dive v(x_{0})
&=\int_{\zeta_{x_{0},x}}\scp{\grad\dive v}{\intd\lambda}
=\frac{3}{2}\int_{\zeta_{x_{0},x}}\scp{\Dive T^{\top}}{\intd\lambda}.
\end{align*}
Therefore,
\begin{align*}
&\qquad v(x)-v(x_{0})
=\int_{\zeta_{x_{0},x}}\Grad v\intd\lambda\\
&=\int_{\zeta_{x_{0},x}}\devGrad v\intd\lambda
+\frac{1}{3}\int_{\zeta_{x_{0},x}}\dive v\id\intd\lambda
=\int_{\zeta_{x_{0},x}}T\intd\lambda
+\frac{1}{3}\int_{\zeta_{x_{0},x}}\dive v(y)\id\intd\lambda_{y}\\
&=\int_{\zeta_{x_{0},x}}T\intd\lambda
+\frac{1}{3}\dive v(x_{0})\int_{\zeta_{x_{0},x}}\id\intd\lambda_{y}+\frac{1}{2}\int_{\zeta_{x_{0},x}}
\Big(\int_{\zeta_{x_{0},y}}\scp{\Dive T^{\top}}{\intd\lambda}\Big)\id\intd\lambda_{y}.
\end{align*}
Moreover,\footnote{Alternatively, note 
$\displaystyle\int_{\zeta_{x_{0},x}}\id\intd\lambda_{y}
=\int_{0}^{1}\id\varphi'(s)\intd s
=\int_{0}^{1}\varphi'(s)\intd s=x-x_{0}$
with the parametrisation $\varphi$ of $\zeta_{x_{0},x}$ from above.} 
$\displaystyle\int_{\zeta_{x_{0},x}}\id\intd\lambda_{y}
=\int_{\zeta_{x_{0},x}}\Grad y\intd\lambda_{y}
=x-x_{0}$.

For (ii) we compute
\begin{align*}
\int_{\zeta_{x_{0},x}}
\Big(\int_{\zeta_{x_{0},y}}\scp{\Dive T^{\top}}{\intd\lambda}\Big)\id\intd\lambda_{y}
&=\int_{0}^{1}
\Big(\int_{\zeta_{x_{0},\varphi(s)}}\scp{\Dive T^{\top}}{\intd\lambda}\Big)
\id\varphi'(s)\intd s\\
&=\int_{0}^{1}
\Big(\int_{0}^{s}\Bscp{(\Dive T^{\top})\big(\varphi(t)\big)}{\varphi'(t)}\intd t\Big)\varphi'(s)\intd s\\
&=\int_{0}^{1}
\int_{t}^{1}\varphi'(s)\intd s
\Bscp{(\Dive T^{\top})\big(\varphi(t)\big)}{\varphi'(t)}\intd t\\
&=\int_{0}^{1}
(x-\varphi(t))
\Bscp{(\Dive T^{\top})\big(\varphi(t)\big)}{\varphi'(t)}\intd t\\
&=\int_{\zeta_{x_{0},x}}
(x-y)\bscp{(\Dive T^{\top})(y)}{\intd\lambda_{y}}
\end{align*}
with $\varphi$ parametrising $\zeta_{x_{0},x}$ as above.
\end{proof}

\begin{proposition}
\label{prop:HSdevGrad} 
Let $x_{0}\in\om_{0}\in\cc(\om)$ and let $S,T\in\Citt(\om_{0})$.
\begin{itemize}
\item[\bf(a)]
The following conditions are equivalent:
\begin{itemize}
\item[\bf(i)]
For all $\zeta\subset\om_{0}$ closed, piecewise $C^1$-curves
$$\int_{\zeta}\bscp{\Dive T^{\top}}{\intd\lambda}=0.$$
\item[\bf(ii)]
For all $\zeta_{x_{0},x},\widetilde{\zeta}_{x_{0},x}\subset\om_{0}$ 
piecewise $C^1$-curves connecting $x_{0}$ with $x$
$$\int_{\zeta_{x_{0},x}}\bscp{\Dive T^{\top}}{\intd\lambda}
=\int_{\widetilde{\zeta}_{x_{0},x}}\bscp{\Dive T^{\top}}{\intd\lambda}.$$
\item[\bf(iii)]
There exists $u\in\Ci(\om_{0})$ such that
$\grad u=\Dive T^{\top}$.
\end{itemize} 
In the case one of the above conditions is true the function
\begin{equation}
\label{eq:chu}
x\mapsto u(x)
=\int_{\zeta_{x_{0},x}}\bscp{\Dive T^{\top}}{\intd\lambda}
\end{equation}
for some $\zeta_{x_{0},x}\subset \om_{0}$ piecewise $C^1$-curve connecting $x_{0}$ with $x$ 
is a (well-defined) possible choice for $u$ in (iii).
\item[\bf(b)]
The following conditions are equivalent:
\begin{itemize}
\item[\bf(i)]
For all $\zeta\subset\om_{0}$ closed, piecewise $C^1$ curves
$$\int_{\zeta}S\intd\lambda=0.$$
\item[\bf(ii)]
For all $\zeta_{x_{0},x},\widetilde{\zeta}_{x_{0},x}\subset \om_{0}$ 
piecewise $C^1$ curves connecting $x_{0}$ with $x$
$$\int_{\zeta_{x_{0},x}}S\intd\lambda
=\int_{\tilde{\zeta}_{x_{0},x}}S\intd\lambda.$$
\item[\bf(iii)]
There exists $v\in\Cit(\om_{0})$ such that $\Grad v=S$.
\end{itemize}
In the case one of the above conditions is true the vector field
\begin{equation}
\label{eq:chv}
x\mapsto v(x)=\int_{\zeta_{x_{0},x}}S\intd\lambda
\end{equation}
for some $\zeta_{x_{0},x}\subset\om_{0}$ piecewise $C^1$-curve connecting $x_{0}$ with $x$ 
is a (well-defined) possible choice for $v$ in (iii). 
\item[\bf(c)]
Let $S:=T+\frac{1}{2}\,u\id$ with $u\in\Ci(\om_{0})$ and $\grad u=\Dive T^{\top}$ as in (a), (iii). 
Moreover, let $v\in\Cit(\om_{0})$ such that $\Grad v=S$ as in (b), (iii). Then
\begin{itemize}
\item[\bf(i)]
$\tr T=0$,
\item[\bf(ii)]
$\devGrad v =T$
\end{itemize}
are equivalent. In either case, we have 
\begin{equation}
\label{eq:remform}
\symCT T=0\quad\text{and}\quad
\grad u=\frac{2}{3}\grad\dive v.
\end{equation}
\end{itemize}
\end{proposition}

\begin{proof}
(a): 
The conditions (i) and (ii) are clearly equivalent. 
Assuming (ii), we obtain that the choice of $u$ in \eqref{eq:chu} is well-defined. 
By the fundamental theorem of calculus it follows that this $u$ 
satisfies the equation in (iii) and, consequently, $u\in\Ci(\om_{0})$. 
If, on the other hand, (iii) is true, then again using the fundamental theorem of calculus, 
we obtain (ii).

(b): 
The proof of the equivalence follows almost exactly the same way as for (a).

(c): 
We compute $\devGrad v=\dev S=\dev(T+\frac{1}{2}u\id)=\dev T$.
Hence (i) and (ii) are equivalent.
%Assume that $T$ is trace-free. Then
%\[
%  \devGrad v= \dev S = \dev(T+\frac12 u\id)=\dev T = T.
%\]
%On the other hand, $\devGrad v=T$ leads to $\tr T=\tr \devGrad v=0$.
Finally, if (i) or (ii) is true, then by the complex property
$$\symCurl T=\symCurl\devGrad v=0,$$
%$$\symCurl T=\symCurl\devGrad v=0,\quad
%\curl\Dive T^{\top}=2\Dive\symCurl T=0,$$
and we conclude
$\grad u=\Dive T^{\top}=\Dive(\devGrad v)^{\top}=\frac{2}{3}\grad\dive v$
by Lemma \ref{PZformulalem}.
\end{proof}

\begin{remark}
\label{curlformula1-bih1}
Related to Proposition \ref{prop:HSdevGrad} we note with Lemma \ref{PZformulalem} the following:
\begin{itemize}
\item[\bf(i)]
For $T\in\Cittom$ we have
$$\curl\Dive T^{\top}
=2\Dive\symCurl T.$$
\item[\bf(ii)]
For $T\in\Citt_{\Tb}(\om)$
and $S:=T+\frac{1}{2}\,u\id$ 
with $\grad u=\Dive T^{\top}$ it holds
$$\Curl S
=\Curl T-\frac{1}{2}\spn\grad u
=\Curl T-\skw\Curl T
=\symCurl T.$$
\item[\bf(iii)]
If $\om_{0}$ is simply connected, 
Proposition \ref{prop:HSdevGrad} (a), (iii)
and (b), (iii) are equivalent to $\curl\Dive T^{\top}=0$ and $\Curl S=0$, respectively.
\end{itemize}
\end{remark}

Arguing for each connected component separately 
(and using formulas \eqref{eq:chu} and \eqref{eq:chv} on every connected component), 
we obtain the following more condensed version of Proposition \ref{prop:HSdevGrad}.

\begin{corollary}
\label{cor:HSdevGrad} 
Let $S,T\in\Cittom$.
\begin{itemize}
\item[\bf(a)]
The following conditions are equivalent:
\begin{itemize}
\item[\bf(i)]
For all $\zeta\subset\om$ closed, piecewise $C^1$-curves 
$\int_{\zeta}\bscp{\Dive T^{\top}}{\intd\lambda}=0$.
\item[\bf(ii)]
There exists $u\in\Ciom$ such that $\grad u=\Dive T^{\top}$.
\end{itemize}
\item[\bf(b)]
The following conditions are equivalent:
\begin{itemize}
\item[\bf(i)]
For all $\zeta\subset\om$ closed, piecewise $C^1$-curves $\int_{\zeta}S\intd\lambda=0$.
\item[\bf(ii)]
There exists $v\in\Citom$ such that $\Grad v=S$. 
\end{itemize}
\item[\bf(c)]
Let $S=T+\frac{1}{2}\,u\id$ with $u\in\Ciom$ and $\grad u=\Dive T^{\top}$ as in (a), (ii).
Moreover, let $v\in\Citom$ with $\Grad v=S$ as in (b), (ii). Then
$\tr T=0$ in $\om$ if and only if $\devGrad v=T$ in $\om$.
\end{itemize}
\end{corollary}

%%%\begin{remark}
%%%\label{curlformula1-bih1}
%%%In Lemma \ref{devGradHSlemma} (iii)
%%%for $T\in\Ci_{\Tb}(\om,\R^{3\times3})$
%%%and $S:=T+\frac{1}{2}\,u\id$ 
%%%with $\grad u=\Dive T^{\top}$ the formulas 
%%%$$\curl\Dive T^{\top}=2\Dive\symCT T,\qquad
%%%\Curl S=\symCT T$$
%%%are crucial. These will be derived in the upcoming proof 
%%%and follow by Lemma \ref{PZformulalem}.
%%%\end{remark}

The construction of the harmonic Neumann tensor fields for the first \pz complex 
form a nontrivial adaptation of the rationale developed in the previous section
for the de Rham complex. 
We shortly recall that for $j\in\{1,\ldots,p\}$, by construction, 
$\theta_{j}=0$ outside of $\widetilde{\Upsilon}_{j,1}$ and that $\theta_{j}$ 
is constant on each connected component $\Upsilon_{j,0}$ and $\Upsilon_{j,1}$
of $\Upsilon_{j}\setminus F_{j}$. 
Let $\widehat{r}_{k}$ be the Raviart--Thomas fields from Section \ref{app:sec:DirTFbih2}
given by $\widehat{r}_{0}(x):=x$ and $\widehat{r}_{k}(x):=e^{k}$ for $k\in\{1,2,3\}$.
We define the vector fields $\theta_{j,k}:=\theta_{j}\widehat{r}_{k}$
and note $\devGrad\theta_{j,k}=0$
in $\dot\bigcup_{l\in\{1,\ldots,p\}}\Upsilon_{l}\setminus F_{l}$ 
for all $j\in\{1,\ldots,p\}$ and all $k\in\{1,2,3\}$.
Thus $\devGrad\theta_{j,k}$ can be continuously extended by zero 
to $\Theta_{j,k}\in\Cittom\cap\LtttTom$
with $\Theta_{j,k}=0$ in all the neighbourhoods 
$\Upsilon_{l}$ of all the surfaces $F_{l}$, $l\in\{1,\ldots,p\}$.

\begin{lemma}
\label{dimNTbih1:lem0}
Let Assumption \ref{ass:curvesandsurfaces} be satisfied. Then
$\Theta_{j,k}\in\ker(\symCT,\om)$.
\end{lemma}

\begin{proof}
Let $\Phi\in\CicttSom$. As $\supp\Theta_{j,k}\subset\widetilde{\Upsilon}_{j}\setminus\Upsilon_{j}$
we can pick another cut-off function $\varphi\in\Cic(\om_{F})$ with $\varphi|_{\supp\Theta_{j,k}\cap\supp\Phi}=1$. Then
\begin{align*}
\scp{\Theta_{j,k}}{\CS\Phi}_{\LtttTom}
&=\scp{\Theta_{j,k}}{\CS\Phi}_{\LtttT(\supp\Theta_{j,k}\cap\supp\Phi)}\\
&=\bscp{\devGrad\theta_{j,k}}{\CS(\varphi\Phi)}_{\LtttT(\om_{F})}\\
&=\bscp{\Grad\theta_{j,k}}{\dev\CS(\varphi\Phi)}_{\LtttT(\om_{F})}\\
&=\bscp{\Grad\theta_{j,k}}{\Curl(\varphi\Phi)}_{\Lttt(\om_{F})}
=0
\end{align*}
as $\varphi\Phi\in\Cictt(\om_{F})$, 
where in the second to last equality sign, 
we used that the $\Curl$ applied to a symmetric tensor fields is trace-free, i.e., 
deviatoric, see Lemma \ref{PZformulalem}.
\end{proof}

Next, we note that for $l,j\in\{1,\dots,p\}$ and $k\in\{0,\dots,3\}$ 
and for the curves $\zeta_{x_{l,0},x_{l,1}}\subset\zeta_{l}$
with the chosen starting points $x_{l,0}\in\Upsilon_{l,0}$ 
and respective endpoints $x_{l,1}\in\Upsilon_{l,1}$
we can compute by Lemma \ref{devGradHSlemma}
\begin{align*}
\R\ni\beta_{l,0}(\Theta_{j,k})
&:=\frac{1}{2}\int_{\zeta_{l}}
\scp{\Dive\Theta_{j,k}^{\top}}{\intd\lambda}
=\frac{1}{2}\int_{\zeta_{x_{l,0},x_{l,1}}}\bscp{\Dive(\devGrad\theta_{j,k})^{\top}}{\intd\lambda}\\
&\phantom{:}=\frac{1}{3}\dive\theta_{j,k}(x_{l,1})-\frac{1}{3}\dive\theta_{j,k}(x_{l,0})
=\frac{1}{3}\dive\theta_{j,k}(x_{l,1})\\
&\phantom{:}=\frac{1}{3}\delta_{l,j}\dive\widehat{r}_{k}(x_{l,1})
=\delta_{l,j}\begin{cases}
1,&\text{if }k=0,\\
0,&\text{if }k\in\{1,2,3\},
\end{cases}
\intertext{and}
\R^{3}\ni b_{l}(\Theta_{j,k})
&:=\int_{\zeta_{l}}\Theta_{j,k}\intd\lambda
+\frac{1}{2}\int_{\zeta_{l}}
(x_{l,1}-y)\bscp{(\Dive\Theta_{j,k}^{\top})(y)}{\intd\lambda_{y}}\\
&\phantom{:}=\int_{\zeta_{x_{l,0},x_{l,1}}}\devGrad\theta_{j,k}\intd\lambda\\
&\qquad\qquad+\frac{1}{2}\int_{\zeta_{x_{l,0},x_{l,1}}}
(x_{l,1}-y)\Bscp{\big(\Dive(\devGrad\theta_{j,k})^{\top}\big)(y)}{\intd\lambda_{y}}\\
&\phantom{:}=\int_{\zeta_{x_{l,0},x_{l,1}}}\!\!\!\bigg(\devGrad\theta_{j,k}(y)\\
&\qquad\qquad+\frac{1}{2}\Big(\int_{\zeta_{x_{l,0},y}}
\!\!\bscp{\Dive(\devGrad\theta_{j,k})^{\top}}{\intd\lambda}\Big)\id\bigg)\intd\lambda_{y}\\
&\phantom{:}=\theta_{j,k}(x_{l,1})
-\theta_{j,k}(x_{l,0})-\frac{1}{3}\dive\theta_{j,k}(x_{l,0})(x_{l,1}-x_{l,0})
=\theta_{j,k}(x_{l,1})\\
&\phantom{:}=\delta_{l,j}\widehat{r}_{k}(x_{l,1})
=\delta_{l,j}\begin{cases}
x_{l,1},&\text{if }k=0,\\
e^{k},&\text{if }k\in\{1,2,3\}.
\end{cases}
\end{align*}
Thus, for $l\in\{1,\dots,p\}$ and $\ell\in\{0,\dots,3\}$
we have functionals $\beta_{l,\ell}$, given by
\begin{align*}
\beta_{l,0}(\Theta_{j,k})
&=\delta_{l,j}\delta_{0,k}
\intertext{for $l,j\in\{1,\dots,p\}$ and $k\in\{0,\dots,3\}$, as well as}
\beta_{l,\ell}(\Theta_{j,k})
&:=\bscp{b_{l}(\Theta_{j,k})}{e^{\ell}}
=\delta_{l,j}\begin{cases}
\scp{x_{l,1}}{e^{\ell}}=(x_{l,1})_{\ell},&\text{if }k=0,\\
\scp{e^{k}}{e^{\ell}}=\delta_{\ell,k},&\text{if }k\in\{1,2,3\},
\end{cases}
\end{align*}
for $l,j\in\{1,\dots,p\}$ and $\ell\in\{1,2,3\}$ and $k\in\{0,\dots,3\}$.
Therefore, we have 
\begin{align}
\label{pathint1-bih1}
\beta_{l,\ell}(\Theta_{j,k})
&=\delta_{l,j}\delta_{\ell,k}
+(1-\delta_{\ell,0})\delta_{0,k}\delta_{l,j}(x_{l,1})_{\ell},\quad
l,j\in\{1,\dots,p\},\quad k,\ell\in\{0,\ldots,3\}.
\end{align}

Let Assumption \ref{ass:stronglip} be satisfied.
For the first \pz complex, similar to \eqref{deco1}, \eqref{deco2}, 
we have the orthogonal decompositions 
\begin{align}
\label{app:deco2-bih1}
\begin{aligned}
\LtttTom&=\ran(\devGrad,\om)\oplus_{\LtttTom}\ker(\rDT,\om),\\
\ker(\symCT,\om)&=\ran(\devGrad,\om)\oplus_{\LtttTom}\harmbihoneNom.
\end{aligned}
\end{align}

\begin{remark}
\label{rem:Helmholtz-bih1N}
By Assumption \ref{ass:stronglip}, \cite[Lemma 3.2]{PZ2020a} yields $\dom(\devGrad,\om)=H^{1,3}(\om)$. 
As a consequence using Rellich's selection theorem, 
the range in \eqref{app:deco2-bih1} is closed and the Poincar\'e type estimate
$$\exists\,c>0\quad\forall\,\phi\in H^{1,3}(\om)\cap\RT_{\pw}^{\bot_{\Lttom}}\qquad
\norm{\phi}_{\Lttom}\leq c\norm{\devGrad\phi}_{\Ltttom},$$
holds, see also \cite[Lemma 3.2]{PZ2020a}.
\end{remark}

Let $\pi:\LtttTom\to\ker(\rDT,\om)$ denote the orthogonal projector 
onto $\ker(\rDT,\om)$ along $\ran(\devGrad,\om)$, 
see \eqref{app:deco2-bih1}.
We have $\pi\big(\ker(\symCT,\om)\big)=\harmbihoneNom$. 
By Lemma \ref{dimNTbih1:lem0} there exists 
some $\psi_{j,k}\in H^{1,3}(\om)$ such that 
$$\harmbihoneNom\ni\pi\Theta_{j,k}=\Theta_{j,k}-\devGrad\psi_{j,k},\quad
(\Theta_{j,k}-\devGrad\psi_{j,k})\big|_{\om_{F}}
=\devGrad(\theta_{j,k}-\psi_{j,k}).$$
As $\harmbihoneNom\subset\Cittom$, cf. \eqref{app:harmreg},
we conclude by $\pi\Theta_{j,k},\Theta_{j,k}\in\Cittom$
that also $\devGrad\psi_{j,k}\in\Cittom$ and hence
$\psi_{j,k}\in\Citom$.
Thus all path integrals over the closed curves $\zeta_{l}$ are well-defined. 
Furthermore, we observe by Lemma \ref{devGradHSlemma}
\begin{align*}
\beta_{l,0}(\devGrad\psi_{j,k})
&=\frac{1}{2}\int_{\zeta_{l}}
\bscp{\Dive(\devGrad\psi_{j,k})^{\top}}{\intd\lambda}\\
&=\frac{1}{3}\dive\psi_{j,k}(x_{l,1})-\frac{1}{3}\dive\psi_{j,k}(x_{l,1})
=0
\end{align*}
and
\begin{align*}
&\qquad b_{l}(\devGrad\psi_{j,k})\\
&=\int_{\zeta_{l}}\devGrad\psi_{j,k}\intd\lambda
+\frac{1}{2}\int_{\zeta_{l}}
(x_{l,1}-y)\Bscp{\big(\Dive(\devGrad\psi_{j,k})^{\top}\big)(y)}{\intd\lambda_{y}}\\
&=\int_{\zeta_{x_{l,1},x_{l,1}}}\bigg(\devGrad\psi_{j,k}(y)
+\frac{1}{2}\Big(\int_{\zeta_{x_{l,1},y}}
\bscp{\Dive(\devGrad\psi_{j,k})^{\top}}{\intd\lambda}\Big)\id\bigg)\intd\lambda_{y}\\
&=\psi_{j,k}(x_{l,1})
-\psi_{j,k}(x_{l,1})-\frac{1}{3}\dive\psi_{j,k}(x_{l,1})(x_{l,1}-x_{l,1})
=0.
\end{align*}
Therefore, $\beta_{l,\ell}(\devGrad\psi_{j,k})=0$ and by \eqref{pathint1-bih1} we have 
\begin{align}
\label{pathint2-bih1}
\beta_{l,\ell}(\pi\Theta_{j,k})
=\beta_{l,\ell}(\Theta_{j,k})
-\beta_{l,\ell}(\devGrad\psi_{j,k})
=\delta_{l,j}\delta_{\ell,k}
+(1-\delta_{\ell,0})\delta_{0,k}\delta_{l,j}(x_{l,1})_{\ell}
\end{align}
for all $l,j\in\{1,\dots,p\}$ and all $\ell,k\in\{0,1,2,3\}$. 
We shall show that 
\begin{align}
\label{basis:bih1N}
\calB^{\bihone}_{N}
:=\big\{\pi\Theta_{j,k}:j\in\{1,\dots,p\},\,k\in\{0,\dots,3\}\big\}
\subset\harmbihoneNom
\end{align}
defines a basis of $\harmbihoneNom$.

\begin{remark}[Characterisation by PDEs]
Note that $\psi_{j,k}\in H^{1,3}(\om)\cap\RT_{\pw}^{\bot_{\Lttom}}$ 
can be found by the variational formulation
$$\forall\,\phi\in H^{1,3}(\om)\qquad
\scp{\devGrad\psi_{j,k}}{\devGrad\phi}_{\Ltttom}
=\scp{\Theta_{j,k}}{\devGrad\phi}_{\Ltttom},$$
i.e., $\psi_{j,k}=\Delta_{\Tb,N}^{-1}\big(\DT\Theta_{j,k},\Theta_{j,k}\nu|_{\ga}\big)$,
where $\Delta_{\Tb,N}\subset\DT\devGrad$
denotes the `deviatoric' Laplacian with inhomogeneous Neumann boundary conditions 
restricted to a subset of $H^{1,3}(\om)\cap\RT_{\pw}^{\bot_{\Lttom}}$. 
Therefore, 
$$\pi\Theta_{j,k}
=\Theta_{j,k}-\devGrad\psi_{j,k}
=\Theta_{j,k}-\devGrad\Delta_{\Tb,N}^{-1}\big(\DT\Theta_{j,k},\Theta_{j,k}\nu|_{\ga}\big).$$
In classical terms, $\psi_{j,k}$
solves the Neumann elasticity type problem
\begin{align}
\label{NeuLap1-bih1}
\begin{aligned}
-\Delta_{\Tb}\psi_{j,k}&=-\DT\Theta_{j,k}
&&\text{in }\om,\\
(\Grad\psi_{j,k})\nu&=\Theta_{j,k}\nu
&&\text{on }\ga,\\
\int_{\om_{l}}(\psi_{j,k})_{\ell}&=0
&&\text{for }l\in\{1,\dots,n\},\quad\ell\in\{1,2,3\},\\
\int_{\om_{l}}x\cdot\psi_{j,k}(x)\intd\lambda_{x}&=0
&&\text{for }l\in\{1,\dots,n\},
\end{aligned}
\end{align}
which is uniquely solvable.
\end{remark}

\begin{lemma}
\label{dimNTbih1:lem1}
Let Assumption \ref{ass:stronglip} as well as Assumption \ref{ass:curvesandsurfaces} be satisfied. 
Then we have $\harmbihoneNom=\lin\calB^{\bihone}_{N}$.
\end{lemma}

\begin{proof}
Let $H\in\harmbihoneNom=\ker(\rDT,\om)\cap\ker(\symCT,\om)\subset\Citt_{\Tb}(\om)$, 
see Lemma \ref{lem:harmreg}.
With the above introduced functionals $\beta_{l,0}$ and $b_{l}$, $l\in\{1,\dots,p\}$, we recall
\begin{align*}
\R\ni\beta_{l,0}(H)
&=\frac{1}{2}\int_{\zeta_{l}}\scp{\Dive H^{\top}}{\intd\lambda},\\
\R^{3}\ni b_{l}(H)
&=\int_{\zeta_{l}}H\intd\lambda
+\frac{1}{2}\int_{\zeta_{l}}
(x_{l,1}-y)\bscp{(\Dive H^{\top})(y)}{\intd\lambda_{y}},
\end{align*}
and define for $l\in\{1,\dots,p\}$ and $\ell\in\{1,2,3\}$ the numbers
\begin{align}
\label{eq:gl0}
\begin{aligned}
\gamma_{l,0}
:=\gamma_{l,0}(H)
&:=\beta_{l,0}(H),\\
\gamma_{l,\ell}
:=\gamma_{l,\ell}(H)
&:=\bscp{b_{l}(H)-\beta_{l,0}(H)x_{l,1}}{e^{\ell}}
=\beta_{l,\ell}(H)-\beta_{l,0}(H)(x_{l,1})_{\ell}.
\end{aligned}
\end{align}
We shall show that
$$\harmbihoneNom\ni\widehat{H}:=H-\sum_{j=1}^{p}\sum_{k=0}^{3}\gamma_{j,k}\pi\Theta_{j,k}=0\quad\text{in }\om.$$
Similar to the proof of Lemma \ref{dimNFrhm:lem1},
the aim is to prove that there exists $v\in H^{1,3}(\om)$
such that $\devGrad v=\widehat{H}$, since then
$$\norm{\widehat{H}}_{\LtttTom}^{2}
=\scp{\devGrad v}{\widehat{H}}_{\LtttTom}=0.$$
For finding $v$, we will apply Corollary \ref{cor:HSdevGrad}
and Remark \ref{curlformula1-bih1} to $T=\widehat{H}$. 
By \eqref{pathint2-bih1} we observe for all $l\in\{1,\ldots,p\}$
\begin{align*}
\frac{1}{2}\int_{\zeta_{l}}\scp{\Dive\widehat{H}^{\top}}{\intd\lambda}
=\beta_{l,0}(\widehat{H}) 
&=\beta_{l,0}(H)
-\sum_{j=1}^{p}\sum_{k=0}^{3}\gamma_{j,k}
\beta_{l,0}(\pi\Theta_{j,k})\\
&=\gamma_{l,0}
-\sum_{j=1}^{p}\sum_{k=0}^{3}\gamma_{j,k}
\delta_{l,j}\delta_{0,k}
=0.
\end{align*}
Note that it holds $\Dive\widehat{H}^{\top}\in\ker(\curl,\om)\cap\Citom$
by Remark \ref{curlformula1-bih1} (i) 
as $\widehat{H}$ belongs to $\ker(\symCT,\om)\cap\Cittom$.
Thus, by Assumption \ref{ass:curvesandsurfaces} (A.1)
for any closed piecewise $C^1$-curve $\zeta$ in $\om$
\begin{align}
\label{pathint3-bih1}
\int_{\zeta}\scp{\Dive\widehat{H}^{\top}}{\intd\lambda}=0.
\end{align}
Let $u\in\Ciom$ be as in Corollary \ref{cor:HSdevGrad} (a), (ii), i.e., 
$\grad u=\Dive\widehat{H}^{\top}$,
and define $S:\om\to\R^{3\times3}$ by
$$S:=\widehat{H}+\frac{1}{2}\,u\id.$$
Our next aim is to show condition (b), (ii) of Corollary \ref{cor:HSdevGrad}. 
For this, let $l\in\{1,\ldots,p\}$. 
Note that $\zeta_{x_{l,0},x_{l,1}}\subset\zeta_{l}\subset\om_{0}$ for some $\om_{0}\in\cc(\om)$. 
Then we have with $c:=u(x_{l,1})\in\R$ for all $x\in\zeta_{l}$
$$u(x)
=u(x)-u(x_{l,1})+c 
=\int_{\zeta_{x_{l,1},x}}\scp{\grad u}{\intd\lambda}+c 
=\int_{\zeta_{x_{l,1},x}}\scp{\Dive\widehat{H}^{\top}}{\intd\lambda}+c,$$
where $\zeta_{x_{l,1},x}$ denotes the path from $x_{l,1}$ to $x$ along $\zeta_{l}$. Moreover,
$$\int_{\zeta_{l}}(c\id)\intd\lambda
=c\int_{\zeta_{l}}\Grad x\intd\lambda_{x}
=0.$$
Next, we consider the closed curve $\zeta_{l}$ 
as the closed curve $\zeta_{x_{l,1},x_{l,1}}$ with circulation $1$ along $\zeta_{l}$. 
Then, using Lemma \ref{devGradHSlemma} and the definition of $b_{l}$, we compute
\begin{align*}
\int_{\zeta_{l}}S\intd\lambda
&=\int_{\zeta_{l}}\widehat{H}\intd\lambda
+\frac{1}{2}\int_{\zeta_{l}}(u\id)\intd\lambda\\
&=\int_{\zeta_{l}}\widehat{H}\intd\lambda
+\frac{1}{2}\int_{\zeta_{x_{l,1},x_{l,1}}}
\Big(\int_{\zeta_{x_{l,1},y}}\scp{\Dive\widehat{H}^{\top}}{\intd\lambda}\Big)\id\intd\lambda_{y}\\
&=\int_{\zeta_{l}}\widehat{H}\intd\lambda
+\frac{1}{2}\int_{\zeta_{l}}
(x_{l,1}-y)\bscp{(\Dive\widehat{H}^{\top})(y)}{\intd\lambda}\intd\lambda_{y}
=b_{l}(\widehat{H}).
\end{align*}
Hence, for $\ell\in\{1,2,3\}$ we get by \eqref{pathint2-bih1} recalling \eqref{eq:gl0}
\begin{align*}
\Big(\int_{\zeta_{l}}S\intd\lambda\Big)_{\ell}
&=\bscp{\int_{\zeta_{l}}S\intd\lambda}{e^{\ell}}
=\scp{b_{l}(\widehat{H})}{e^{\ell}}
=\beta_{l,\ell}(\widehat{H})\\
&=\beta_{l,\ell}(H)
-\sum_{j=1}^{p}\sum_{k=0}^{3}\gamma_{j,k}
\beta_{l,\ell}(\pi\Theta_{j,k})\\
&=\beta_{l,\ell}(H)
-\sum_{j=1}^{p}\sum_{k=0}^{3}\gamma_{j,k}
\big(\delta_{l,j}\delta_{\ell,k}
+(1-\delta_{\ell,0})\delta_{0,k}\delta_{l,j}(x_{l,1})_{\ell}\big)\\
&=\beta_{l,\ell}(H)
-\gamma_{l,0}(x_{l,1})_{\ell}
-\gamma_{l,\ell}=\beta_{l,\ell}(H)
-\beta_{l,0}(H)(x_{l,1})_{\ell}
-\gamma_{l,\ell}=0.
\end{align*}
Therefore, $\int_{\zeta_{l}}S\intd\lambda=0$ for all $l\in\{1,\ldots,p\}$.
Note that $S\in\ker(\Curl,\om)\cap\Cittom$
by Remark \ref{curlformula1-bih1} (ii) 
as $\widehat{H}\in\ker(\symCT,\om)\cap\Citt_{\Tb}(\om)$.
Thus, by Assumption \ref{ass:curvesandsurfaces} (A.1)
for any closed piecewise $C^1$-curve $\zeta$ in $\om$
\begin{align}
\label{pathint4-bih1}
\int_{\zeta}S\intd\lambda=0.
\end{align}
Hence, Corollary \ref{cor:HSdevGrad} (b) and (c) (note $\tr\widehat{H}=0$) imply the existence 
of a smooth vector field $v:\om\to\R^{3}$ such that $\devGrad v=\widehat{H}$.
Finally, similar to the end of the proof of Lemma \ref{dimNFrhm:lem1},
elliptic regularity and, e.g., \cite[Theorem 2.6 (1)]{L1986} or \cite[Theorem 3.2 (2)]{A1965},
show that $v\in\Cit(\om_{0})$ and $\devGrad v\in\LtttT(\om_{0})$ imply
$v\in H^{1,3}(\om_{0})$ for all $\om_{0}\in\cc(\om)$ and thus $v\in H^{1,3}(\om)$, completing the proof. 
(Let us note that $v\in H^{1,3}(\om)$
implies also $S\in\Ltttom$ and hence $u\in\Ltom$.)
\end{proof}

\begin{lemma}
\label{dimNTbih1:lem2}
Let Assumption \ref{ass:stronglip} and Assumption \ref{ass:curvesandsurfaces} be satisfied. Then
$\calB^{\bihone}_{N}$ is linearly independent.
\end{lemma}

\begin{proof}
Let $\gamma_{j,k}\in\R$, $j\in\{1,\ldots,p\}$, $k\in\{0,\ldots,3\}$, 
be such that $\displaystyle\sum_{j=1}^{p}\sum_{k=0}^{3}\gamma_{j,k}\pi\Theta_{j,k}=0$. 
Then \eqref{pathint2-bih1} implies for $l\in\{1,\dots,p\}$
\begin{align*}
0&=\sum_{j=1}^{p}\sum_{k=0}^{3}\gamma_{j,k}
\beta_{l,\ell}(\pi\Theta_{j,k})
=\gamma_{l,0},
&\ell&=0,\\
0&=\sum_{j=1}^{p}\sum_{k=0}^{3}\gamma_{j,k}
\beta_{l,\ell}(\pi\Theta_{j,k})
=\gamma_{l,\ell}+\gamma_{l,0}(x_{l,1})_{\ell}=\gamma_{l,\ell},
&\ell&\in\{1,2,3\},
\end{align*}
finishing the proof.
\end{proof}

\begin{theorem}
\label{dimNTbih1}
Let Assumptions \ref{ass:stronglip} and \ref{ass:curvesandsurfaces} be satisfied.
Then $\dim\harmbihoneNom=4p$ and a basis of $\harmbihoneNom$ is given by \eqref{basis:bih1N}.
\end{theorem}

\begin{proof}
Use Lemma \ref{dimNTbih1:lem1} and Lemma \ref{dimNTbih1:lem2}.
\end{proof}

\subsection{\except{toc}{Neumann Tensor Fields of the }Second \PZ Complex}
\label{app:sec:NeuTFbih2}

The rationale for the second \pz complex in comparison to the first one has to be changed appropriately. 
For this we also use Lemma \ref{PZformulalem}, the Poincar\'e maps, however, 
differ from one another ones.

In Lemma \ref{GgHSlemma} below for a tensor field $S$ and a parametrisation 
$\varphi\in C^{1,3}_{\pw}\big([0,1]\big)$ of a curve $\zeta$ we define
$$\int_{\zeta}\bscp{x-y}{S(y)\intd\lambda_{y}}
:=\int_{0}^{1}\Bscp{x-\varphi(t)}{S\big(\varphi(t)\big)\varphi'(t)}\intd t.$$

\begin{lemma}
\label{GgHSlemma}
Let $x,x_{0}\in\om$ and let $\zeta_{x_{0},x}\subset\om$ 
be a piecewise $C^{1}$-curve connecting $x_{0}$ and $x$.
\begin{itemize}
\item[\bf(i)]
Let $u\in\Ciom$. Then $u$ and its gradient $\grad u$ can be represented by
\begin{align*}
u(x)-u(x_{0})-\bscp{\grad u(x_{0})}{x-x_{0}}
&=\int_{\zeta_{x_{0},x}}\bscp{\int_{\zeta_{x_{0},y}}\Gg u\intd\lambda}{\intd\lambda_{y}}
\intertext{and}
\grad u(x)-\grad u(x_{0})
&=\int_{\zeta_{x_{0},x}}\Gg u\intd\lambda.
\end{align*}
\item[\bf(ii)]
For all $S\in\Cittom$ it holds
$$\int_{\zeta_{x_{0},x}}\bscp{\int_{\zeta_{x_{0},y}}S\intd\lambda}{\intd\lambda_{y}}
=\int_{\zeta_{x_{0},x}}\bscp{x-y}{S(y)\intd\lambda_{y}}.$$
\end{itemize}
\end{lemma}

\begin{proof}
For (i), we have
\begin{align*}
u(x)-u(x_{0})
&=\int_{\zeta_{x_{0},x}}\scp{\grad u}{\intd\lambda},\\
\p_{k}u(x)-\p_{k}u(x_{0})
&=\int_{\zeta_{x_{0},x}}\scp{\grad\p_{k}u}{\intd\lambda},\qquad
k\in\{1,2,3\},
\intertext{i.e.,}
\grad u(x)-\grad u(x_{0})
&=\int_{\zeta_{x_{0},x}}\Grad\grad u\intd\lambda.
\end{align*}
Therefore, 
\begin{align*}
u(x)-u(x_{0})
&=\int_{\zeta_{x_{0},x}}\bscp{\grad u(y)}{\intd\lambda_{y}}\\
&=\int_{\zeta_{x_{0},x}}\bscp{\int_{\zeta_{x_{0},y}}\Grad\grad u\intd\lambda}{\intd\lambda_{y}}
+\int_{\zeta_{x_{0},x}}\bscp{\grad u(x_{0})}{\intd\lambda_{y}}.
\end{align*}
Using $\varphi\in C^{1,3}_{\pw}\big([0,1]\big)$ as a parametrisation of $\zeta_{x_{0},x}$, 
we conclude the proof of (i) by
$$\int_{\zeta_{x_{0},x}}\bscp{\grad u(x_{0})}{\intd\lambda_{y}}
=\int_{0}^{1}\bscp{\grad u(x_{0})}{\varphi'(t)}\intd t
=\bscp{\grad u(x_{0})}{x-x_{0}}.$$

For (ii), we compute
\begin{align*}
\int_{\zeta_{x_{0},x}}\bscp{\int_{\zeta_{x_{0},y}}S\intd\lambda}{\intd\lambda_{y}}
&=\int_{0}^{1}\bscp{\int_{\zeta_{x_{0},\varphi(s)}}S\intd\lambda}{\varphi'(s)}\intd s\\
&=\int_{0}^{1}\Bscp{\int_{0}^{s}S\big(\varphi(t)\big)\varphi'(t)\intd t}{\varphi'(s)}\intd s\\
&=\int_{0}^{1}\Bscp{S\big(\varphi(t)\big)\varphi'(t)}
{\int_{t}^{1}\varphi'(s)\intd s}\intd t\\
&=\int_{0}^{1}\Bscp{S\big(\varphi(t)\big)\varphi'(t)}{x-\varphi(t)}\intd t
=\int_{\zeta_{x_{0},x}}\bscp{x-y}{S(y)\intd\lambda_{y}}
\end{align*}
again with $\varphi$ parametrising $\zeta_{x_{0},x}$.
\end{proof}

\begin{proposition}
\label{prop:bih2}
Let $x_{0}\in\om_{0}\in\cc(\om)$ and let $w\in\Cit(\om_{0})$ and $S\in\Citt(\om_{0})$.
\begin{itemize}
\item[\bf(a)]
The following conditions are equivalent:
\begin{itemize}
\item[\bf(i)]
For all $\zeta\subset\om_{0}$ closed, piecewise $C^1$-curves
$$\int_{\zeta}S\intd\lambda=0.$$
\item[\bf(ii)]
For all $\zeta_{x_{0},x},\widetilde{\zeta}_{x_{0},x}\subset\om_{0}$ 
piecewise $C^1$-curves connecting $x_{0}$ with $x$
$$\int_{\zeta_{x_{0},x}}S\intd\lambda
=\int_{\widetilde{\zeta}_{x_{0},x}}S \intd\lambda.$$
\item[\bf(iii)]
There exists $v\in\Cit(\om_{0})$ such that $\Grad v=S$.
\end{itemize} 
In the case one of the above conditions is true the vector field
\begin{equation}
\label{eq:chvbih2}
x\mapsto v(x)=\int_{\zeta_{x_{0},x}}S\intd\lambda
\end{equation}
for some $\zeta_{x_{0},x}\subset \om_{0}$ piecewise $C^1$-curve 
connecting $x_{0}$ with $x$ is a (well-defined) possible choice for $v$ in (iii). 
\item[\bf(b)]
The following conditions are equivalent:
\begin{itemize}
\item[\bf(i)]
For all $\zeta\subset\om_{0}$ closed, piecewise $C^1$ curves
$$\int_{\zeta}\scp{w}{\intd\lambda}=0.$$
\item[\bf(ii)]
For all $\zeta_{x_{0},x},\widetilde{\zeta}_{x_{0},x}\subset\om_{0}$ 
piecewise $C^1$ curves connecting $x_{0}$ with $x$
$$\int_{\zeta_{x_{0},x}}\scp{w}{\intd\lambda}
=\int_{\widetilde{\zeta}_{x_{0},x}}\scp{w}{\intd\lambda}.$$
\item[\bf(iii)]
There exists $u\in\Ci(\om_{0})$ such that $\grad u=w$.
\end{itemize}
In the case one of the above conditions is true the function
\begin{equation}
\label{eq:chubih2}
x\mapsto u(x)=\int_{\zeta_{x_{0},x}}\bscp{w}{\intd\lambda}
\end{equation}
for some $\zeta_{x_{0},x}\subset\om_{0}$ piecewise $C^1$-curve 
connecting $x_{0}$ with $x$ is a (well-defined) possible choice for $u$ in (iii).
\item[\bf(c)]
Let $v\in\Cit(\om_{0})$ with $\Grad v=S$ as in (a), (iii)
and let $u\in\Ci(\om_{0})$ with $\grad u=v$ as in (b), (iii).
Then $\Gg u=S$, $\skw S=0$, and $\CS S=0$.
\end{itemize} 
\end{proposition}

\begin{proof}
The statements in (a) and (b) are straightforward consequences 
of the fundamental theorem of calculus 
and follow essentially the same lines as (a) and (b) of Proposition \ref{prop:HSdevGrad}. 
Schwarz's Lemma and the complex property show (c).
\end{proof}

\begin{remark}
\label{curlformula1-bih2}
Related to Proposition \ref{prop:bih2} we note with Lemma \ref{PZformulalem} the following:
\begin{itemize}
\item[\bf(i)]
For $v\in\Citom$ we have $\curl v=2\spn^{-1}\skw\Grad v$.
\item[\bf(ii)]
If $\om_{0}$ is simply connected, 
Proposition \ref{prop:bih2} (a), (iii)
and (b), (iii) are equivalent to $\Curl S=0$ and $\curl w=0$, respectively.
\end{itemize}
\end{remark}

Similar to the first \pz complex, there exists an analogous version 
of Proposition \ref{prop:bih2} irrespective of the components. 
We only formulate the following slightly weaker statement, 
which is an easy consequence of Proposition \ref{prop:bih2}.

\begin{corollary}
\label{cor:bih2}
Let $w\in\Citom$ and $S\in\Cittom$.
\begin{itemize}
\item[\bf(a)]
The following conditions are equivalent:
\begin{itemize}
\item[\bf(i)] 
For all $\zeta\subset\om$ closed, piecewise $C^1$-curves $\int_{\zeta}S\intd\lambda=0$.
\item[\bf(ii)]
There exists $v\in\Citom$ such that $\Grad v=S$.
\end{itemize} 
\item[\bf(b)]
The following conditions are equivalent:
\begin{itemize}
\item[\bf(i)] 
For all $\zeta\subset\om$ closed, piecewise $C^1$-curves $\int_{\zeta}\scp{w}{\intd\lambda}=0$.
\item[\bf(ii)]
There exists $u\in\Ciom$ such that $\grad u=w$.
\end{itemize} 
\item[\bf(c)] 
Let $v\in\Citom$ with $\Grad v=S$ as in (a), (ii)
and let $u\in\Ciom$ with $\grad u=v$ as in (b), (ii).
Then $\Gg u=S$, $\skw S=0$, and $\CS S=0$.
\end{itemize} 
\end{corollary}

%%%\begin{remark}
%%%\label{curlformula1-bih2}
%%%In Lemma \ref{GgHSlemma} (iii)
%%%for $S\in\Ci_{\Sb}(\om,\R^{3\times3})$ with $\Grad v=S$ the formula
%%%$$\curl v=2\spn^{-1}\skw S=0$$
%%%is crucial. 
%%%\end{remark}

Next, we turn to the exact construction of the Neumann fields for the second \pz complex. 
Let $j\in\{1,\ldots,p\}$. For this, recall from the beginning of Section \ref{app:sec:NeuF} 
that $\theta_{j}$ is constant on each connected component $\Upsilon_{j,0}$ and $\Upsilon_{j,1}$
of $\Upsilon_{j}\setminus F_{j}$ and vanishes outside of $\widetilde{\Upsilon}_{j,1}$.
Moreover, let $\widehat{p}_{k}$ be the polynomials from Section \ref{app:sec:DirTFbih1}
given by $\widehat{p}_{0}(x):=1$ and $\widehat{p}_{k}(x):=x_{k}$ for $k\in\{1,2,3\}$.
We define the functions $\theta_{j,k}:=\theta_{j}\widehat{p}_{k}$
and note $\Gg\theta_{j,k}=0$
in $\dot\bigcup_{l\in\{1,\ldots,p\}}\Upsilon_{l}\setminus F_{l}$ for all $k\in\{1,2,3\}$.
Thus $\Gg\theta_{j,k}$ can be continuously extended by zero 
to $\Theta_{j,k}\in\Cittom\cap\LtttSom$
with $\Theta_{j,k}=0$ in all the neighbourhoods 
$\Upsilon_{l}$ of all the surfaces $F_{l}$, $l\in\{1,\ldots,p\}$.

\begin{lemma}
\label{dimNTbih2:lem0}
Let Assumption \ref{ass:curvesandsurfaces} be satisfied. Then
$\Theta_{j,k}\in\ker(\CS,\om)$.
\end{lemma}

\begin{proof}
Let $\Phi\in\CicttTom$. 
As $\supp\Theta_{j,k}\subset\widetilde{\Upsilon}_{j}\setminus\Upsilon_{j}$
we can pick another cut-off function $\varphi\in\Cic(\om_{F})$ with $\varphi|_{\supp\Theta_{j,k}\cap\supp\Phi}=1$. Then
\begin{align*}
&\qquad\scp{\Theta_{j,k}}{\symCT\Phi}_{\LtttSom}
=\scp{\Theta_{j,k}}{\symCT\Phi}_{\LtttS(\supp\Theta_{j,k}\cap\supp\Phi)}\\
&=\bscp{\Gg\theta_{j,k}}{\symCT(\varphi\Phi)}_{\LtttS(\om_{F})}
=\bscp{\Grad(\grad\theta_{j,k})}{\Curl(\varphi\Phi)}_{\Lttt(\om_{F})}
=0
\end{align*}
as $\varphi\Phi,\Curl(\varphi\Phi)\in\Cictt(\om_{F})$.
\end{proof}

Similar to the first \pz complex, we introduce a set of functionals. 

For $l,j\in\{1,\dots,p\}$ and $k\in\{0,\dots,3\}$ 
and for the curves $\zeta_{x_{l,0},x_{l,1}}\subset\zeta_{l}$
with the chosen starting points $x_{l,0}\in\Upsilon_{l,0}$ 
and respective endpoints $x_{l,1}\in\Upsilon_{l,1}$
we can compute by Lemma \ref{GgHSlemma}
\begin{align*}
\R^{3}\ni b_{l}(\Theta_{j,k})
&:=\int_{\zeta_{l}}\Theta_{j,k}\intd\lambda
=\int_{\zeta_{x_{l,0},x_{l,1}}}\Gg\theta_{j,k}\intd\lambda\\
&\phantom{:}=\grad\theta_{j,k}(x_{l,1})-\grad\theta_{j,k}(x_{l,0})
=\grad\theta_{j,k}(x_{l,1})\\
&\phantom{:}=\delta_{l,j}\grad\widehat{p}_{k}(x_{l,1})
=\delta_{l,j}\begin{cases}
0,&\text{if }k=0,\\
e^{k},&\text{if }k=1,2,3,
\end{cases}
\intertext{and}
\R\ni\beta_{l,0}(\Theta_{j,k})
&:=\int_{\zeta_{l}}
\bscp{x_{l,1}-y}{\Theta_{j,k}(y)\intd\lambda_{y}}\\
&\phantom{:}=\int_{\zeta_{x_{l,0},x_{l,1}}}
\bscp{x_{l,1}-y}{\Gg\theta_{j,k}(y)\intd\lambda_{y}}\\
&\phantom{:}=\int_{\zeta_{x_{l,0},x_{l,1}}}
\bscp{\int_{\zeta_{x_{l,0},y}}\Gg\theta_{j,k}\intd\lambda}{\intd\lambda_{y}}\\
&\phantom{:}=\theta_{j,k}(x_{l,1})
-\theta_{j,k}(x_{l,0})-\bscp{\grad\theta_{j,k}(x_{l,0})}{x_{l,1}-x_{l,0}}
=\theta_{j,k}(x_{l,1})\\
&\phantom{:}=\delta_{l,j}\widehat{p}_{k}(x_{l,1})
=\delta_{l,j}\begin{cases}
1,&\text{if }k=0,\\
(x_{l,1})_{k},&\text{if }k\in\{1,2,3\}.
\end{cases}
\end{align*}
Thus, for $l\in\{1,\dots,p\}$ and $\ell\in\{0,\dots,3\}$
we have functionals $\beta_{l,\ell}$ given by
\begin{align*}
\beta_{l,\ell}(\Theta_{j,k})
&:=\bscp{b_{l}(\Theta_{j,k})}{e^{\ell}}
=\delta_{l,j}\begin{cases}
0,&\text{if }k=0,\\
\delta_{\ell,k},&\text{if }k\in\{1,2,3\},
\end{cases}
\intertext{for $l,j\in\{1,\dots,p\}$ and $\ell\in\{1,2,3\}$ and $k\in\{0,1,2,3\}$, as well as}
\beta_{l,0}(\Theta_{j,k})
&=\delta_{l,j}\delta_{0,k}
+\delta_{l,j}(1-\delta_{0,k})(x_{l,1})_{k}
\end{align*}
for $l,j\in\{1,\dots,p\}$ and $k\in\{0,1,2,3\}$.
Therefore, we have 
\begin{align}
\label{pathint1-bih2}
\beta_{l,\ell}(\Theta_{j,k})
&=\delta_{l,j}\delta_{\ell,k}
+(1-\delta_{0,k})\delta_{\ell,0}\delta_{l,j}(x_{l,1})_{k},\quad
l,j\in\{1,\dots,p\},\;k,\ell\in\{0,1,2,3\}.
\end{align}

Let Assumption \ref{ass:stronglip} be satisfied.
For the second \pz complex, similar to \eqref{deco1}, \eqref{deco2}, 
we have the orthogonal decompositions 
\begin{align}
\label{app:deco2-bih2}
\begin{aligned}
\LtttSom&=\ran(\Gg,\om)\oplus_{\LtttSom}\ker(\rdDS,\om),\\
\ker(\CS,\om)&=\ran(\Gg,\om)\oplus_{\LtttSom}\harmbihtwoNom.
\end{aligned}
\end{align}

\begin{remark}
\label{rem:Helmholtz-bih2N}
Lemma \ref{lem:regGg} shows $\dom(\Gg,\om)=H^{2}(\om)$. 
Thus, employing a contradiction argument together with Rellich's selection theorem, 
we obtain the Poincar\'e type estimate
$$\exists\,c>0\quad\forall\,\phi\in H^{2}(\om)\cap(\Pol^{1}_{\pw})^{\bot_{\Ltom}}\qquad
\norm{\phi}_{\Ltom}\leq c\norm{\Grad\grad\phi}_{\Ltttom},$$
as Assumption \ref{ass:stronglip} holds. 
Thus, the range in \eqref{app:deco2-bih2} is closed.
\end{remark}

Let $\pi:\LtttSom \to \ker(\rdDS,\om)$ 
denote the orthogonal projector onto $\ker(\rdDS,\om)$ along $\ran(\Gg,\om)$, 
see \eqref{app:deco2-bih2}. 
In particular, $\pi\big(\ker(\CS,\om)\big)=\harmbihtwoNom$.
By Lemma \ref{dimNTbih2:lem0} there exists 
some $\psi_{j,k}\in H^{2}(\om)$ such that 
\begin{align*}
\harmbihtwoNom\ni\pi\Theta_{j,k}
&=\Theta_{j,k}-\Gg\psi_{j,k},\\
(\Theta_{j,k}-\Gg\psi_{j,k})\big|_{\om_{F}}
&=\Gg(\theta_{j,k}-\psi_{j,k}).
\end{align*}
As $\harmbihtwoNom\subset\Cittom$, see Lemma \ref{lem:harmreg},
we conclude by $\pi\Theta_{j,k},\Theta_{j,k}\in\Cittom$
that also $\Gg\psi_{j,k}\in\Cittom$ and hence
$\psi_{j,k}\in\Ciom$.
Hence all path integrals over the closed curves $\zeta_{l}$ are well-defined. 
Furthermore, we observe by Lemma \ref{GgHSlemma}
\begin{align*}
b_{l}(\Gg\psi_{j,k})
&=\int_{\zeta_{l}}\Gg\psi_{j,k}\intd\lambda
=\grad\psi_{j,k}(x_{l,1})-\grad\psi_{j,k}(x_{l,1})
=0
\intertext{and}
\beta_{l,0}(\Gg\psi_{j,k})
&=\int_{\zeta_{l}}
\bscp{x_{l,1}-y}{\Gg\psi_{j,k}(y)\intd\lambda_{y}}\\
&=\int_{\zeta_{x_{l,1},x_{l,1}}}
\bscp{\int_{\zeta_{x_{l,1},y}}\Gg\psi_{j,k}\intd\lambda}{\intd\lambda_{y}}\\
&=\psi_{j,k}(x_{l,1})
-\psi_{j,k}(x_{l,1})-\bscp{\grad\psi_{j,k}(x_{l,1})}{x_{l,1}-x_{l,1}}
=0.
\end{align*}
Therefore, $\beta_{l,\ell}(\Gg\psi_{j,k})=0$ and by \eqref{pathint1-bih2} we get
\begin{align}
\label{pathint2-bih2}
\beta_{l,\ell}(\pi\Theta_{j,k})
=\beta_{l,\ell}(\Theta_{j,k})
-\beta_{l,\ell}(\Gg\psi_{j,k})
=\delta_{l,j}\delta_{\ell,k}
+(1-\delta_{0,k})\delta_{\ell,0}\delta_{l,j}(x_{l,1})_{k}
\end{align}
for all $l,j\in\{1,\dots,p\}$ and all $\ell,k\in\{0,1,2,3\}$. 
We shall show that 
\begin{align}
\label{basis:bih2N}
\calB^{\bihtwo}_{N}:=\big\{\pi\Theta_{j,k}:j\in\{1,\dots,p\},\,k\in\{0,\ldots,3\}\big\}
\subset\harmbihtwoNom
\end{align}
defines a basis of $\harmbihtwoNom$.

\begin{remark}[Characterisation by PDEs]
Let $j\in\{1,\dots,p\}$ and $k\in\{0,\ldots,3\}$. 
Then $\psi_{j,k}\in H^{2}(\om)\cap(\Pol^{1}_{\pw})^{\bot_{\Ltom}}$ 
can be found by the variational formulation
$$\forall\,\phi\in H^{2}(\om)\qquad
\scp{\Gg\psi_{j,k}}{\Gg\phi}_{\Ltttom}
=\scp{\Theta_{j,k}}{\Gg\phi}_{\Ltttom},$$
i.e., $\psi_{j,k}=(\Delta_{NN}^{2})^{-1}\big(\dDS\Theta_{j,k},\Theta_{j,k}\nu|_{\ga},\nu\cdot\Dive\Theta_{j,k}|_{\ga}\big)$, 
where $\Delta_{NN}^{2}\subset\dDS\Gg$ is the bi-Laplacian 
with inhomogeneous Neumann type boundary conditions
restricted to a subset of $H^{2}(\om)\cap(\Pol^{1}_{\pw})^{\bot_{\Ltom}}$.
Therefore, 
\begin{align*}
\pi\Theta_{j,k}
&=\Theta_{j,k}-\Gg\psi_{j,k}\\
&=\Theta_{j,k}-\Gg(\Delta_{NN}^{2})^{-1}\big(\dDS\Theta_{j,k},\Theta_{j,k}\nu|_{\ga},\nu\cdot\Dive\Theta_{j,k}|_{\ga}\big).
\end{align*}
In classical terms, $\psi_{j,k}$
solves the \pz Neumann problem
\begin{align}
\label{NeuLap1-bih2}
\begin{aligned}
\Delta^{2}\psi_{j,k}&=\dDS\Theta_{j,k}
&&\text{in }\om,\\
(\Gg\psi_{j,k})\nu&=\Theta_{j,k}\nu
&&\text{on }\ga,\\
\nu\cdot\Dive\Gg\psi_{j,k}&=\nu\cdot\Dive\Theta_{j,k}
&&\text{on }\ga,\\
\int_{\om_{l}}\psi_{j,k}&=0
&&\text{for }l\in\{1,\dots,n\},\\
\int_{\om_{l}}x_{\ell}\psi_{j,k}(x)\intd\lambda_{x}&=0
&&\text{for }l\in\{1,\dots,n\},\quad\ell\in\{1,2,3\},
\end{aligned}
\end{align}
which is uniquely solvable.
\end{remark}

\begin{lemma}
\label{dimNTbih2:lem1}
Let Assumptions \ref{ass:stronglip} and \ref{ass:curvesandsurfaces} be satisfied. 
Then $\harmbihtwoNom=\lin\calB^{\bihtwo}_{N}$.
\end{lemma}

\begin{proof}
Let $H\in\harmbihtwoNom=\ker(\rdDS,\om)\cap\ker(\CS,\om)\subset\Citt_{\Sb}(\om)$, 
see Lemma \ref{lem:harmreg}.
With the above introduced functions $\beta_{l,0}$ and $b_{l}$, $l\in\{1,\dots,p\}$, we recall
\begin{align*}
\R^{3}\ni b_{l}(H)
&=\int_{\zeta_{l}}H\intd\lambda,\\
\R\ni\beta_{l,0}(H)
&=\int_{\zeta_{l}}
\bscp{x_{l,1}-y}{H(y)\intd\lambda_{y}},
\end{align*}
and define for $l\in\{1,\dots,p\}$ the numbers
\begin{align*}
\gamma_{l,\ell}
:=\gamma_{l,\ell}(H)
&:=\bscp{b_{l}(H)}{e^{\ell}}
=\beta_{l,\ell}(H),\qquad
\ell\in\{1,2,3\},\\
\gamma_{l,0}
:=\gamma_{l,0}(H)
&:=\beta_{l,0}(H)-\sum_{k=1}^{3}\beta_{l,k}(H)(x_{l,1})_{k}.
\end{align*}
We shall show that
$$\harmbihtwoNom\ni\widehat{H}
:=H-\sum_{j=1}^{p}\sum_{k=0}^{3}\gamma_{j,k}\pi\Theta_{j,k}=0\quad\text{in }\om.$$
Similar to the proof of Lemma \ref{dimNFrhm:lem1}, 
the aim is to prove that there exists $u\in H^{2}(\om)$
such that $\Gg u=\widehat{H}$, since then
$$\norm{\widehat{H}}_{\LtttSom}^{2}
=\scp{\Gg u}{\widehat{H}}_{\LtttSom}=0.$$ 
For this, we shall apply Corollary \ref{cor:bih2}
and Remark \ref{curlformula1-bih2} to $S=\widehat{H}$.
By \eqref{pathint2-bih2} we observe for $\ell\in\{1,2,3\}$ and $l\in\{1,\ldots,p\}$
\begin{align*}
\big(\int_{\zeta_{l}}\widehat{H}\intd\lambda\big)_{\ell}
&=\bscp{\int_{\zeta_{l}}\widehat{H}\intd\lambda}{e^{\ell}}
=\bscp{b_{l}(\widehat{H})}{e^{\ell}}
=\beta_{l,\ell}(\widehat{H})\\
&=\beta_{l,\ell}(H)
-\sum_{j=1}^{p}\sum_{k=0}^{3}\gamma_{j,k}
\beta_{l,\ell}(\pi\Theta_{j,k})
=\gamma_{l,\ell}
-\sum_{j=1}^{p}\sum_{k=0}^{3}\gamma_{j,k}
\delta_{l,j}\delta_{\ell,k}=0.
\end{align*}
Note that $\widehat{H}\in\ker(\Curl,\om)\cap\Cittom$.
Thus by Assumption \ref{ass:curvesandsurfaces} (A.1)
for any closed piecewise $C^1$-curve $\zeta$ in $\om$
\begin{align}
\label{pathint3-bih2}
\int_{\zeta}\widehat{H}\intd\lambda=0.
\end{align}
By Corollary \ref{cor:bih2} (a), we find $v\in\Citom$ such that $\Grad v=\widehat{H}$. 
Next, let $l\in\{1,\ldots,p\}$. 
Then, with $\zeta_{x_{l,0},x_{l,1}}\subset\zeta_{l}\subset\om_{0}$ for some $\om_{0}\in\cc(\om)$,
we obtain with $c:=v(x_{l,1})\in\R^{3}$ for all $x\in\zeta_{l}$
$$v(x)
=v(x)-v(x_{l,1})+c
=\int_{\zeta_{x_{l,1},x}}\Grad v\intd\lambda+c 
=\int_{\zeta_{x_{l,1},x}}\widehat{H}\intd\lambda+c$$
and
$$\int_{\zeta_{l}}\scp{c}{\intd\lambda}
=\sum_{\ell=1}^{3}c_{\ell}\int_{\zeta_{l}}\scp{\grad x_{\ell}}{\intd\lambda}
=0.$$
We consider the closed curve $\zeta_{l}$ 
as the closed curve $\zeta_{x_{l,1},x_{l,1}}$ with circulation $1$ along $\zeta_{l}$.
By Lemma \ref{GgHSlemma}, the definition of $\beta_{l,0}$,
and \eqref{pathint2-bih2} we have
\begin{align*}
\int_{\zeta_{l}}\scp{v}{\intd\lambda}
&=\int_{\zeta_{l}}\bscp{\int_{\zeta_{x_{l,1},y}}\widehat{H}\intd\lambda}{\intd\lambda_{y}}
=\int_{\zeta_{x_{l,1},x_{l,1}}}
\bscp{\int_{\zeta_{x_{l,1},y}}\widehat{H}\intd\lambda}{\intd\lambda_{y}}\\
&=\int_{\zeta_{l}}
\bscp{x_{l,1}-y}{\widehat{H}(y)\intd\lambda_{y}}
=\beta_{l,0}(\widehat{H})
=\beta_{l,0}(H)
-\sum_{j=1}^{p}\sum_{k=0}^{3}\gamma_{j,k}
\beta_{l,0}(\pi\Theta_{j,k})\\
&=\beta_{l,0}(H)
-\sum_{j=1}^{p}\sum_{k=0}^{3}\gamma_{j,k}
\big(\delta_{l,j}\delta_{0,k}
+(1-\delta_{0,k})\delta_{l,j}(x_{l,1})_{k}\big)\\
&=\beta_{l,0}(H)
-\gamma_{l,0}
-\sum_{k=1}^{3}\gamma_{l,k}(x_{l,1})_{k}
=\beta_{l,0}(H)
-\gamma_{l,0}
-\sum_{k=1}^{3}\beta_{l,k}(H)(x_{l,1})_{k}
=0.
\end{align*}
Note that $v\in\ker(\curl,\om)\cap\Citom$
by Remark \ref{curlformula1-bih2} (i) 
as $\Grad v=\widehat{H}\in\LtttSom$.
Therefore, by Assumption \ref{ass:curvesandsurfaces} (A.1)
for any closed piecewise $C^1$-curve $\zeta$ in $\om$
\begin{align}
\label{pathint4-bih2}
\int_{\zeta}\scp{v}{\intd\lambda}=0.
\end{align}
Hence, by Corollary \ref{cor:bih2} (b), we find $u\in\Ciom$ with $\grad u=v$ and thus
$$\Gg u=\Grad v=\widehat{H}\in\Cittom\cap\LtttSom.$$
Similar to the end of the proof of Lemma \ref{dimNFrhm:lem1},
elliptic regularity and, e.g., \cite[Theorem 2.6 (1)]{L1986} or \cite[Theorem 3.2 (2)]{A1965}
show that $v\in\Citom$ together with $\Grad v\in\Ltttom$ imply
$v\in H^{1,3}(\om)$.
Then, analogously, $u\in\Ciom$
and $\grad u=v\in\Lttom$ imply $u\in H^{1}(\om)$
and hence $u\in H^{2}(\om)$, completing the proof. 
\end{proof}

\begin{lemma}
\label{dimNTbih2:lem2}
Let Assumption \ref{ass:stronglip} and Assumption \ref{ass:curvesandsurfaces} be satisfied. Then
$\calB^{\bihtwo}_{N}$ is linearly independent.
\end{lemma}

\begin{proof}
Let $\gamma_{j,k}\in\R$, $j\in\{1,\ldots,p\}$, and $k\in\{0,\ldots,3\}$
be such that $\displaystyle\sum_{j=1}^{p}\sum_{k=0}^{3}\gamma_{j,k}\pi\Theta_{j,k}=0$. Then
\eqref{pathint2-bih2} implies for $l\in\{1,\dots,p\}$
\begin{align*}
0&=\sum_{j=1}^{p}\sum_{k=0}^{3}\gamma_{j,k}
\beta_{l,\ell}(\pi\Theta_{j,k})
=\gamma_{l,\ell},
&\ell&\in\{1,2,3\},\\
0&=\sum_{j=1}^{p}\sum_{k=0}^{3}\gamma_{j,k}
\beta_{l,\ell}(\pi\Theta_{j,k})
=\gamma_{l,0}+\sum_{k=1}^{3}\gamma_{l,k}(x_{l,1})_{k}=\gamma_{l,0},
&\ell&=0,
\end{align*}
finishing the proof.
\end{proof}

\begin{theorem}
\label{dimNTbih2}
Let Assumptions \ref{ass:stronglip} and \ref{ass:curvesandsurfaces} be satisfied.
Then $\dim\harmbihtwoNom=4p$ and a basis of $\harmbihtwoNom$ is given by \eqref{basis:bih2N}.
\end{theorem}

\begin{proof}
Use Lemma \ref{dimNTbih2:lem1} and Lemma \ref{dimNTbih2:lem2}.
\end{proof}

\subsection{\except{toc}{Neumann Tensor Fields of the }Elasticity Complex}
\label{app:sec:NeuTFela}

The concluding example for our general construction principle is the elasticity complex. 
Again, we require some preparations regarding integration along curves 
and the operators involved in the elasticity complex. 
We need the formulas providing the interaction of matrix and vector analytic operations 
as outlined in Lemma \ref{PZformulalem} (in particular, note we defined $\spn$ there.)

In Lemma \ref{symGradHSlemma} below for a tensor field $S$ 
and a parametrisation $\varphi\in C^{1,3}_{\pw}\big([0,1]\big)$ of a curve $\zeta$ in $\om$ we define
$$\int_{\zeta}
\spn\big((\Curl S)^{\top}(y)\intd\lambda_{y}\big)(x-y)
:=\int_{0}^{1}\spn\Big((\Curl S)^{\top}\big(\varphi(t)\big)\varphi'(t)\Big)
\big(x-\varphi(t)\big)\intd t.$$

\begin{lemma}
\label{symGradHSlemma}
Let $x,x_{0}\in\om$ and let $\zeta_{x_{0},x}\subset\om$ be a piecewise $C^{1}$-curve 
connecting $x_{0}$ to $x$.
\begin{itemize}
\item[\bf(i)]
Let $v\in\Citom$. Then $v$ and its rotation $\curl v$ can be represented by
\begin{align*}
&\qquad v(x)-v(x_{0})-\frac{1}{2}\big(\curl v(x_{0})\big)\times(x-x_{0})\\
&=\int_{\zeta_{x_{0},x}}\symGrad v\intd\lambda
+\int_{\zeta_{x_{0},x}}\int_{\zeta_{x_{0},y}}
\spn\big((\Curl\symGrad v)^{\top}\intd\lambda\big)\intd\lambda_{y}
\end{align*}
and
$$\curl v(x)-\curl v(x_{0})
=2\int_{\zeta_{x_{0},x}}(\Curl\symGrad v)^{\top}\intd\lambda.$$
\item[\bf(ii)]
Let $S\in\Cittom$. Then
$$\int_{\zeta_{x_{0},x}}\int_{\zeta_{x_{0},y}}
\spn\big((\Curl S)^{\top}\intd\lambda\big)\intd\lambda_{y}
=\int_{\zeta_{x_{0},x}}
\spn\big((\Curl S)^{\top}(y)\intd\lambda_{y}\big)(x-y).$$
\end{itemize}
\end{lemma}

\begin{proof}
For (i), let
$$S:=\symGrad v=\Grad v-\skw\Grad v$$
and observe $2\Curl S=-2\Curl\skw\Grad v=(\Grad\curl v)^{\top}$ 
by Lemma \ref{PZformulalem}. Thus 
\begin{align*}
v_{k}(x)-v_{k}(x_{0})
&=\int_{\zeta_{x_{0},x}}\scp{\grad v_{k}}{\intd\lambda},\qquad
k\in\{1,2,3\},\\
v(x)-v(x_{0})
&=\int_{\zeta_{x_{0},x}}\Grad v\intd\lambda,\\
\curl v(x)-\curl v(x_{0})
&=\int_{\zeta_{x_{0},x}}\Grad\curl v\intd\lambda
=2\int_{\zeta_{x_{0},x}}(\Curl S)^{\top}\intd\lambda.
\end{align*}
Therefore, by Lemma \ref{PZformulalem}
\begin{align*}
&\qquad v(x)-v(x_{0})\\
&=\int_{\zeta_{x_{0},x}}\Grad v\intd\lambda
=\int_{\zeta_{x_{0},x}}\symGrad v\intd\lambda
+\int_{\zeta_{x_{0},x}}\skw\Grad v\intd\lambda\\
&=\int_{\zeta_{x_{0},x}}\symGrad v\intd\lambda
+\frac{1}{2}\int_{\zeta_{x_{0},x}}\spn\curl v(y)\intd\lambda_{y}\\
&=\int_{\zeta_{x_{0},x}}S\intd\lambda
+\frac{1}{2}\int_{\zeta_{x_{0},x}}\spn\curl v(x_{0})\intd\lambda_{y}
+\int_{\zeta_{x_{0},x}}\spn\big(\int_{\zeta_{x_{0},y}}(\Curl S)^{\top}\intd\lambda\big)\intd\lambda_{y}\\
&=\int_{\zeta_{x_{0},x}}S\intd\lambda
+\frac{1}{2}\int_{\zeta_{x_{0},x}}\spn\curl v(x_{0})\intd\lambda_{y}+\int_{\zeta_{x_{0},x}}\int_{\zeta_{x_{0},y}}
\spn\big((\Curl S)^{\top}\intd\lambda\big)\intd\lambda_{y}.
\end{align*}
Moreover, with $\varphi\in C^{1,3}_{\pw}\big([0,1]\big)$ parametrising $\zeta_{x_{0},x}$
\footnote{Alternatively, we can compute with $\id=\Grad y$
\begin{align*}
\int_{\zeta_{x_{0},x}}\underbrace{\spn\curl v(x_{0})}_{=(\spn\curl v(x_{0}))\id}\intd\lambda_{y}
&=\spn\curl v(x_{0})\int_{\zeta_{x_{0},x}}\Grad y\intd\lambda_{y}
=\big(\spn\curl v(x_{0})\big)(x-x_{0}).
\end{align*}}
\begin{align*}
\int_{\zeta_{x_{0},x}}\spn\curl v(x_{0})\intd\lambda_{y}
&=\int_{0}^{1}\big(\spn\curl v(x_{0})\big)\varphi'(s)\intd s \\
&=\big(\spn\curl v(x_{0})\big)(x-x_{0})
=\big(\curl v(x_{0})\big)\times(x-x_{0}).
\end{align*}
For (ii), we compute
\begin{align*}
\int_{\zeta_{x_{0},x}}\int_{\zeta_{x_{0},y}}
\spn\big((\Curl S)^{\top}\intd\lambda\big)\intd\lambda_{y}
&=\int_{0}^{1}\Big(\int_{\zeta_{x_{0},\varphi(s)}}
\spn\big((\Curl S)^{\top}\intd\lambda\big)\Big)\varphi'(s)\intd s\\
&=\int_{0}^{1}\Big(\int_{0}^{s}
\spn\Big((\Curl S)^{\top}\big(\varphi(t)\big)\varphi'(t)\Big)\intd t\Big)\varphi'(s)\intd s\\
&=\int_{0}^{1}\spn\Big((\Curl S)^{\top}\big(\varphi(t)\big)\varphi'(t)\Big)
\int_{t}^{1}\varphi'(s)\intd s\intd t\\
&=\int_{0}^{1}\spn\Big((\Curl S)^{\top}\big(\varphi(t)\big)\varphi'(t)\Big)\big(x-\varphi(t)\big)\intd t\\
&=\int_{\zeta_{x_{0},x}}
\spn\big((\Curl S)^{\top}(y)\intd\lambda_{y}\big)(x-y)
\end{align*}
with $\varphi$ from above.
\end{proof}

\begin{proposition}
\label{prop:ela}
Let $x_{0}\in\om_{0}\in\cc(\om)$ and let $S,T\in\Citt(\om_{0})$.
\begin{itemize}
\item[\bf(a)]
The following conditions are equivalent:
\begin{itemize}
\item[\bf(i)]
For all $\zeta\subset\om_{0}$ closed, piecewise $C^1$-curves
$$\int_{\zeta}(\Curl S)^{\top}\intd\lambda=0.$$
\item[\bf(ii)]
For all $\zeta_{x_{0},x},\widetilde{\zeta}_{x_{0},x}\subset\om_{0}$ 
piecewise $C^1$-curves connecting $x_{0}$ with $x$
$$\int_{\zeta_{x_{0},x}}(\Curl S)^{\top}\intd\lambda
=\int_{\widetilde{\zeta}_{x_{0},x}}(\Curl S)^{\top}\intd\lambda.$$
\item[\bf(iii)]
There exists $w\in\Cit(\om_{0})$ such that $\Grad w=(\Curl S)^{\top}$.
\end{itemize} 
In the case one of the above conditions is true
\begin{equation}
\label{eq:chwela}
x\mapsto w(x)=\int_{\zeta_{x_{0},x}}(\Curl S)^{\top}\intd\lambda
\end{equation}
for some $\zeta_{x_{0},x}\subset\om_{0}$ piecewise $C^1$-curve 
connecting $x_{0}$ with $x$ is a (well-defined) possible choice for $w$ in (iii).
\item[\bf(b)]
The following conditions are equivalent:
\begin{itemize}
\item[\bf(i)]
For all $\zeta\subset\om_{0}$ closed, piecewise $C^1$-curves
$$\int_{\zeta}T\intd\lambda=0.$$
\item[\bf(ii)]
For all $\zeta_{x_{0},x},\widetilde{\zeta}_{x_{0},x}\subset\om_{0}$ 
piecewise $C^1$-curves connecting $x_{0}$ with $x$
$$\int_{\zeta_{x_{0},x}}T\intd\lambda
=\int_{\widetilde{\zeta}_{x_{0},x}}T\intd\lambda.$$
\item[\bf(iii)]
There exists $v\in\Cit(\om_{0})$ such that $\Grad v=T$.
\end{itemize}
 In the case one of the above conditions is true
\begin{equation}
\label{eq:chvela}
x\mapsto v(x)=\int_{\zeta_{x_{0},x}}T\intd\lambda
\end{equation}
for some $\zeta_{x_{0},x}\subset\om_{0}$ piecewise $C^1$-curve 
connecting $x_{0}$ with $x$ is a (well-defined) possible choice for $v$ in (iii). 
\item[\bf(c)]
Let $T:=S+\spn w$ with $w\in\Cit(\om_{0})$ and $\Grad w=(\Curl S)^{\top}$ as in (a), (iii).
Moreover, let $v\in\Cit(\om_{0})$ with $\Grad v=T$ as in (b), (iii). Then 
\begin{itemize}
\item[\bf(i)]
$\skw S=0$,
\item[\bf(ii)]
$\symGrad v=S$
\end{itemize}
are equivalent. In either case, we have 
\begin{equation}\label{eq:remformela}
\CCtS S=0,\qquad
\Grad w=\frac{1}{2}\Grad\curl v.
\end{equation}
\end{itemize}
\end{proposition}

\begin{proof}
The proofs of (a) and (b) are easy (fundamental theorem of calculus)
and follow in a similar way to Propositions \ref{prop:bih2}. 
For (c), we compute $\symGrad v=\sym T=\sym S$.
Hence (i) and (ii) are equivalent.
Finally, if (i) or (ii) is true, then by the complex property
$$\CCt S=\CCt\symGrad v=0,$$
and we conclude
$\Grad w=(\Curl\sym\Grad v)^{\top}=\frac{1}{2}\Grad\curl v$
by Lemma \ref{PZformulalem}.
\end{proof}

The respective result for the whole of $\om$ reads as follows. 

\begin{corollary}
\label{cor:ela}
Let $S,T\in\Cittom$.
\begin{itemize}
\item[\bf(a)]
The following conditions are equivalent:
\begin{itemize}
\item[\bf(i)]
For all $\zeta\subset\om$ closed, piecewise $C^1$-curves
$\int_{\zeta}(\Curl S)^{\top}\intd\lambda=0$.
\item[\bf(ii)]
There exists $w\in\Citom$ such that $\Grad w=(\Curl S)^{\top}$.
\end{itemize}
\item[\bf(b)]
The following conditions are equivalent:
\begin{itemize}
\item[\bf(i)]
For all $\zeta\subset\om$ closed, piecewise $C^1$-curves
$\int_{\zeta}T\intd\lambda=0$.
\item[\bf(ii)]
There exists $v\in\Cit(\om)$ such that $\Grad v=T$.
\end{itemize}
\item[\bf(c)]
Let $T=S+\spn w$ with $w\in\Cit(\om)$ and $\Grad w=(\Curl S)^{\top}$ as in (a), (ii).
Moreover, let $v\in\Cit(\om)$ with $\Grad v=T$ as in (b), (ii). Then 
$\skw S=0$ in $\om$ if and only if $\symGrad v=S$ in $\om$.
\end{itemize}
\end{corollary}

\begin{remark}
\label{curlformula1-ela}
Related to Proposition \ref{prop:ela} and Corollary \ref{cor:ela} we note with Lemma \ref{PZformulalem} the following:
\begin{itemize}
\item[\bf(i)]
For $S\in\Citt_{\Sb}(\om)$
and $T:=S+\spn w$
with $\Grad w=(\Curl S)^{\top}$ it holds
$$\Curl T
=\Curl S+(\dive w)\id-((\Curl S)^{\top})^{\top}
=\tr\Grad w
=\tr\Curl S
=\tr\Curl\skw S
=0.$$
\item[\bf(ii)]
If $\om_{0}$ is simply connected, 
Proposition \ref{prop:ela} (a), (iii)
and (b), (iii) are equivalent to $\Curl(\Curl S)^{\top}=0$ and $\Curl T=0$, respectively.
\end{itemize}
\end{remark}

Next, we provide the construction of the basis tensor fields for the Neumann fields for the elasticity complex.
Let $j\in\{1,\ldots,p\}$. 
From the beginning of Section \ref{app:sec:NeuF} recall that $\theta_{j}$ 
is constant on each connected component $\Upsilon_{j,0}$ and $\Upsilon_{j,1}$
of $\Upsilon_{j}\setminus F_{j}$
and vanishes outside of $\widetilde{\Upsilon}_{j,1}$.
Let $\widehat{r}_{k}$ be the rigid motions (Nedelec fields) 
from Section \ref{app:sec:DirTFela}
given by $\widehat{r}_{k}(x):=e^{k}\times x=\spn(e^{k})\,x$
and $\widehat{r}_{k+3}(x):=e^{k}$ for $k\in\{1,2,3\}$.
We define the vector fields $\theta_{j,k}:=\theta_{j}\widehat{r}_{k}$ and note $\symGrad\theta_{j,k}=0$
in $\dot\bigcup_{l\in\{1,\ldots,p\}}\Upsilon_{l}\setminus F_{l}$ for all $k\in\{1,2,3\}$.
Thus $\symGrad\theta_{j,k}$ can be continuously extended by zero 
to $\Theta_{j,k}\in\Cittom\cap\LtttSom$ with $\Theta_{j,k}=0$ 
in all the neighbourhoods $\Upsilon_{l}$ of all surfaces
$F_{l}$, $l\in\{1,\ldots,p\}$.

\begin{lemma}
\label{dimNTela:lem0}
Let Assumption \ref{ass:curvesandsurfaces} be satisfied. Then
$\Theta_{j,k}\in\ker(\CCtS,\om)$.
\end{lemma}

\begin{proof}
Let $\Phi\in\CicttSom$. 
As $\supp\Theta_{j,k}\subset\widetilde{\Upsilon}_{j}\setminus\Upsilon_{j}$
we can pick another cut-off function $\varphi\in\Cic(\om_{F})$ 
with $\varphi|_{\supp\Theta_{j,k}\cap\supp\Phi}=1$. Then
\begin{align*}
&\qquad\scp{\Theta_{j,k}}{\CCtS\Phi}_{\LtttSom}
=\scp{\Theta_{j,k}}{\CCtS\Phi}_{\LtttS(\supp\Theta_{j,k}\cap\supp\Phi)}\\
&=\bscp{\symGrad\theta_{j,k}}{\CCtS(\varphi\Phi)}_{\LtttS(\om_{F})}
=\bscp{\Grad\theta_{j,k}}{\CCtS(\varphi\Phi)}_{\LtttS(\om_{F})}\\
&=\Bscp{\Grad\theta_{j,k}}{\Curl\big(\Curl(\varphi\Phi)\big)^{\top}}_{\Lttt(\om_{F})}
=0
\end{align*}
as $\varphi\Phi,\CCtS(\varphi\Phi)\in\CicttS(\om_{F})$
by Lemma \ref{PZformulalem}.
\end{proof}

Next, we construct functionals similar to the previous sections. 
Here, however, due to the complex structure, 
we need six times as many (instead of four) as for the de Rham complex. 
For starters, note that for $l,j\in\{1,\dots,p\}$ and $k\in\{1,\dots,6\}$ 
and for the curves $\zeta_{x_{l,0},x_{l,1}}\subset\zeta_{l}$
with the chosen starting points $x_{l,0}\in\Upsilon_{l,0}$ 
and respective endpoints $x_{l,1}\in\Upsilon_{l,1}$
we can compute\footnote{Note that 
$\curl\widehat{r}_{k}=2e^{k}$
for $k\in\{1,2,3\}$, since, e.g.,
\begin{align*}
\curl\widehat{r}_{1}(x)
&=\curl\,(e^{1}\times x)
=\curl\,(x_{2}\,e^{1}\times e^{2}+x_{3}\,e^{1}\times e^{3})
=\curl\,(x_{2}e^{3}-x_{3}e^{2})\\
&=\grad\,(x_{2})\times e^{3}-\grad\,(x_{3})\times e^{2}
=e^{2}\times e^{3}-e^{3}\times e^{2}
=2e^{1}.
\end{align*}} 
by Lemma \ref{symGradHSlemma}
\begin{align*}
\R^{3}\ni a_{l}(\Theta_{j,k})
&:=\int_{\zeta_{l}}
(\Curl\Theta_{j,k})^{\top}\intd\lambda
=\int_{\zeta_{x_{l,0},x_{l,1}}}(\Curl\symGrad\theta_{j,k})^{\top}\intd\lambda\\
&\phantom{:}=\frac{1}{2}\curl\theta_{j,k}(x_{l,1})-\frac{1}{2}\curl\theta_{j,k}(x_{l,0})
=\frac{1}{2}\curl\theta_{j,k}(x_{l,1})\\
&\phantom{:}=\frac{1}{2}\delta_{l,j}\curl\widehat{r}_{k}(x_{l,1})
=\delta_{l,j}\begin{cases}
e^{k},&\text{if }k\in\{1,2,3\},\\
0,&\text{if }k\in\{4,5,6\},
\end{cases}
\intertext{and}
\R^{3}\ni b_{l}(\Theta_{j,k})
&:=\int_{\zeta_{l}}\Theta_{j,k}\intd\lambda
+\int_{\zeta_{l}}\spn\big((\Curl\Theta_{j,k})^{\top}(y)\intd\lambda_{y}\big)(x_{l,1}-y)\\
&\phantom{:}=\int_{\zeta_{x_{l,0},x_{l,1}}}\symGrad\theta_{j,k}\intd\lambda\\
&\qquad\qquad+\int_{\zeta_{x_{l,0},x_{l,1}}}\spn\big((\Curl\symGrad\theta_{j,k})^{\top}(y)\intd\lambda_{y}\big)(x_{l,1}-y)\\
&\phantom{:}=\int_{\zeta_{x_{l,0},x_{l,1}}}\Big(\symGrad\theta_{j,k}(y)\\
&\qquad\qquad+\int_{\zeta_{x_{l,0},y}}
\spn\big((\Curl\symGrad\theta_{j,k})^{\top}\intd\lambda\big)\Big)\intd\lambda_{y}\\
&\phantom{:}=\theta_{j,k}(x_{l,1})-\theta_{j,k}(x_{l,0})-\frac{1}{2}\curl\theta_{j,k}(x_{l,0})\times(x_{l,1}-x_{l,0})
=\theta_{j,k}(x_{l,1})\\
&\phantom{:}=\delta_{l,j}\widehat{r}_{k}(x_{l,1})
=\delta_{l,j}\begin{cases}
e^{k}\times x_{l,1},&\text{if }k\in\{1,2,3\},\\
e^{k-3},&\text{if }k\in\{4,5,6\}.
\end{cases}
\end{align*}
Thus, for $l\in\{1,\dots,p\}$ and $\ell\in\{1,\dots,6\}$ we have functionals $\beta_{l,\ell}$ given by
$$\beta_{l,\ell}(\Theta_{j,k}):=
\begin{cases}
\bscp{a_{l}(\Theta_{j,k})}{e^{\ell}},&\text{if }\ell\in\{1,2,3\},\\
\bscp{b_{l}(\Theta_{j,k})}{e^{\ell-3}},&\text{if }\ell\in\{4,5,6\},
\end{cases}\qquad
j\in\{1,\dots,p\},\quad k\in\{1,\dots,6\}.$$
Then for $l,j\in\{1,\dots,p\}$ and for $\ell\in\{1,2,3\}$
\begin{align*}
\beta_{l,\ell}(\Theta_{j,k})
&=\bscp{a_{l}(\Theta_{j,k})}{e^{\ell}}
=\delta_{l,j}
\begin{cases}
\scp{e^{k}}{e^{\ell}}=\delta_{\ell,k},&\text{if }k\in\{1,2,3\},\\
\scp{0}{e^{\ell}}=0,&\text{if }k\in\{4,5,6\},
\end{cases}
\intertext{i.e.,}
\beta_{l,\ell}(\Theta_{j,k})
&=\delta_{l,j}\delta_{\ell,k},\qquad
k\in\{1,\dots,6\},
\intertext{and for $\ell\in\{4,5,6\}$}
\beta_{l,\ell}(\Theta_{j,k})
&=\bscp{b_{l}(\Theta_{j,k})}{e^{\ell-3}}
=\delta_{l,j}\begin{cases}
\scp{e^{k}\times x_{l,1}}{e^{\ell-3}}=\scp{e^{\ell-3}\times e^{k}}{x_{l,1}},&\text{if }k\in\{1,2,3\},\\
\scp{e^{k-3}}{e^{\ell-3}}=\delta_{\ell,k},&\text{if }k\in\{4,5,6\},
\end{cases}
\intertext{i.e.,}
\beta_{l,\ell}(\Theta_{j,k})
&=\delta_{l,j}\delta_{\ell,k}
+\delta_{l,j}(\delta_{1,k}+\delta_{2,k}+\delta_{3,k})(x_{l,1})_{\widehat{\ell-3,k}},\qquad
k\in\{1,\dots,6\},
\end{align*}
where we define
\begin{align*}(x_{l,1})_{\widehat{\ell-3,k}}
&:=\scp{e^{\ell-3}\times e^{k}}{x_{l,1}}
=\sum_{i=1}^3\scp{e^{\ell-3}\times e^{k}}{e^{i}}(x_{l,1})_{i}\\
&\phantom{:}=
\begin{cases}(x_{l,1})_{i},&\exists i\in\{1,2,3\}:(\ell-3,k,i)\text{ even permutation of }(1,2,3),\\
-(x_{l,1})_{i},&\exists i\in\{1,2,3\}:(\ell-3,k,i)\text{ odd permutation of }(1,2,3),\\
0,&\forall i\in\{1,2,3\}:(\ell-3,k,i)\text{ no permutation of }(1,2,3).
\end{cases}
\end{align*}
In particular, $(x_{l,1})_{\widehat{\ell-3,k}}=0$
if $\ell-3=k$ or $\ell\in\{1,2,3\}$ or $k\in\{4,5,6\}$.
Therefore, we have for $l,j\in\{1,\dots,p\}$ and $k,\ell\in\{1,\dots,6\}$
\begin{align}
\label{pathint1-ela}
\begin{aligned}
\beta_{l,\ell}(\Theta_{j,k})
&=\delta_{l,j}\delta_{\ell,k}
+\delta_{l,j}(x_{l,1})_{\widehat{\ell-3,k}}\\
&=\delta_{l,j}\delta_{\ell,k}
+\delta_{l,j}(\delta_{\ell,4}+\delta_{\ell,5}+\delta_{\ell,6})
(\delta_{1,k}+\delta_{2,k}+\delta_{3,k})
(1-\delta_{\ell-3,k})
(x_{l,1})_{\widehat{\ell-3,k}}.
\end{aligned}
\end{align}

Let Assumption \ref{ass:stronglip} be satisfied.
For the elasticity complex, similar to \eqref{deco1}, \eqref{deco2},
see also \eqref{app:deco2-derham}, \eqref{app:deco2-bih1}, \eqref{app:deco2-bih2},
we have the orthogonal decompositions 
\begin{align}
\label{app:deco2-ela}
\begin{aligned}
\LtttSom&=\ran(\symGrad,\om)\oplus_{\LtttSom}\ker(\rDS,\om),\\
\ker(\CCtS,\om)&=\ran(\symGrad,\om)\oplus_{\LtttSom}\harmelaNom.
\end{aligned}
\end{align}

\begin{remark}
\label{rem:Helmholtz-elaN}
\cite[Lemma 3.2]{PZ2020b} yields $\dom(\symGrad,\om)=H^{1,3}(\om)$. 
Thus, as before, a combination of Rellich's selection theorem and a contradiction argument 
implies the Poincar\'e type estimate
$$\exists\,c>0\quad\forall\,\phi\in H^{1,3}(\om)\cap\RM_{\pw}^{\bot_{\Lttom}}\qquad
\norm{\phi}_{\Lttom}\leq c\norm{\symGrad\phi}_{\Ltttom}.$$
Thus, the range in \eqref{app:deco2-ela} is closed.
\end{remark}

Let $\pi:\LtttSom\to \ker(\rDS,\om)$ denote the orthogonal projector 
along $\ran(\symGrad,\om)$ according to \eqref{app:deco2-ela}.
We infer $\pi\big(\ker(\CCtS,\om)\big)=\harmelaNom$.
Thus, using Lemma \ref{dimNTela:lem0}, 
for all $j\in\{1,\ldots,p\}$ and $k\in\{1,\ldots,6\}$ 
we find $\psi_{j,k}\in H^{1,3}(\om)$ such that 
$$\harmelaNom\ni\pi\Theta_{j,k}=\Theta_{j,k}-\symGrad\psi_{j,k},\quad
(\Theta_{j,k}-\symGrad\psi_{j,k})\big|_{\om_{F}}
=\symGrad(\theta_{j,k}-\psi_{j,k}).$$
Next, Lemma \ref{lem:harmreg} implies $\harmelaNom\subset\Cittom$. 
Thus, $\pi\Theta_{j,k},\Theta_{j,k}\in\Cittom$ implies $\symGrad\psi_{j,k}\in\Cittom$ 
and hence $\psi_{j,k}\in\Citom$.
Therefore, all path integrals over the closed curves $\zeta_{l}$ are well-defined. 
Furthermore, we observe by Lemma \ref{symGradHSlemma}
\begin{align*}
a_{l}(\symGrad\psi_{j,k})
&=\int_{\zeta_{l}}
(\Curl\symGrad\psi_{j,k})^{\top}\intd\lambda\\
&=\frac{1}{2}\big(\curl\psi_{j,k}(x_{l,1})-\curl\psi_{j,k}(x_{l,1})\big)
=0,
\intertext{and}
b_{l}(\symGrad\psi_{j,k})
&=\int_{\zeta_{l}}\symGrad\psi_{j,k}\intd\lambda\\
&\qquad\qquad+\int_{\zeta_{l}}\spn\big((\Curl\symGrad\psi_{j,k})^{\top}(y)\intd\lambda_{y}\big)(x_{l,1}-y)\\
&=\int_{\zeta_{x_{l,1},x_{l,1}}}\Big(\symGrad\psi_{j,k}(y)\\
&\qquad\qquad+\int_{\zeta_{x_{l,1},y}}
\spn\big((\Curl\symGrad\psi_{j,k})^{\top}\intd\lambda\big)\Big)\intd\lambda_{y}\\
&=\psi_{j,k}(x_{l,1})
-\psi_{j,k}(x_{l,1})-\frac{1}{2}\curl\psi_{j,k}(x_{l,1})\times(x_{l,1}-x_{l,1})
=0.
\end{align*}
Hence, $\beta_{l,\ell}(\symGrad\psi_{j,k})=0$
and by \eqref{pathint1-ela}
\begin{align}
\label{pathint2-ela}
\begin{aligned}
\beta_{l,\ell}(\pi\Theta_{j,k})
&=\beta_{l,\ell}(\Theta_{j,k})
-\beta_{l,\ell}(\symGrad\psi_{j,k})
=\beta_{l,\ell}(\Theta_{j,k})
=\delta_{l,j}\delta_{\ell,k}
+\delta_{l,j}(x_{l,1})_{\widehat{\ell-3,k}}\\
&=\delta_{l,j}\delta_{\ell,k}
+\delta_{l,j}(\delta_{\ell,4}+\delta_{\ell,5}+\delta_{\ell,6})
(\delta_{1,k}+\delta_{2,k}+\delta_{3,k})
(1-\delta_{\ell-3,k})
(x_{l,1})_{\widehat{\ell-3,k}}
\end{aligned}
\end{align}
for all $l,j\in\{1,\dots,p\}$ and all $\ell,k\in\{1,\dots,6\}$. 
We shall show that 
\begin{align}
\label{basis:elaN}
\calB^{\ela}_{N}:=\big\{\pi\Theta_{j,k}:j\in\{1,\dots,p\},\,k\in\{1,\dots,6\}\big\}
\subset\harmelaNom
\end{align}
defines a basis of $\harmelaNom$.

The tensor fields $\Theta_{j,k}$ being constructed explicitly, 
we provide a way of finding $\psi_{j,k}$ by means of solutions of PDEs.

\begin{remark}[Characterisation by PDEs]
\label{rem:charPDEela}
Let $j\in\{1,\dots,p\}$ and $k\in\{1,\dots,6\}$. 
Then $\psi_{j,k}\in H^{1,3}(\om)\cap\RM_{\pw}^{\bot_{\Lttom}}$ 
can be found by solving the following elasticity PDE 
formulated in the standard variational formulation
$$\forall\,\phi\in H^{1,3}(\om)\qquad
\scp{\symGrad\psi_{j,k}}{\symGrad\phi}_{\Ltttom}
=\scp{\Theta_{j,k}}{\symGrad\phi}_{\Ltttom},$$
i.e., $\psi_{j,k}=\Delta_{\Sb,N}^{-1}\big(\DS\Theta_{j,k},\Theta_{j,k}\nu|_{\ga}\big)$, 
where $\Delta_{\Sb,N}\subset\DS\symGrad$ is the `symmetric' Laplacian 
with inhomogeneous Neumann boundary conditions 
restricted to a subset of $H^{1,3}(\om)\cap\RM_{\pw}^{\bot_{\Lttom}}$. Therefore, 
$$\pi\Theta_{j,k}
=\Theta_{j,k}-\symGrad\psi_{j,k}
=\Theta_{j,k}-\symGrad\Delta_{\Sb,N}^{-1}\big(\DS\Theta_{j,k},\Theta_{j,k}\nu|_{\ga}\big).$$
In classical terms, $\psi_{j,k}$
solves the Neumann elasticity problem
\begin{align}
\label{NeuLap1-ela}
\begin{aligned}
-\Delta_{\Sb}\psi_{j,k}&=-\DS\Theta_{j,k}
&&\text{in }\om,\\
(\Grad\psi_{j,k})\nu&=\Theta_{j,k}\nu
&&\text{on }\ga,\\
\int_{\om_{l}}(\psi_{j,k})_{\ell}&=0
&&\text{for }l\in\{1,\dots,n\},\quad\ell\in\{1,2,3\},\\
\int_{\om_{l}}\big(x\times\psi_{j,k}(x)\big)_{\ell}\intd\lambda_{x}&=0
&&\text{for }l\in\{1,\dots,n\},\quad\ell\in\{1,2,3\},
\end{aligned}
\end{align}
which is uniquely solvable.
\end{remark}

\begin{lemma}
\label{dimNTela:lem1}
Let Assumptions \ref{ass:stronglip} and \ref{ass:curvesandsurfaces} be satisfied. 
Then $\harmelaNom=\lin\calB^{\ela}_{N}$.
\end{lemma}

\begin{proof}
Let $H\in\harmelaNom=\ker(\rDS,\om)\cap\ker(\CCtS,\om)\subset\Citt_{\Sb}(\om)$, 
see Lemma \ref{lem:harmreg}.
With the above introduced functionals $a_{l}$ and $b_{l}$, $l\in\{1,\dots,p\}$, we recall
\begin{align*}
\R^{3}\ni a_{l}(H)
&=\int_{\zeta_{l}}
(\Curl H)^{\top}\intd\lambda,\\
\R^{3}\ni b_{l}(H)
&=\int_{\zeta_{l}}H\intd\lambda
+\int_{\zeta_{l}}\spn\big((\Curl H)^{\top}(y)\intd\lambda_{y}\big)(x_{l,1}-y),
\end{align*}
and define for $l\in\{1,\dots,p\}$ the numbers
\begin{align*}
\gamma_{l,\ell}
:=\gamma_{l,\ell}(H)
&:=\bscp{a_{l}(H)}{e^{\ell}}
=\beta_{l,\ell}(H),
&
\text{for }\ell&\in\{1,2,3\},\\
\gamma_{l,\ell}
:=\gamma_{l,\ell}(H)
&:=\Bscp{b_{l}(H)-\sum_{k=1}^{3}\beta_{l,k}(H)e^{k}\times x_{l,1}}{e^{\ell-3}}\\
&\phantom{:}=\beta_{l,\ell}(H)
-\sum_{k=1}^{3}\beta_{l,k}(H)(x_{l,1})_{\widehat{\ell-3,k}},
&
\text{for }\ell&\in\{4,5,6\},
\end{align*}
where we recall
$(x_{l,1})_{\widehat{\ell-3,k}}=(\delta_{\ell,4}+\delta_{\ell,5}+\delta_{\ell,6})
(\delta_{1,k}+\delta_{2,k}+\delta_{3,k})(1-\delta_{\ell-3,k})(x_{l,1})_{\widehat{\ell-3,k}}$
by definition, see \eqref{pathint1-ela}.
We shall show that
$$\harmelaNom\ni\widehat{H}
:=H-\sum_{j=1}^{p}\sum_{k=1}^{6}\gamma_{j,k}\pi\Theta_{j,k}=0\quad\text{in }\om.$$
Similar to the proof of Lemma \ref{dimNFrhm:lem1} (or \ref{dimNTbih1:lem1}, \ref{dimNTbih2:lem1})
the aim is to prove the existence of $v\in H^{1,3}(\om)$
such that $\symGrad v=\widehat{H}$, since then 
$$\norm{\widehat{H}}_{\LtttSom}^{2}
=\scp{\symGrad v}{\widehat{H}}_{\LtttSom}=0.$$ 
For the construction of $v$, 
we apply Corollary \ref{cor:ela} and Remark \ref{curlformula1-ela} to $S=\widehat{H}$. 
In order to show condition Corollary \ref{cor:ela} (a), (i), 
we observe for $l\in\{1,\dots,p\}$ and $\ell\in\{1,2,3\}$ by \eqref{pathint2-ela} 
\begin{align*}
\Big(\int_{\zeta_{l}}(\Curl\widehat{H})^{\top}\intd\lambda\Big)_{\ell}
&=\big(a_{l}(\widehat{H})\big)_{\ell}=\beta_{l,\ell}(\widehat{H})
=\beta_{l,\ell}(H)
-\sum_{j=1}^{p}\sum_{k=1}^{6}\gamma_{j,k}
\beta_{l,\ell}(\pi\Theta_{j,k})\\
&=\gamma_{l,\ell}
-\sum_{j=1}^{p}\sum_{k=1}^{6}\gamma_{j,k}
\delta_{l,j}\delta_{\ell,k}
=0.
\end{align*}
Note that $\Curl(\Curl\widehat{H})^{\top}=\CCtS\widehat{H}=0$,
i.e., $(\Curl\widehat{H})^{\top}\in\ker(\Curl)\cap\Cittom$.
Thus, by Assumption \ref{ass:curvesandsurfaces} (A.1)
for any closed piecewise $C^1$-curve $\zeta$ in $\om$
\begin{align}
\label{pathint3-ela}
\int_{\zeta}(\Curl\widehat{H})^{\top}\intd\lambda=0.
\end{align}
Hence, by Corollary \ref{cor:ela} (a), (ii), there exists $w\in\Citom$ such that
$$\Grad w=(\Curl \widehat{H})^{\top}.$$
We define $T:=\widehat{H}+\spn w$.
Let $l\in\{1,\ldots,p\}$. 
Then $\zeta_{x_{l,0},x_{l,1}}\subset\zeta_{l}\subset\om_{0}$
for some $\om_{0}\in\cc(\om)$.
With $c:=w(x_{l,1})\in\R^{3}$ we compute for all $x\in\zeta_{l}$
$$w(x)
=w(x)-w(x_{l,1})+c
=\int_{\zeta_{x_{l,1},x}}\Grad w\intd\lambda+c
=\int_{\zeta_{x_{l,1},x}}(\Curl\widehat{H})^{\top}\intd\lambda+c,$$
and
$$\int_{\zeta_{l}}(\spn c)\intd\lambda
=(\spn c)\int_{\zeta_{l}}\id\intd\lambda
=(\spn c)\int_{\zeta_{l}}\Grad x\intd\lambda_{x}
=0.$$
Again, we consider the curve $\zeta_{l}$ 
as the closed curve $\zeta_{x_{l,1},x_{l,1}}$ with circulation $1$ along $\zeta_{l}$.
By Lemma \ref{symGradHSlemma} and by the definition of $b_{l}$ we have for $l\in\{1,\dots,p\}$
\begin{align*}
\int_{\zeta_{l}}T\intd\lambda
&=\int_{\zeta_{l}}\widehat{H}\intd\lambda
+\int_{\zeta_{l}}(\spn w)\intd\lambda\\
&=\int_{\zeta_{l}}\widehat{H}\intd\lambda
+\int_{\zeta_{x_{l,1},x_{l,1}}}
\spn\Big(\int_{\zeta_{x_{l,1},y}}(\Curl\widehat{H})^{\top}\intd\lambda\Big)\intd\lambda_{y}\\
&=\int_{\zeta_{l}}\widehat{H}\intd\lambda
+\int_{\zeta_{l}}\spn\big((\Curl\widehat{H})^{\top}(y)\intd\lambda_{y}\big)(x_{l,1}-y)
=b_{l}(\widehat{H}).
\end{align*}
Hence, for $\ell\in\{4,5,6\}$ we get by \eqref{pathint2-ela}
\begin{align*}
\Big(\int_{\zeta_{l}}T\intd\lambda\Big)_{\ell-3}
&=\bscp{\int_{\zeta_{l}}T\intd\lambda}{e^{\ell-3}}
=\scp{b_{l}(\widehat{H})}{e^{\ell-3}}
=\beta_{l,\ell}(\widehat{H})\\
&=\beta_{l,\ell}(H)
-\sum_{j=1}^{p}\sum_{k=1}^{6}\gamma_{j,k}
\beta_{l,\ell}(\pi\Theta_{j,k})\\
&=\beta_{l,\ell}(H)
-\sum_{j=1}^{p}\sum_{k=1}^{6}\gamma_{j,k}
(\delta_{l,j}\delta_{\ell,k}
+\delta_{l,j}(x_{l,1})_{\widehat{\ell-3,k}})\\
&=\beta_{l,\ell}(H)
-\gamma_{l,\ell}
-\sum_{k=1}^{3}\gamma_{l,k}(x_{l,1})_{\widehat{\ell-3,k}}\\
&=\beta_{l,\ell}(H)
-\gamma_{l,\ell}
-\sum_{k=1}^{3}\beta_{l,k}(H)(x_{l,1})_{\widehat{\ell-3,k}}
=0.
\end{align*}
Therefore, $\int_{\zeta_{l}}T\intd\lambda=0$
for all $l\in\{1,\dots,p\}$.
Note that $T\in\ker(\Curl)\cap\Cittom$ by Remark \ref{curlformula1-ela}
as $S=\widehat{H}\in\Citt_{\Sb}(\om)$ and $T=S+\spn w$ with $\Grad w=(\Curl\widehat{H})^{\top}$.
Thus, by Assumption \ref{ass:curvesandsurfaces} (A.1)
for any closed piecewise $C^1$-curve $\zeta$ in $\om$
\begin{align}
\label{pathint4-ela}
\int_{\zeta}T\intd\lambda=0.
\end{align}
Hence, by the symmetry of $\widehat{H}$ 
and Corollary \ref{cor:ela} (b), (c), there exists $v\in\Citom$ 
such that $\Grad v=T$ as well as $\symGrad v=\widehat{H}$.
Similar to the end of the proof of Lemma \ref{dimNFrhm:lem1},
elliptic regularity and, e.g., \cite[Theorem 2.6 (1)]{L1986} or \cite[Theorem 3.2 (2)]{A1965}
show that $v\in\Citom$ with $\symGrad v\in\LtttSom$ implies
$v\in H^{1,3}(\om)$, completing the proof. 
(Let us note that $v\in H^{1,3}(\om)$
implies also $T\in\Ltttom$ and hence $w\in\Lttom$.)
\end{proof}

\begin{lemma}
\label{dimNTela:lem2}
Let Assumption \ref{ass:stronglip} and Assumption \ref{ass:curvesandsurfaces} be satisfied. Then
$\calB^{\ela}_{N}$ is linearly independent.
\end{lemma}

\begin{proof}
Let $\gamma_{j,k}\in\R$, $j\in\{1,\ldots,p\}$, $k\in\{1,\ldots,6\}$, 
be such that $\displaystyle\sum_{j=1}^{p}\sum_{k=1}^{6}\gamma_{j,k}\pi\Theta_{j,k}=0$.
Then \eqref{pathint2-ela} implies for $l\in\{1,\dots,p\}$
\begin{align*}
0&=\sum_{j=1}^{p}\sum_{k=1}^{6}\gamma_{j,k}
\beta_{l,\ell}(\pi\Theta_{j,k})
=\gamma_{l,\ell},
&\ell&\in\{1,2,3\},\\
0&=\sum_{j=1}^{p}\sum_{k=1}^{6}\gamma_{j,k}
\beta_{l,\ell}(\pi\Theta_{j,k})
=\gamma_{l,\ell}+\sum_{k=1}^{3}\gamma_{l,k}(x_{l,1})_{\widehat{\ell-3,k}}=\gamma_{l,\ell},
&\ell&\in\{4,5,6\},
\end{align*}
finishing the proof.
\end{proof}

\begin{theorem}
\label{dimNTela}
Let Assumptions \ref{ass:stronglip} and \ref{ass:curvesandsurfaces} be satisfied.
Then $\dim\harmelaNom=6p$ and a basis of $\harmelaNom$ is given by \eqref{basis:elaN}.
\end{theorem}

\begin{proof}
Use Lemma \ref{dimNTela:lem1} and Lemma \ref{dimNTela:lem2}.
\end{proof}

\section{Conclusion}
\label{sec:con}

The index theorems presented rest on the abstract construction principle provided in \cite{BL1992}
and the results on the newly found \pz complex from \cite{PZ2016,PZ2020a}
and the elasticity complex from \cite{PZ2020b,PZ2020c}. 
With this insight it is possible to construct basis fields
for the generalised harmonic Dirichlet and Neumann tensor fields. 
The number of basis fields of the considered cohomology groups 
provide additional topological invariants. 
The construction of the generalised Dirichlet fields being somewhat similar to the de Rham complex, 
the same for the generalised Neumann fields requires some more machinery particularly relying 
on the introduction of Poincar\'e maps defining the functionals.

Furthermore, we have outlined numerical strategies 
to compute the generalised Neumann and Dirichlet fields in practice. 
In passing we have also provided a set of Friedrichs--Poincar\'e type estimates 
and included variable coefficients relevant for numerical studies.

All these constructions heavily rely on the choice of boundary conditions 
and we emphasise that the considered mixed order operators \emph{cannot} be viewed 
as leading order plus relatively compact perturbation, when it comes to computation of the Fredholm index. 
In particular, techniques from pseudo-differential calculus successfully applied 
to obtain index formulas for operators defined on non-compact manifolds or compact manifolds 
without boundary, see e.g.~\cite{H1979,I1977}, are likely to be very difficult 
to be applicable in the present situation. It would be interesting to see, 
whether the operators considered above defined on an unbounded domain enjoy
similar index formulas (maybe a comparable Witten index of some sort) 
even though the operator itself might not be of Fredholm type anymore.

%%%%%%%%%%%%%%%%%%%%%%%%%%%%%%%%%%%%%%%%%%%%%%%%%%%%%%%%%%%%%%%%%%%%%%%%%%%%%%%%%%%%%%%%%%%%

\bibliographystyle{abbrv}

%%%%%%%%%%%%%%%%%%%%%%%%%%%%%%%%%%%%%%%%%%%%%%%%%%%%%%%%%%%%%%%%%%%%%%%%%%%%%%%%%%%%%%%%%%%%

\vspace*{5mm}
\hrule
\vspace*{3mm}

%%%%%%%%%%%%%%%%%%%%%%%%%%%%%%%%%%%%%%%%%%%%%%%%%%%%%%%%%%%%%%%%%%%%%%%%%%%%%%%%%%%%%%%%%%%%

\end{document}